\documentclass[11pt, reqno]{amsart}
\usepackage{amssymb}
\usepackage{eucal}
\usepackage{amsmath}
\usepackage{amscd}
\usepackage[dvips]{color}
\usepackage{multicol}
\usepackage[all]{xy}           
\usepackage{graphicx}
\usepackage{color}
\usepackage{colordvi}
\usepackage{xspace}
\usepackage{tikz}
\usepackage{ulem}

\usepackage{ifpdf}
\ifpdf
 \usepackage[colorlinks,final,
 hyperindex]{hyperref}
\else
 \usepackage[colorlinks,final,
 hyperindex,hypertex]{hyperref}
\fi

\begin{document}
\newcommand {\emptycomment}[1]{} 

\newcommand{\nc}{\newcommand}
\newcommand{\delete}[1]{}
\nc{\mfootnote}[1]{\footnote{#1}} 
\nc{\todo}[1]{\tred{To do:} #1}

\delete{
\nc{\mlabel}[1]{\label{#1}}  
\nc{\mcite}[1]{\cite{#1}}  
\nc{\mref}[1]{\ref{#1}}  
\nc{\meqref}[1]{\eqref{#1}} 
\nc{\bibitem}[1]{\bibitem{#1}} 
}

\nc{\mlabel}[1]{\label{#1}  
{\hfill \hspace{1cm}{\bf{{\ }\hfill(#1)}}}}
\nc{\mcite}[1]{\cite{#1}{{\bf{{\ }(#1)}}}}  
\nc{\mref}[1]{\ref{#1}{{\bf{{\ }(#1)}}}}  
\nc{\meqref}[1]{\eqref{#1}{{\bf{{\ }(#1)}}}} 
\nc{\mbibitem}[1]{\bibitem[\bf #1]{#1}} 


\newtheorem{thm}{Theorem}[section]
\newtheorem{lem}[thm]{Lemma}
\newtheorem{cor}[thm]{Corollary}
\newtheorem{pro}[thm]{Proposition}
\newtheorem{conj}[thm]{Conjecture}
\theoremstyle{definition}
\newtheorem{defi}[thm]{Definition}
\newtheorem{ex}[thm]{Example}
\newtheorem{rmk}[thm]{Remark}
\newtheorem{pdef}[thm]{Proposition-Definition}
\newtheorem{condition}[thm]{Condition}



\title[Relative Poisson bialgebras and Frobenius Jacobi algebras]{Relative Poisson bialgebras and Frobenius Jacobi algebras}

\author{Guilai Liu}
\address{Chern Institute of Mathematics \& LPMC, Nankai University, Tianjin 300071, China}
\email{1120190007@mail.nankai.edu.cn}

\author{Chengming Bai}
\address{Chern Institute of Mathematics \& LPMC, Nankai University, Tianjin 300071, China }
\email{baicm@nankai.edu.cn}


\begin{abstract}
Jacobi algebras, as the algebraic counterparts of Jacobi
manifolds, are exactly the unital relative Poisson algebras. The
direct approach of constructing Frobenius Jacobi algebras in terms
of Manin triples is not available due to the existence of the
units, and hence alternatively we replace it by studying Manin
triples of relative Poisson algebras. Such structures are
equivalent to certain bialgebra structures, namely, relative
Poisson bialgebras. The study of coboundary cases leads to the
introduction of the relative Poisson Yang-Baxter equation (RPYBE).
Antisymmetric solutions of the RPYBE give coboundary relative
Poisson bialgebras. The notions of $\mathcal O$-operators of
relative Poisson algebras and  relative pre-Poisson algebras are
introduced to give antisymmetric solutions of the RPYBE. A direct
application is that relative Poisson bialgebras can be used to
construct Frobenius Jacobi algebras, and in particular, there is a
construction of Frobenius Jacobi algebras from relative
pre-Poisson algebras.
\end{abstract}

\subjclass[2010]{
17B63,  
17A36,  
18M70,  
17D25,  
17A40,  
37J39,   
53D17  
}

\keywords{Jacobi algebra; relative Poisson algebra;  Frobenius
Jacobi algebra; bialgebra; classical Yang-Baxter equation;
relative pre-Poisson algebra}

\maketitle


\tableofcontents

\allowdisplaybreaks

\section{Introduction}

The aim of this paper is to give a bialgebra theory for
relative Poisson algebras, in which one of the motivations is to
construct Frobenius Jacobi algebras.

\subsection{Generalizations of Poisson algebras}\

Poisson algebras arose in the study of Poisson geometry
(\cite{BV1,Li77,Wei77}) and are closely related to a lot of
topics in mathematics and physics.
Recall that a \textbf{Poisson algebra} is a vector space $A$ equipped
with two bilinear operations $\cdot, [-,-]:A\otimes A\rightarrow A$ such
that $(A,\cdot)$ is a commutative associative algebra, $(A,[-,-])$
is a Lie algebra and they are compatible in the sense of the Leibniz
rule:
\begin{equation}\label{eq:PA}
[z,x\cdot y]=[z,x]\cdot y+x\cdot[z,y],\;\;\forall x,y,z\in A.
\end{equation}

There are  generalizations of Poisson algebras such as
noncommutative Poisson algebras in which the commutativity of the
associative algebras is cancelled (\cite{Xu}) or variations of
algebraic structures by changing the compatible condition
(\ref{eq:PA}) (cf. \cite{BBGW,Dot,MZ}).  Among the latter, Jacobi
algebras are abstract algebraic counterparts of Jacobi manifolds
(\cite{Kir, Lic}), which are generalizations of symplectic or more
generally Poisson manifolds. Recall that a Jacobi manifold is a
smooth manifold endowed with a bivector $\Lambda$ and a vector
field $E$ satisfying some compatible conditions, or
equivalently, it is a smooth manifold $M$ such that the
commutative associative algebra $A:=C^{\infty}(M)$ of real smooth
functions on $M$ is endowed with a Lie bracket $[-,-]$ satisfying
\begin{equation}\label{eq:JA}
[z,x\cdot y]=[z,x]\cdot y+x\cdot[z,y]+x\cdot y\cdot[1_{A},z],
\;\;\forall x,y,z\in A,
\end{equation}
where $1_A$ is the unit of $A$.
Correspondingly, a
  {\bf Jacobi algebra} is a triple $(A,\cdot,$ $[-,-])$, such that
$(A,\cdot)$ is a unital commutative associative algebra,
$(A,[-,-])$ is a Lie algebra, and they satisfy Eq.~ (\ref{eq:JA}).

Taking $z=1_{A}$ in Eq.~(\ref{eq:JA}), one gets
\begin{equation}\label{eq:JA2}
[1_{A},x\cdot y]=[1_{A},x]\cdot y+x\cdot[1_{A},y],\;\;\forall x,y\in
A.
\end{equation}
Hence $\mathrm{ad}(1_{A})$ is a derivation of both $(A,\cdot)$ and
$(A,[-,-])$, where ${\rm ad}(1_A)(x)=[1_A,x]$ for all $x\in A$.
Therefore there is the following natural generalization of a Jacobi
algebra.

\begin{defi}\label{de:GPA}
A \textbf{relative Poisson algebra} is a quadruple $(A,\cdot,
[-,-],P)$, where $(A,\cdot)$ is a commutative associative algebra,
$(A,[-,-])$ is a Lie algebra, $P$ is a derivation of both
$(A,\cdot)$ and $(A,[-,-])$, that is, $P$ satisfies the following
conditions:
\begin{equation}\label{eq:DCAA}
P(x\cdot y)=P(x)\cdot y+x\cdot P(y),
\end{equation}
\begin{equation}\label{eq:DLA}
P([x,y])=[P(x),y]+[x,P(y)],
\end{equation}
for all $x,y\in A$, and the relative Leibniz rule is satisfied, that is,
\begin{equation}\label{eq:GLR}
[z,x\cdot y]=[z,x]\cdot y+x\cdot[z,y]+x\cdot y\cdot P(z),\;\;\forall
x,y,z\in A.
\end{equation}
We also denote it simply by $(A,P)$.
\end{defi}

A relative Poisson algebra $(A,\cdot,[-,-],P)$ is called unital
if $(A,\cdot)$ is a unital commutative associative algebra.
There is a one-to-one correspondence between Jacobi algebras and unital relative Poisson algebras: a Jacobi algebra $(A,\cdot,[-,-])$ is a unital
    relative Poisson algebra $(A,\cdot,[-,-],\mathrm{ad}(1_{A}))$, whereas a unital relative Poisson algebra
    $(A,\cdot,[-,-],P)$ is  a Jacobi algebra since one gets
    $P(x)=[1_{A},x]$ for all $x\in A$ by Eq.~(\ref{eq:GLR}).
Moreover, one can derive unital relative Poisson algebras (and
thus Jacobi algebras) from not necessarily unital ones (see Lemma
\ref{pro:standard Jacobi algebra}). That is, relative Poisson
algebras not only generalize Jacobi algebras in a natural way, but
also provide a good supplement for the latter. It is also natural
to consider the possible geometry related to relative Poisson
algebras, which is still little known and will be an interesting
topic in the future study.

\delete{ Obviously a Jacobi algebra $(A,\cdot,[-,-])$ is indeed a
unital relative Poisson algebra
$(A,\cdot,[-,-],\mathrm{ad}(1_{A}))$. Conversely, any unital
relative Poisson algebra $(A,\cdot,[-,-],P)$ is a Jacobi algebra
since one gets $P(x)=[1_{A},x]$ for all $x\in A$ by
Eq.~(\ref{eq:GLR}).}

On the other hand, relative Poisson algebras are also motivated by
different approaches. In fact, they appeared first in a graded
form under the name of generalized Poisson superalgebras
(\cite{Can}), which arose in the classification of a class of
simple Jordan superalgebras, named Kantor series (\cite{Kan,
Kin}). See \cite{Can2010, Kay} for more details and the related
topics. Note that there are other different algebraic structures
also called generalized Poisson algebras (\cite{Fum, She}). Hence
we adopt the present notion of relative Poisson algebras.



Note that
  any Poisson algebra can be viewed as a
relative Poisson algebra with the derivation $P=0$. On the other
hand, a relative Poisson algebra $(A,\cdot, [-,-], P)$ in which
the commutative associative algebra $(A,\cdot)$ is trivial, that
is, $x\cdot y=0$ for all $x,y\in A$, is exactly the structure
consisting of a Lie algebra $(A, [-,-])$ and a derivation $P$ of
it. A less trivial example of a relative Poisson algebra is given
as follows.

\begin{ex}\label{ex:GPA}
Let $(A,\cdot)$ be a commutative associative algebra with a
derivation $P$. Define
\begin{equation}\label{eq:EXGPA}
[x,y]=x\cdot P(y)-P(x)\cdot y,\;\;\forall x,y\in A.
\end{equation}
Then $(A,[-,-])$ is a Lie algebra and $(A,\cdot, [-,-], P)$ is a
relative Poisson algebra.
\end{ex}

\subsection{Frobenius Jacobi algebras and relative Poisson bialgebras}\

It is important to study Frobenius algebras due to their close
relationships with many areas such as topology, algebraic geometry,
category theory, Hochschild cohomology and graph theory
(\cite{Bai2010,Iov,Koc,Lam}). An associative algebra $(A,\cdot)$
is called  Frobenius  (\cite{Bra, Fro}) if there exists a
nondegenerate bilinear form $\mathcal{B}$ which is
 invariant in the sense that
\begin{equation}\label{eq:Fro asso}
    \mathcal{B}(x\cdot y,z)=\mathcal{B}(x,y\cdot z),\;\;\forall x,y,z\in A.
\end{equation}
Note that if $(A,\cdot)$ is unital and commutative, then 
Eq.~(\ref{eq:Fro asso}) requires $\mathcal{B}$ to be symmetric.
Such bilinear forms on associative algebras reflect a certain
natural symmetry, characterizing the so-called Frobenius property.
In the context of Lie algebras, that is, a Lie algebra
$(\mathfrak{g},[-,-])$ equipped with a nondegenerate bilinear form
$\mathcal{B}$ which is invariant in the sense that
\begin{equation}\label{eq:Fro Lie}
    \mathcal{B}([x,y],z)=\mathcal{B}(x,[y,z]),\;\;\forall x,y,z\in
    \mathfrak{g},
\end{equation}
is called self-dual (\cite{JMF}). When $\mathcal{B}$ is symmetric,
such a structure is also called a quadratic (\cite{Med1}) or a
metric Lie algebra (\cite{Kat}).


\delete{
The Frobenius property  reflects a certain natural symmetry  and can be extended to nonassociative algebras (see \cite{Bor}).
For instance, a Lie algebra $(\mathfrak{g},[-,-])$ equipped with a  bilinear form $\mathcal{B}$ which is nondegenerate and invariant in the sense that
\begin{equation}\label{eq:Fro Lie}
    \mathcal{B}([x,y],z)=\mathcal{B}(x,[y,z]),\;\;\forall x,y,z\in \mathfrak{g}
\end{equation}
is called pseudo-metrised therein, and is called metrised when $\mathcal{B}$ is moreover symmetric.}
\delete{different types of algebras, such as Lie algebras (\cite{Ela, Oom}) and Poisson algebras (\cite{Luo,Zhu}) for different motivations.}

In order to study the Frobenius property of Jacobi algebras, the
notion of Frobenius Jacobi algebras was introduced in \cite{Ago}.
That is, a Jacobi algebra  $(A,\cdot,[-,-])$ is called Frobenius
if there exists a nondegenerate symmetric bilinear form $\mathcal
B$ on $A$ such that Eqs.~(\ref{eq:Fro asso}) and (\ref{eq:Fro
Lie}) hold. To get more examples of Frobenius Jacobi algebras is
obviously an important problem.


An important class of quadratic Lie algebras is given by the
following construction of (standard) Manin triples of Lie algebras
(\cite{CP1}). Suppose that $\frak g$ is a Lie algebra and there
exists a Lie algebra structure on the linear dual space $\frak
g^*$, such that there is a Lie algebra structure on the direct sum
$\frak g\oplus \frak g^*$ of vector spaces including $\frak g$ and
$\frak g^*$ as Lie subalgebras and the bilinear form
\begin{equation}\label{eq:BilinearForm1}
    \mathcal{B}_{d}(x+a^{*},y+b^{*})=\langle x,b^{*}\rangle+\langle
    a^{*},y\rangle,\;\;\forall x,y\in \frak g, a^{*},b^{*}\in \frak g^{*},
\end{equation}
on $\frak g\oplus \frak g^*$ is invariant. $(\frak g\oplus \frak
g^*,\frak g,\frak g^*)$ is called a (standard) Manin triple of Lie
algebras and obviously, the Lie algebra $\frak g\oplus \frak g^*$
with the symmetric nondegenerate bilinear form $\mathcal B_{d}$ given by
Eq.~(\ref{eq:BilinearForm1}) is a quadratic Lie algebra. The above
construction in the context of associative algebras is called a
double construction of a Frobenius algebra (\cite{Bai2010}), whereas in the context of Poisson algebras is called a Manin triple of Poisson algebras (\cite{NB}).

On the other hand, the above constructions also correspond to certain
bialgebra structures with different motivations and applications. A bialgebra is a coupling of an algebra and
a coalgebra satisfying certain compatible conditions. In fact, a
Manin triple of Lie algebras is equivalent to a Lie bialgebra
(\cite{CP1,Dri}) which serves as  the algebraic structure  of a
Poisson-Lie group  and plays an important role in the
infinitesimalization of a quantum group.
A double construction of a Frobenius algebra is
equivalent to  an antisymmetric infinitesimal bialgebra
(\cite{Agu2000, Agu2001, Agu2004, Bai2010}), which is a special infinitesimal bialgebra introduced by Joni and Rota in order to provide an algebraic framework for the calculus of divided
differences~(\cite{JR}). A Manin triple of Poisson algebras is equivalent to a Poisson bialgebra (\cite{NB}), which naturally fits into a framework to construct compatible Poisson brackets in integrable systems.

Therefore in order to get Frobenius Jacobi algebras, it is natural
to consider generalizing the above construction to the case of
Jacobi algebras, that is, to consider  the analogue structures
like ``Manin triples of Jacobi algebras" or ``double constructions
of Frobenius Jacobi algebras" and the corresponding bialgebra structures. Unfortunately, such an approach is
unavailable, since in general a unital commutative associative
algebra cannot be decomposed into the direct sum of the underlying
vector spaces of two unital commutative associative subalgebras.
That is, the structure of Manin triples does not make sense for
Jacobi algebras due to the existence of the units in the
commutative associative algebras. On the other hand,
alternatively, the approach is still available for relative
Poisson algebras, that is,  ``Manin triples of relative
Poisson algebras" make sense for relative Poisson algebras
without ``the constraint of the units". More interestingly, we
still may get examples of Frobenius Jacobi algebras from Manin
triples of relative Poisson algebras due to the fact that a
unital commutative  associative algebra might be decomposed into the
direct sum of the underlying vector spaces of two commutative
associative subalgebras, in which one of them is unital.

In this paper, we realize the above approach for relative Poisson
algebras to construct Frobenius Jacobi algebras. Explicitly, we
introduce the notions of Manin triples of relative Poisson
algebras and the corresponding bialgebra structures, namely,
relative Poisson bialgebras. The equivalence between them is
interpreted  in terms of matched pairs of relative Poisson
algebras. Relative Poisson bialgebras  share some similar
properties of Lie bialgebras or antisymmetric infinitesimal
bialgebras. In particular, there are also coboundary cases which
lead to the introduction of an analogue of the classical
Yang-Baxter equation in a Lie algebra, called relative Poisson
Yang-Baxter equation (RPYBE) in a relative Poisson algebra.
Antisymmetric solutions of the RPYBE in relative Poisson algebras
give rise to coboundary relative Poisson bialgebras, whereas the
notions of $\mathcal{O}$-operators of relative Poisson algebras
and relative pre-Poisson algebras are introduced to construct such
solutions. Moreover, one can get Frobenius Jacobi algebras from
relative Poisson bialgebras satisfying some additional conditions, and in particular, there is a construction of Frobenius Jacobi
algebras from relative pre-Poisson algebras.

\subsection{Layout of the paper}
 The paper is organized as
follows.

In Section 2, we introduce the notions of representations and
matched pairs of relative Poisson algebras. The condition that a
relative Poisson algebra is dually represented is considered in
order to construct a reasonable representation on the dual space.
There is a relationship between representations of relative
Poisson algebras and representations of Jacobi algebras.  We also
explain why there is not a ``matched pair theory" for Jacobi
algebras.

In Section 3, we introduce the notion of  Manin triples of
relative Poisson algebras, which are equivalent to  certain
matched pairs of relative Poisson algebras. Then we introduce the
notion of relative Poisson bialgebras as the equivalent structures
to the two former structures.

In Section 4, the coboundary relative Poisson bialgebras are
studied, leading to the introduction of the RPYBE in a relative
Poisson algebra. Antisymmetric solutions of the RPYBE in
relative Poisson algebras give rise to coboundary relative Poisson
bialgebras.  The notions of $\mathcal{O}$-operators of relative
Poisson algebras and relative pre-Poisson algebras are introduced
to construct antisymmetric solutions of the RPYBE and hence
relative Poisson bialgebras.

In Section 5, we give a construction of Frobenius Jacobi algebras
from relative Poisson bialgebras.  In particular, Frobenius Jacobi
algebras can be obtained from relative pre-Poisson algebras.

Throughout this paper,  unless otherwise specified, all the vector
spaces and algebras are finite-dimensional over an algebraically
closed field $\mathbb {K}$ of characteristic zero, although many
results and notions remain valid in the infinite-dimensional case.

\section{Representations and matched pairs of relative Poisson algebras}

In Section~\ref{ss:2.1}, we  introduce the notion of a
representation of a relative Poisson algebra.
The case when a relative Poisson algebra is dually represented  is then studied, which gives a reasonable representation on the
dual space.
 We also give a relationship between representations
of relative Poisson algebras and representations of Jacobi
algebras  given in \cite{Ago}. In Section~\ref{ss:2.2}, we
introduce the notion of matched pairs of relative Poisson algebras
and explain that a ``matched pair theory" is not available for
Jacobi algebras.

\delete{
At first, we give an identity for a relative Poisson algebra by
Eq.~(\ref{eq:GLR}) directly.

\begin{pro}
    Let $(A,\cdot, [-,-],P)$ be a relative Poisson algebra. Then
    the following identity holds:
    \begin{equation}\label{eq:GLR2}
        [x,y\cdot z]+[y,z\cdot x]+[z,x\cdot y]=P(x\cdot y\cdot z),\;\;\forall
        x, y,z\in A.
    \end{equation}
\end{pro}
}

\subsection{Representations of relative Poisson algebras and Jacobi algebras}\
\label{ss:2.1}

Let $(A,\cdot)$ be a commutative associative algebra. A pair $(\mu,V)$ is called a \textbf{representation} of $(A,\cdot)$ if $V$ is a vector space and $\mu:A\rightarrow\mathrm{End}(V)$ is a linear map satisfying
$$\mu(x\cdot y)=\mu(x)\mu(y),\;\;\forall x,y\in A.$$
Moreover, if $(A,\cdot)$ is unital with the unit $1_A$, then the
representation is assumed to be unital, that is, we assume
$\mu(1_{A})v=v,\forall v\in V.$ In fact, $(\mu,V)$ is a
representation of a (unital) commutative associative algebra
$(A,\cdot)$ if and only if the direct sum $A\oplus V$ of vector
spaces is a (unital) commutative associative algebra (the {\bf
semi-direct product}) by defining the multiplication on $A\oplus
V$ by
\begin{equation}\label{eq:SDASSO}
(x+u)\cdot(y+v)=x\cdot y+\mu(x)v+\mu(y)u,\;\;\forall x,y\in A, u,v\in V.
\end{equation}
We denote it by $A\ltimes_{\mu}V$. If $1_{A}$ is the unit of
$(A,\cdot)$, then it is also the unit of $A\ltimes_{\mu}V$.

 Let $\mathcal L:A\rightarrow \mathrm{End}(A)$
be a linear map defined by $\mathcal L(x)(y)=x\cdot y$ for all
$x,y\in A$. Then $(\mathcal{L},A)$ is a representation of
$(A,\cdot)$, called the \textbf{adjoint representation} of
$(A,\cdot)$.

Let $V$ be a vector space. For a linear map
$\varphi:A\rightarrow\mathrm{End}(V)$, define a linear map
$\varphi^{*}:A\rightarrow\mathrm{End}(V^{*})$ by
$$\langle\varphi^{*}(x)u^{*},v\rangle=-\langle u^{*},\varphi(x)v\rangle,\;\;\forall x\in A, u^{*}\in V^{*},v\in V,$$
where $\langle- , -\rangle$ is the ordinary pair between $V$ and
$V^*$. If $(\mu,V)$ is a representation of a commutative
associative algebra $(A,\cdot)$, then $(-\mu^{*},V^{*})$ is also a
representation of $(A,\cdot)$. In particular,
$(-\mathcal{L}^{*},A^{*})$ is a representation of $(A,\cdot)$.

Similarly, let $(A,[-,-])$ be a Lie algebra. A pair $(\rho,V)$ is
called a \textbf{representation} of $(A,[-,-])$ if $V$ is a vector
space and $\rho:A\rightarrow\mathrm{End}(V)$ is a linear map
satisfying
$$\rho([x,y])=[\rho(x),\rho(y)],\;\;\forall x,y\in A.$$
In fact, $(\rho,V)$ is a representation of a Lie algebra
$(A,[-,-])$ if and only if 
$A\oplus V$
is a Lie algebra (the {\bf semi-direct product}) by
defining the multiplication on $A\oplus V$ by
\begin{equation}\label{eq:SDLie}
[x+u,y+v]=[x,y]+\rho(x)v-\rho(y)u,\;\;\forall x,y\in A, u,v\in V.
\end{equation}
We denote it by $A\ltimes_{\rho}V$.

Let ${\rm ad}:A\rightarrow {\rm End}(A)$ be a linear map defined
by ${\rm ad}(x)(y)=[x,y]$ for all $x,y\in A$. Then
$(\mathrm{ad},A)$ is a representation of  $(A,[-,-])$, called the
\textbf{adjoint representation} of $(A,[-,-])$. Moreover, if
$(\rho,V)$ is a representation of a Lie algebra $(A,[-,-])$, then
$(\rho^{*},V^{*})$  is also a representation of $(A,[-,-])$. In
particular,
  $(\mathrm{ad}^{*},A^{*})$ is a representation of $(A,[-,-])$.

\begin{defi}\label{de:rep}
Let $(A,\cdot,[-,-],P)$ be a relative Poisson algebra,
$(\mu,V)$ be a representation of the commutative associative
algebra $(A,\cdot)$ and $(\rho,V)$ be a representation of the Lie
algebra $(A,[-,-])$.
 We say the two representations $(\mu,V)$ and
$(\rho,V)$ are \textbf{compatible}, or $(\mu,\rho,V)$ is a
\textbf{compatible structure} on $(A,P)$ if the following equation
holds:
\begin{equation}\label{eq:CompStru}
\rho(y)\mu(x)v-\mu(x)\rho(y)v+\mu([x,y])v-\mu(x\cdot
P(y))v=0,\;\;\forall x,y\in A, v\in V.
\end{equation}
If in addition, there is a linear map $\alpha: V\rightarrow V$
satisfying
\begin{equation}\label{eq:rep1}
\alpha(\mu(x)v)-\mu(P(x))v-\mu(x)\alpha(v)=0,
\end{equation}
\begin{equation}\label{eq:rep2}
\alpha(\rho(x)v)-\rho(P(x))v-\rho(x)\alpha(v)=0,
\end{equation}
\begin{equation}\label{eq:rep3}
\rho(x\cdot y)v-\mu(x)\rho(y)v-\mu(y)\rho(x)v+\mu(x\cdot y)\alpha(v)=0,
\end{equation}
for all  $x,y\in A, v\in V$, we say the quadruple
$(\mu,\rho,\alpha,V)$ is a
\textbf{representation} of 
$(A,P)$. Two representations $(\mu_{1},\rho_{1},\alpha_{1},V_{1})$
and  $(\mu_{2},\rho_{2},\alpha_{2},V_{2})$ of a relative
Poisson algebra $(A,P)$ are called \textbf{equivalent} if there
exists a linear isomorphism $\varphi:V_{1}\rightarrow V_{2}$ such
that
\begin{equation}\label{eq:eqrep1}
\varphi(\mu_{1}(x)v)=\mu_{2}(x)\varphi(v),\;
\varphi(\rho_{1}(x)v)=\rho_{2}(x)\varphi(v),\;
\varphi(\alpha_{1}(v))=\alpha_{2}(\varphi(v)),\;\;\forall x\in A, v\in
V.
\end{equation}
\end{defi}

For vector spaces $V_1$ and $V_2$, and linear maps
$\phi_1:V_1\rightarrow V_1$ and $\phi_2:V_2\rightarrow V_2$, we
abbreviate $\phi_1+\phi_2$ for the linear map
\begin{equation}
\phi_{V_1\oplus V_2}: V_1\oplus V_2\rightarrow V_1\oplus V_2,
\quad \phi_{V_1\oplus V_2}(v_1+v_2):=\phi_1(v_1)+\phi_2(v_2),\;\;
\forall v_1\in V_1, v_2\in V_2. \label{eq:Liehom}
\end{equation}

\begin{pro}\label{pro:semipro}
Let $(A,P)$ be a relative Poisson algebra, $V$ be a vector
space,  and  $\mu,\rho: A\rightarrow \mathrm{End}(V)$ and $\alpha:
V\rightarrow V$ be linear maps. Define two bilinear operations
$\cdot,[-,-]$ on $A\oplus V$ by Eqs.~(\ref{eq:SDASSO}) and
(\ref{eq:SDLie}) respectively. Then $(A\oplus V,\cdot, [-,-],
P+\alpha)$ is a relative Poisson algebra if and only if
$(\mu,\rho,\alpha,V)$ is a representation of $(A,P)$. We denote this
relative Poisson algebra structure ({\bf semi-direct product}) on $A\oplus V$ by
$(A\ltimes_{\mu,\rho}V,P+\alpha)$ or simply $(A\ltimes
V,P+\alpha)$.
\end{pro}

The proof is omitted since this result is a special case of the matched pairs of relative Poisson algebras in Proposition~\ref{pro:matched pairs}, when $A_{2}=V$ is equipped with the zero multiplication.

\begin{ex}
Let $(A,P)$ be a relative Poisson algebra. Then
$(\mathcal{L},\mathrm{ad},P,A)$ is a representation of $(A,P)$, called
the \textbf{adjoint representation} of $(A,P)$.
\end{ex}

\begin{lem}
Let $(A,P)$ be a relative Poisson algebra. If $(\mu,\rho,V)$ is
a compatible structure on $(A,P)$, then
$(-\mu^{*},\rho^{*},V^{*})$ is also a compatible structure on
$(A,P)$.
\end{lem}
\begin{proof}
For all $x,y\in A,u^{*}\in V^{*},v\in V$, we have
\begin{eqnarray*}
&&\langle(-\rho^{*}(y)\mu^{*}(x)+\mu^{*}(x)\rho^{*}(y)-\mu^{*}([x,y])+\mu^{*}(x\cdot P(y))u^{*},v\rangle\\
&&=\langle
u^{*},\rho(y)\mu(x)v-\mu(x)\rho(y)v+\mu([x,y])v-\mu(x\cdot
P(y))v\rangle=0.
\end{eqnarray*}
Hence the conclusion holds.
\end{proof}

Let $V_{1},V_{2}$ be two vector spaces and $T:V_{1}\rightarrow
V_{2}$ be a linear map. Denote the dual map
$T^{*}:V_{2}^{*}\rightarrow V_{1}^{*}$ by
$$\langle v_{1},T^{*}(v_{2}^{*})\rangle=\langle T(v_{1}),v_{2}^{*}\rangle,\;\;\forall v_{1}\in V_{1},v^{*}_{2}\in V^{*}_{2}.$$

\begin{pro}\label{pro:dual}
Let $(A,P)$ be a relative Poisson algebra and $(\mu,\rho,V)$ be
a compatible structure on $(A,P)$.  Let $\beta:V\rightarrow V$ be
a linear map. Then $(-\mu^{*},\rho^{*},\beta^{*},V^{*})$ is a
representation of $(A,P)$ if and only if the following equations
are satisfied:
\begin{equation}\label{eq:dualrep1}
\mu(x)\beta(v)-\mu(P(x))v-\beta(\mu(x)v)=0,
\end{equation}
\begin{equation}\label{eq:dualrep2}
\rho(x)\beta(v)-\rho(P(x))v-\beta(\rho(x)v)=0,
\end{equation}
\begin{equation}\label{eq:dualrep3}
-\rho(x\cdot y)v+\rho(y)\mu(x)v+\rho(x)\mu(y)v+\beta(\mu(x\cdot
y)v)=0,
\end{equation}
 for all $x,y\in A,v\in V$.
\end{pro}
\begin{proof}
For all $x,y\in A, u^{*}\in V^{*},v\in V$, we have
\begin{eqnarray*}&&\langle
\mu^{*}(P(x))u^{*}+\mu^{*}(x)\beta^{*}(u^{*})-\beta^{*}(\mu^{*}(x)u^{*}),v\rangle
    =\langle u^{*},
    -\mu(P(x))v-\beta(\mu(x)v)+\mu(x)\beta(v)\rangle,\\
&&\langle
\rho^{*}(P(x))u^{*}+\rho^{*}(x)\beta^{*}(u^{*})-\beta^{*}(\rho^{*}(x)u^{*}),v\rangle
    =\langle u^{*},
    -\rho(P(x))v-\beta(\rho(x)v)+\rho(x)\beta(v)\rangle,\\
&&\langle \rho^{*}(x\cdot y)u^{*}+\mu^{*}(x)\rho^{*}(y)u^{*}+\mu^{*}(y)\rho^{*}(x)u^{*}-\mu^{*}(x\cdot y)\beta^{*}(u^{*}),v\rangle\\
    &&=\langle u^{*},-\rho(x\cdot y)v+\rho(y)\mu(x)v+\rho(x)\mu(y)v+\beta(\mu(x\cdot y)v)\rangle.
\end{eqnarray*}
Then by Definition \ref{de:rep},
$(-\mu^{*},\rho^{*},\beta^{*},V^{*})$ is a representation of
$(A,P)$ if and only if Eqs.~
(\ref{eq:dualrep1})-(\ref{eq:dualrep3}) hold.
\end{proof}

\begin{defi}\label{defi:dual}
Let $(A,P)$ be a relative Poisson algebra and $(\mu,\rho,V)$ be
a compatible structure on $(A,P)$. Let $\beta:V\rightarrow V$ be a
linear map. If $(-\mu^{*},\rho^{*},\beta^{*},V^{*})$ is a
representation of $(A,P)$, that is, Eqs.~
(\ref{eq:dualrep1})-(\ref{eq:dualrep3}) hold, then we say that
$\beta$ \textbf{dually represents  $(A,P)$ on $(\mu,\rho,V)$}.
When $(\mu,\rho,V)$ is taken to be $(\mathcal{L},\mathrm{ad},A)$
which is the compatible structure of the relative Poisson
algebra $(A,P)$ composing the adjoint representation
$(\mathcal{L},\mathrm{ad},P,A)$, we simply say that $\beta$
\textbf{dually represents} $(A,P)$.
\end{defi}

By Proposition~\ref{pro:dual}, we have the following conclusion.

\begin{cor}\label{cor:dualadj}
Let $(A,P)$ be a relative Poisson algebra. A linear map $Q:
A\rightarrow A$ dually represents $(A,P)$ if and only if the
following equations are satisfied:
   \begin{equation}\label{eq:dualadj1}
   x\cdot Q(y)-P(x)\cdot y-Q(x\cdot y)=0,
   \end{equation}
   \begin{equation}\label{eq:dualadj2}
   [x,Q(y)]-[P(x),y]-Q([x,y])=0,
   \end{equation}
   \begin{equation}\label{eq:dualadj3}
   [x,y\cdot z]+[y,z\cdot x]+[z,x\cdot y]+Q(x\cdot y\cdot z)=0,
   \end{equation}
for all $x,y,z\in A$.
\end{cor}

\begin{pro}\label{pro:equivalentconditions}
Let $(A,P)$ be a relative Poisson algebra,
$(\mu,\rho,\alpha,V)$ be a representation of $(A,P)$ and $\beta:
V\rightarrow V$ be a linear map.
\begin{enumerate}
\item Eq.~(\ref{eq:dualrep1}) holds if and only if the following
condition holds:
\begin{equation}\label{eq:alpha-beta1}
(\alpha+\beta)(\mu (x)v)-\mu(x)(\alpha+\beta)(v)=0,\;\;\forall
x\in A,v\in V.\end{equation} \item Eq.~(\ref{eq:dualrep2}) holds
if and only if the following condition holds:
\begin{equation}\label{eq:alpha-beta2}
(\alpha+\beta)(\rho (x)v)-\rho(x)(\alpha+\beta)(v)=0,\;\;\forall
x\in A,v\in V.\end{equation} \item Eq.~(\ref{eq:dualrep3}) holds
if and only if the following equation holds:
\begin{equation}\label{eq:eqdualrep1}
(\alpha+\beta)(\mu(x\cdot y)v)=0,\;\;\forall x,y\in A, v\in V.
\end{equation}
\item If Eq.~(\ref{eq:dualrep1}) holds, then Eq.~
(\ref{eq:eqdualrep1}) holds if and only if the following equation
holds:
\begin{equation}\label{eq:eqdualrep2}
\mu(x\cdot y)(\alpha+\beta)v=0,\;\;\forall x,y\in A, v\in V.
\end{equation}
\end{enumerate}
\end{pro}

\begin{proof} Let $x,y\in A$ and $v\in V$.

(1) By Eq.~(\ref{eq:rep1}), we have
$$\mu(x)\beta(v)-\mu(P(x))v-\beta(\mu(x)v)=-(\alpha+\beta)(\mu
(x)v)+\mu(x)(\alpha+\beta)(v).$$ Hence Eq.~(\ref{eq:dualrep1})
holds if and only if Eq.~(\ref{eq:alpha-beta1}) holds.

(2) By Eq.~(\ref{eq:rep2}), we have
$$\rho(x)\beta(v)-\rho(P(x))v-\beta(\rho(x)v)=-(\alpha+\beta)(\rho
(x)v)+\rho(x)(\alpha+\beta)(v).$$ Hence Eq.~(\ref{eq:dualrep2})
holds if and only if Eq.~(\ref{eq:alpha-beta2}) holds.

(3) We have
\begin{eqnarray*}
&&-\rho(x\cdot y)v+\rho(y)\mu(x)v+\rho(x)\mu(y)v+\beta(\mu(x\cdot y)v)\\
&&\stackrel{(\ref{eq:rep3})}{=}-\mu(x)\rho(y)v-\mu(y)\rho(x)v+\mu(x\cdot y)\alpha(v)+\rho(y)\mu(x)v+\rho(x)\mu(y)v+\beta(\mu(x\cdot y)v)\\
&&\stackrel{(\ref{eq:CompStru})}{=}\mu(x\cdot P(y))-\mu([x,y])v-\rho(y)\mu(x)v+\mu(y\cdot P(x))-\mu([y,x])v-\rho(x)\mu(y)v\\
&&\mbox{}\hspace{0.3cm}+\mu(x\cdot y)\alpha(v)+\rho(y)\mu(x)v+\rho(x)\mu(y)v+\beta(\mu(x\cdot y)v)\\
&&=\mu(P(x\cdot y))v+\mu(x\cdot y)\alpha(v)+\beta(\mu(x\cdot y)v)\\
&&\stackrel{(\ref{eq:rep1})}{=}(\alpha+\beta)\mu(x\cdot y)v,
\end{eqnarray*}
 Thus Eq.~(\ref{eq:dualrep3}) holds if and only if
Eq.~(\ref{eq:eqdualrep1}) holds.

(4) If Eq.~(\ref{eq:dualrep1}) holds, then by Eq.~(\ref{eq:rep1})
again, we have $$ (\alpha+\beta)\mu(x\cdot y)v =\mu(x\cdot
y)\beta(v)-\mu(P(x\cdot y))v+\mu(P(x\cdot y))v+\mu(x\cdot
y)\alpha(v)=\mu(x\cdot y)(\alpha+\beta)v.
$$
Hence in this case, Eq.~(\ref{eq:eqdualrep1}) holds if and only if
Eq.~(\ref{eq:eqdualrep2}) holds.
\end{proof}

\begin{cor}\label{cor:automap}
Let $(\mu,\rho,\alpha,V)$ be a representation of a relative
Poisson algebra $(A,P)$. Then $-\alpha$ dually represents $(A,P)$
on $(\mu,\rho,V)$ automatically, that is,
$(-\mu^*,\rho^*,-\alpha^*,V^*)$ is a representation of $(A,P)$.
In particular, $-P$ dually represents $(A,P)$.
\end{cor}

\begin{proof}
The first part of the conclusion follows from
Proposition~\ref{pro:equivalentconditions} since $\beta=-\alpha$
satisfies Eqs.~(\ref{eq:alpha-beta1})-(\ref{eq:eqdualrep1}). The
second part of the conclusion follows immediately when
$(\mu,\rho,\alpha,V)$ is taken to be the adjoint representation
$(\mathcal{L},\mathrm{ad},P,A)$.
\end{proof}

\begin{cor}\label{cor:dualadj2}
Let $(A,P)$ be a relative Poisson algebra and $Q:A\rightarrow
A$ be a linear map. Then $Q$ dually represents $(A,P)$ if and only
if Eqs.~(\ref{eq:dualadj1}), (\ref{eq:dualadj2}) and the following
equation hold:
\begin{equation}\label{eq:eqdualadj1}
(P+Q)(x\cdot y\cdot z)=0,\;\;\forall x,y,z\in A.
\end{equation}
In particular, if  $Q$ dually represents $(A,P)$, then the
following equation holds:
\begin{equation}\label{eq:eqdualadj2}
x\cdot y\cdot (P+Q)(z)=0,\;\;\forall x,y,z\in A.
\end{equation}
\end{cor}

\begin{proof}
It follows from Proposition~\ref{pro:equivalentconditions} when
the representation $(\mu,\rho,\alpha,V)$ is taken to be the
adjoint representation $(\mathcal{L},\mathrm{ad},P,A)$ of $(A,P)$.
\end{proof}

At the end of the subsection, we give a relationship between
representations of Jacobi algebras and representations of (unital)
relative Poisson algebras.

\begin{defi}\label{de:repJacobi} {\rm (\cite{Ago})}
Let $(A,\cdot,[-,-])$ be a Jacobi algebra. A triple $(\mu,\rho,V)$
is called a \textbf{representation} of $(A,\cdot,[-,-])$ if
$(\mu,V)$ is a representation of the unital commutative
associative algebra $(A,\cdot)$, $(\rho,V)$ is a representation of
the Lie algebra $(A,[-,-])$, and they satisfy the following
conditions:
\begin{equation}\label{eq:de:repJacobi2}
\rho(x\cdot y)v-\mu(x)\rho(y)v-\mu(y)\rho(x)v+\mu(x\cdot y)\rho(1_{A})v=0,
\end{equation}
\begin{equation}\label{eq:de:repJacobi3}
\rho(y)\mu(x)v-\mu(x)\rho(y)v+\mu([x,y])v-\mu(x\cdot[1_{A},y])v=0,
\end{equation}
for all $x,y\in A,v\in V$.
\end{defi}

The following conclusion is a direct consequence of
\cite[Theorem 3.2]{Ago}.

\begin{pro}\label{pro:semiJacobi} 
Suppose that $(A,\cdot,[-,-])$ is a Jacobi algebra, $V$ is a vector
space and $\mu,\rho:A\rightarrow\mathrm{End}(V)$ are linear maps.
Then $(\mu,\rho,V)$ is a representation of $(A,\cdot,[-,-])$ if
and only if there is a Jacobi algebra structure on the direct sum
$A\oplus V$ of vector spaces, where the bilinear operations $\cdot$ and
$[-,-]$ on $A\oplus V$ are given by Eqs.~(\ref{eq:SDASSO}) and
(\ref{eq:SDLie}) respectively. We denote this Jacobi algebra
structure on $A\oplus V$ by $A\ltimes_{\mu,\rho}V$.
\end{pro}



\begin{pro}\label{pro:repJacobi and GPA}
Let $(A,\cdot,[-,-])$ be a Jacobi algebra and $(A,\cdot,$
$[-,-],\mathrm{ad}(1_{A}))$ be the corresponding unital
relative Poisson algebra. Then $(\mu,\rho,V)$ is a
representation of the Jacobi algebra $(A,\cdot,[-,-])$ if and only
if $(\mu,\rho,\rho(1_{A}),V)$ is a representation of the unital
relative Poisson algebra $(A,\cdot,$
$[-,-],\mathrm{ad}(1_{A}))$.
\end{pro}

\begin{proof}
Suppose that $(\mu,\rho,V)$ is a representation of the Jacobi
algebra $(A,\cdot,[-,-])$. Then $A\ltimes_{\mu,\rho}V$ is also a
Jacobi algebra with the unit $1_{A}$. Since
$$[1_{A},x+u]=[1_{A},x]+\rho(1_{A})u=(\mathrm{ad}(1_{A})+\rho(1_{A}))(x+u),\;\;\forall x\in A, v\in V,$$
$(A\ltimes_{\mu,\rho}V,\mathrm{ad}(1_{A})+\rho(1_{A}))$ is a
unital relative Poisson algebra.  Hence by
Proposition~\ref{pro:semipro}, $(\mu,\rho$, $\rho(1_{A})$, $V)$ is
a representation of the unital relative Poisson algebra
$(A,\cdot,[-,-],\mathrm{ad}(1_{A}))$. Conversely, by a similar
proof, the conclusion is obtained.
\end{proof}

\begin{cor}\label{pro:dualrepJA}
Let $(\mu,\rho,V)$ be a representation of a Jacobi algebra
$(A,\cdot,[-,-])$. Then $(-\mu^{*},\rho^{*}$, $V^{*})$ is also a
representation of $(A,\cdot,[-,-])$.
\end{cor}

\begin{proof} Since
$$\langle \rho^*(1_A)u^*, v\rangle=\langle
u^*,-\rho(1_A)v\rangle=\langle -(\rho(1_A))^*u^*,v\rangle,
\;\;\forall u^*\in V^*, v\in V,$$ we have
$\rho^*(1_A)=-(\rho(1_A))^*$. If  $(\mu,\rho,V)$ is a
representation of a Jacobi algebra $(A,\cdot,[-,-])$, then by
Proposition~\ref{pro:repJacobi and GPA},
$(\mu,\rho,\rho(1_{A}),V)$ is a representation of the unital
relative Poisson algebra $(A,\cdot,$
$[-,-],\mathrm{ad}(1_{A}))$. By Corollary~\ref{cor:automap},
$(-\mu^*,\rho^*,-(\rho(1_{A}))^*=\rho^*(1_A),V^*)$ is a
representation of the unital relative Poisson algebra
$(A,\cdot,$ $[-,-],\mathrm{ad}(1_{A}))$. By
Proposition~\ref{pro:repJacobi and GPA} again,
$(-\mu^{*},\rho^{*},V^{*})$ is  a representation of the Jacobi
algebra $(A,\cdot,[-,-])$.
\end{proof}

\begin{rmk} Note that for a representation $(\mu,\rho,V)$ of a Jacobi algebra $(A,\cdot,[-,-])$,  $(-\mu^{*},\rho^{*}$, $V^{*})$ is automatically a representation
of $(A,\cdot,[-,-])$ without any additional condition, whereas in
\cite{Ago}, there is the following additional equation for
$(-\mu^{*},\rho^{*},V^{*})$ to be a representation of the Jacobi
algebra $(A,\cdot,[-,-])$:
\begin{equation}\label{pr6}
-\rho(x\cdot
y)v+\rho(y)\mu(x)v+\rho(x)\mu(y)v-\rho(1_{A})\mu(x\cdot
y)v=0,\;\;\forall x,y\in A, v\in V.
\end{equation}
In fact, Eq.~(\ref{pr6}) is not needed, that is, by a direct
proof, Eq.~(\ref{pr6}) can be obtained by Eqs.
(\ref{eq:de:repJacobi2}) and (\ref{eq:de:repJacobi3}).
\end{rmk}


\subsection{Matched pairs of relative Poisson algebras}\
\label{ss:2.2}

We first recall the notions of matched pairs of commutative associative algebras
(\cite{Bai2010}) and Lie algebras (\cite{Maj}) respectively.

Let $(A_{1},\cdot_{1})$ and $(A_{2},\cdot_{2})$ be two commutative
associative algebras, and $(\mu_{1},A_{2})$ and $(\mu_{2},A_{1})$
be representations of $(A_{1},\cdot_{1})$ and $(A_{2},\cdot_{2})$
respectively. If the following equations are satisfied:
$$\mu_{1}(x)(a\cdot_{2} b)=(\mu_{1}(x)a)\cdot_{2} b+\mu_{1}(\mu_{2}(a)x)b,$$
$$\mu_{2}(a)(x\cdot_{1} y)=(\mu_{2}(a)x)\cdot_{1} y+\mu_{2}(\mu_{1}(x)a)y,$$
 for all $x,y\in A_{1},a,b\in A_{2}$, then $(A_{1},A_{2},\mu_{1},\mu_{2})$ is called a \textbf{matched
pair of commutative associative algebras}. In this case, there
exists a commutative associative algebra structure on the vector
space  $A_{1}\oplus A_{2}$  given by
\begin{equation}\label{eq:Asso} (x+a)\cdot (y+b)=x\cdot_{1}
y+\mu_{2}(a)y+\mu_{2}(b)x+a\cdot_{2}
b+\mu_{1}(x)b+\mu_{1}(y)a,\;\;\forall x,y\in A_1, a,b\in A_2.
\end{equation}
Moreover, every commutative associative algebra which is the
direct sum of the underlying vector spaces of two subalgebras can
be obtained from a matched pair of commutative associative
algebras.

Let $(A_{1},[-,-]_{1})$ and $(A_{2},[-,-]_{2})$ be two Lie
algebras, and $(\rho_{1},A_{2})$ and $(\rho_{2},A_{1})$ be
representations of $(A_{1},[-,-]_{1})$ and $(A_{2},[-,-]_{2})$
respectively. If the following equations are satisfied:
$$\rho_{1}(x)[a,b]_{2}-[\rho_{1}(x)a,b]_{2}-[a,\rho_{1}(x)b]_{2}+\rho_{1}(\rho_{2}(a)x)b-\rho_{1}(\rho_{2}(b)x)a=0,$$
$$\rho_{2}(a)[x,y]_{1}-[\rho_{2}(a)x,y]_{1}-[x,\rho_{2}(a)y]_{1}+\rho_{2}(\rho_{1}(x)a)y-\rho_{2}(\rho_{1}(y)a)x=0,$$
for all $x,y\in A_{1},a,b\in A_{2}$, then
$(A_{1},A_{2},\rho_{1},\rho_{2})$ is called a \textbf{matched pair
of Lie algebras}. In this case, there is a Lie algebra structure
on the vector space $A_{1}\oplus A_{2}$ given by
\begin{equation}\label{eq:Lie}
[x+a,y+b]=[x,y]_{1}+\rho_{2}(a)y-\rho_{2}(b)x+[a,b]_{2}+\rho_{1}(x)b-\rho_{1}(y)a,\;\;\forall
x,y\in A_1, a,b\in A_2.
\end{equation}
Moreover, every Lie algebra which is the direct sum of the
underlying vector spaces of two subalgebras can be obtained from a
matched pair of Lie algebras.

Now we consider the case of relative Poisson algebras.
\begin{defi}
Let $(A_{1},\cdot_{1},[-,-]_{1},P_{1})$ and
$(A_{2},\cdot_{2},[-,-]_{2},P_{2})$ be two relative Poisson
algebras. Suppose that $(\mu_{1},\rho_{1},P_{2},A_{2})$ is a
representation of $(A_{1},\cdot_{1},[-,-]_{1},P_{1})$ and
$(\mu_{2},\rho_{2},P_{1},A_{1})$ is a representation of
$(A_{2},\cdot_{2},[-,-]_{2},P_{2})$, such that
$(A_{1},A_{2},\mu_{1},\mu_{2})$ is a matched pair of commutative
associative algebras and $(A_{1},A_{2},\rho_{1},\rho_{2})$ is a
matched pair of Lie algebras. Suppose that the following
compatible conditions are satisfied:
\begin{equation}\label{eq:MP1}
\rho_{2}(a)(x\cdot_{1} y)+\mu_{2}(\rho_{1}(y)a)x-x\cdot_{1}\rho_{2}(a)y+\mu_{2}(\rho_{1}(x)a)y-y\cdot_{1}\rho_{2}(a)x-\mu_{2}(P_{2}(a))(x\cdot_{1} y)=0,
\end{equation}
\begin{equation}\label{eq:MP2}
\rho_{1}(x)(a\cdot_{2} b)+\mu_{1}(\rho_{2}(b)x)a-a\cdot_{2}\rho_{1}(x)b+\mu_{1}(\rho_{2}(a)x)b-b\cdot_{2}\rho_{1}(x)a-\mu_{1}(P_{1}(x))(a\cdot_{2} b)=0,
\end{equation}
\begin{equation}\label{eq:MP3}
\rho_{2}(\mu_{1}(x)a)y+[\mu_{2}(a)x,y]_{1}-x\cdot_{1}\rho_{2}(a)y+\mu_{2}(\rho_{1}(y)a)x-\mu_{2}(a)([x,y]_{1})+\mu_{2}(a)(x\cdot_{1} P_{1}(y))=0,
\end{equation}
\begin{equation}\label{eq:MP4}
\rho_{1}(\mu_{2}(a)x)b+[\mu_{1}(x)a,b]_{2}-a\cdot_{2}\rho_{1}(x)b+\mu_{1}(\rho_{2}(b)x)a-\mu_{1}(x)([a,b]_{2})+\mu_{1}(x)(a\cdot_{2}
P_{2}(b))=0,
\end{equation}
 for all $x,y\in A_{1},a,b\in
A_{2}$. Such a structure is called  a \textbf{matched pair of
relative Poisson algebras} $(A_{1},P_{1})$ and $(A_{2},P_{2})$.
We denote it by
$((A_{1},P_{1}),(A_{2},P_{2}),\mu_{1},\rho_{1},\mu_{2},\rho_{2})$.
\end{defi}

\begin{pro}\label{pro:matched pairs}
Suppose that $(A_{1},P_{1})$ and $(A_{2},P_{2})$ are relative
Poisson algebras. For linear maps
$\mu_{1},\rho_{1}:A_{1}\rightarrow\mathrm{End}(A_{2})$ and
$\mu_{2},\rho_{2}:A_{2}\rightarrow\mathrm{End}(A_{1})$, define two bilinear
operations  $\cdot$ and $[-,-]$ on $A_{1}\oplus A_{2}$ by
Eqs.~(\ref{eq:Asso}) and (\ref{eq:Lie}) respectively. Then
$(A_{1}\oplus A_{2}, \cdot, [-,-], P_{1}+P_{2})$ is a relative
Poisson algebra if and only if
$((A_{1},P_{1}),(A_{2},P_{2}),\mu_{1},\rho_{1},\mu_{2},\rho_{2})$
is a matched pair of relative Poisson algebras. In this case,
we denote this relative Poisson algebra by $(A_1\bowtie A_2,
P_1+P_2)$. Moreover, every relative Poisson algebra which is
the direct sum of the underlying vector spaces of two subalgebras
can be obtained from a matched pair of relative Poisson
algebras.
\end{pro}
\begin{proof}
It is known that $(A_1\oplus A_2,\cdot)$ is a commutative
associative algebra if and only if $(A_{1},A_{2},\mu_{1},\mu_{2})$
is a matched pair of commutative associative algebras and
$(A_1\oplus A_2, [-,-])$ is a Lie algebra if and only if
$(A_{1},A_{2},\rho_{1},\rho_{2})$ is a matched pair of Lie
algebras.

Let $x,y\in A_{1}, a,b\in A_{2}$. By Eq.~(\ref{eq:Lie}), we have
\begin{eqnarray*}
&&(P_{1}+P_{2})([x+a,y+b])
=P_{1}([x,y]_{1}+\rho_{2}(a)y-\rho_{2}(b)x)+P_{2}([a,b]_{2}+\rho_{1}(x)b-\rho_{1}(y)a),\\
&&[(P_{1}+P_{2})(x+a),y+b]=
[P_{1}(x),y]_1+\rho_{2}(P_{2}(a)y)-\rho_{2}(b)P_{1}(x)+[P_{2}(a),b]_2+\rho_{1}(P_{1}(x))b\\&&\mbox{}\hspace{4.8cm}-\rho_{1}(y)P_{2}(a),\\
&&[x+a,(P_{1}+P_{2})(y+b)]
=[x,P_{1}(y)]_1+\rho_{2}(a)P_{1}(y)-\rho_{2}(P_{2}(b))x+[a,P_{2}(b)]_2+\rho_{1}(x)P_{2}(b)\\&&\mbox{}\hspace{4.8cm}-\rho_{1}(P_{1}(y))a.
\end{eqnarray*}
If $P_1+P_2$ is a derivation of the Lie algebra $(A_1\oplus A_2,
[-,-])$, then $P_1$ and $P_2$ are derivations of the Lie algebras
$(A_1,[-,-]_1)$ and $(A_2,[-,-]_2)$ by taking $a=b=0$ and $x=y=0$
respectively. Moreover, Eq.~(\ref{eq:rep2}) holds for
$(\mu_{1},\rho_{1},P_{2},A_{2})$ and
$(\mu_{2},\rho_{2},P_{1},A_{1})$ as the representations of
$(A_1,P_1)$ and $(A_2,P_2)$  by taking $a=y=0$ and
$x=b=0$ respectively. Conversely, if $P_1$ and $P_2$ are derivations of the Lie
algebras $(A_1,[-,-]_1)$ and $(A_2,[-,-])$ respectively and
$(\mu_{1},\rho_{1},P_{2},A_{2})$ and
$(\mu_{2},\rho_{2},P_{1},A_{1})$ are representations of
$(A_1,P_1)$ and $(A_2,P_2)$ respectively, then by
Eq.~(\ref{eq:rep2}), $P_{1}+P_{2}$ is a derivation of the Lie
algebra $(A_1\oplus A_2, [-,-])$. Similarly, $P_{1}+P_{2}$ is a
derivation of the commutative associative algebra $(A_1\oplus A_2,
\cdot)$ if and only if $P_1$ and $P_2$ are derivations of the
commutative associative  algebras $(A_1,\cdot_1)$ and
$(A_2,\cdot_2)$ respectively and Eq.~(\ref{eq:rep1}) holds for
$(\mu_{1},\rho_{1},P_{2},A_{2})$ and
$(\mu_{2},\rho_{2},P_{1},A_{1})$ as the representations of
$(A_1,P_1)$ and $(A_2,P_2)$ respectively.

Let $x,y,z\in A_1, a,b,c\in A_2$. We consider the relative
Leibniz rule~(\ref{eq:GLR}) on $A_{1}\oplus A_{2}$:
\begin{equation}\label{eq:GLR3}
[z+c,(x+a)\cdot(y+b)]=(x+a)\cdot[z+c,y+b]+[z+c,x+a]\cdot(y+b)+(x+a)\cdot(y+b)\cdot(P_{1}+P_{2})(z+c).
\end{equation}
If Eq.~(\ref{eq:GLR3}) holds, then Eq.~(\ref{eq:GLR}) holds for
$(A_1,P_1)$ and $(A_2,P_2)$ as relative Poisson algebras and
the following equations hold:
\begin{equation}\label{eq:DS1}
[c,x\cdot_{1}y]=x\cdot[c,y]+[c,x]\cdot y+(x\cdot_{1}y)\cdot
P_{2}(c),
\end{equation}
\begin{equation}\label{eq:DS2}
[z,a\cdot_{2}b]=a\cdot[z,b]+[z,a]\cdot b+(a\cdot_{2}b)\cdot
P_{1}(z),
\end{equation}
\begin{equation}\label{eq:DS3}
[z,x\cdot b]=x\cdot[z,b]+[z,x]_{1}\cdot b+b\cdot(x\cdot_{1}
P_{1}(z)),
\end{equation}
\begin{equation}\label{eq:DS4}
[c,x\cdot b]=b\cdot[c,x]+[c,b]_{2}\cdot
x+x\cdot(b\cdot_{2}P_{2}(c)).
\end{equation}
by taking $a=b=c=0$, $x=y=z=0$, $a=b=z=0,x=y=c=0,a=y=c=0,a=y=z=0$
respectively in Eq.~(\ref{eq:GLR3}). Conversely, since
$(A_{1},P_{1})$ and $(A_{2},P_{2})$ are relative Poisson
algebras, if Eqs.~(\ref{eq:DS1})-(\ref{eq:DS4}) hold, then it is
straightforward to show that Eq.~(\ref{eq:GLR3}) holds on
$A_{1}\oplus A_{2}$. Moreover, we have
\begin{eqnarray*}
&&[c,x\cdot_{1}y]=\rho_{2}(c)(x\cdot_{1}y)-\rho_{1}(x\cdot_{1}y)c,\\
&&x\cdot[c,y]=x\cdot(\rho_{2}(c)y-\rho_{1}(y)c)=x\cdot_{1}\rho_{2}(c)y-\mu_{1}(x)\rho_{1}(y)c-\mu_{2}(\rho_{1}(y)c)x,\\
&&[c,x]\cdot y=(\rho_{2}(c)x-\rho_{1}(x)c)\cdot
y=\rho_{2}(c)x\cdot_{1}y-\mu_{1}(y)\rho_{1}(x)c-\mu_{2}(\rho_{1}(x)c)y,\\
&&(x\cdot_{1}y)\cdot
P_{2}(c)=\mu_{1}(x\cdot_{1}y)P_{2}(c)+\mu_{2}(P_{2}(c))(x\cdot_{1}y).
\end{eqnarray*}
Therefore Eq.~(\ref{eq:DS1}) holds if and only if
Eq.~(\ref{eq:MP1}) by replacing $a$ by $c$, and
Eq.~(\ref{eq:rep3}) for $(\mu_1,\rho_1,P_2,A_2)$ as a
representation of $(A_1,P_1)$ hold. Similarly, we have
\begin{enumerate}
\item[(1)] Eq.~(\ref{eq:DS2}) holds if and only if
Eq.~(\ref{eq:MP2}) by replacing $x$ by $z$, and
Eq.~(\ref{eq:rep3}) for $(\mu_2,\rho_2,P_1,A_1)$ as a
representation of $(A_2,P_2)$ hold; \item[(2)] Eq.~(\ref{eq:DS3})
holds if and only if Eq.~(\ref{eq:MP3}) by replacing $a$ by $b$,
$y$ by $z$, and Eq.~(\ref{eq:CompStru}) for
$(\mu_1,\rho_1,P_2,A_2)$ as a representation of $(A_1,P_1)$ hold;
\item[(3)] Eq.~(\ref{eq:DS4}) holds if and only if
Eq.~(\ref{eq:MP4}) by replacing $a$ by $b$, $b$ by $c$, and
Eq.~(\ref{eq:CompStru}) for $(\mu_2,\rho_2,P_1,A_1)$ as a
representation of $(A_2,P_2)$ hold.
\end{enumerate}
Hence the conclusion holds.
\end{proof}

\begin{ex}
    Let $(A_{1},\cdot_{1})$ and $(A_{2},\cdot_{2})$ be two commutative associative
    algebras, and $P_{1}$ and $P_{2}$ be their derivations
    respectively. Let $(A_{1},[-,-]_{1})$ and $(A_{2}$,
    $[-,-]_{2})$ be the Lie algebras defined by
    Eq.~(\ref{eq:EXGPA}) respectively. Hence
    $(A_1,\cdot_1,[-,-]_1,P_1)$ and $(A_2,\cdot_2,[-,-]_2,P_2)$
    are relative Poisson algebras. Suppose that there are linear maps
$\mu_{1}:A_{1}\rightarrow\mathrm{End}(A_{2})$ and
$\mu_{2}:A_{2}\rightarrow\mathrm{End}(A_{1})$, such that
$(A_{1},A_{2},\mu_{1},\mu_{2})$ is a matched pair of commutative
associative algebras and the following conditions hold:
    \begin{equation}
        P_{2}(\mu_{1}(x)a)-\mu_{1}(P_{1}(x))a-\mu_{1}(x)P_{2}(a)=0,
    \end{equation}
\begin{equation}
    P_{1}(\mu_{2}(a)x)-\mu_{2}(P_{2}(a))x-\mu_{2}(a)P_{1}(x)=0,
\end{equation}
for all $x\in A_{1},a\in A_{2}$. Then $P_{1}+P_{2}$ is a
derivation of the resulting commutative associative algebra
$(A_{1}\oplus A_{2},\cdot)$ defined by Eq.~(\ref{eq:Asso}).
Moreover, there is a Lie bracket $[-,-]$ on $A_{1}\oplus A_{2}$
given by
\begin{eqnarray*}
    &&[x+a,y+b]=(x+a)\cdot(P_{1}(y)+P_{2}(b))-(P_{1}(x)+P_{2}(a))\cdot(y+b)\\
    &&=x\cdot_{1} P_{1}(y)+\mu_{2}(a)P_{1}(y)+\mu_{2}(P_{2}(b))x+a\cdot_{2} P_{2}(b)+\mu_{1}(x)P_{2}(b)+\mu_{1}(P_{1}(y))a\\
    &&\ \ -(P_{1}(x)\cdot y+\mu_{2}(P_{2}(a))y+\mu_{2}(b)P_{1}(x)+P_{2}(a)\cdot_{2}b+\mu_{1}(P_{1}(x))b+\mu_{1}(y)P_{2}(a) )\\
    &&=[x,y]_{1}+\rho_{2}(a)y-\rho_{2}(b)x+[a,b]_{2}+\rho_{1}(x)b-\rho_{1}(y)a,
\end{eqnarray*}
for all $x,y\in A_{1}, a,b\in A_{2}$.
Here $\rho_{1}:A_{1}\rightarrow\mathrm{End}(A_{2})$ and $\rho_{2}:A_{2}\rightarrow\mathrm{End}(A_{1})$ are linear maps  given by
\begin{equation*}
    \rho_{1}(x)a=\mu_{1}(x)P_{2}(a)-\mu_{1}(P_{1}(x))a,\;\;\rho_{2}(a)x=\mu_{2}(a)P_{1}(x)-\mu_{2}(P_{2}(a))x,\;\;\forall x\in A_{1}, a\in A_{2}.
\end{equation*}
Then by Example \ref{ex:GPA}, $(A_{1}\oplus
A_{2},\cdot,[-,-],P_{1}+P_{2})$ is a relative Poisson algebra.
Hence by Proposition \ref{pro:matched pairs},
$((A_{1},P_{1}),(A_{2},P_{2}),\mu_{1},\rho_{1},\mu_{2},\rho_{2})$
is a matched pair of relative Poisson algebras.
\end{ex}

\begin{rmk}\label{rmk:matchedpairJacobi}
We would like to point out that there is not a ``matched pair
theory" for Jacobi algebras or unital commutative  associative
algebras due to the appearance of the units. In fact, if  a unital
commutative associative algebra $(A\oplus B,1_{A\oplus B})$ is
decomposed into the direct sum of the underlying spaces of two unital
commutative associative algebras $(A,1_A)$ and $(B,1_B)$, then
there are representations $(\mu_B, A)$ and $(\mu_A,B)$ of the
commutative associative algebras $B$ and $A$ respectively such
that Eq.~(\ref{eq:Asso}) gives the commutative associative algebra
structure on $A\oplus B$ due to the matched pairs of commutative
associative algebras. Suppose that $1_{A\oplus B}=a+b$, where
$a\in A, b\in B$. Then
$$1_A=1_A\cdot 1_{A\oplus B}=1_A\cdot(a+b)=a+\mu_A(1_A)b+\mu_B(b)1_A=a+b+\mu_B(b) 1_A.$$
Therefore $b=0$ and $a=1_A$. Thus $1_{A\oplus B}=1_A$. Similarly,
$1_{A\oplus B}=1_B$. Hence $1_{A\oplus B}\in A\cap B=\lbrace
0\rbrace$, which is a contradiction.
\end{rmk}

\section{Relative Poisson bialgebras}
We introduce the notions of Manin triples of relative Poisson
algebras and relative Poisson bialgebras. The equivalence between
them is interpreted in terms of certain matched pairs of relative
Poisson algebras.

\subsection{Frobenius relative Poisson algebras and Manin triples of relative Poisson algebras}\

We generalize the notion of Frobenius Jacobi algebras (\cite{Ago})
to the following notion of Frobenius relative Poisson algebras.

\begin{defi}
A bilinear form $\mathcal{B}$ on a relative Poisson algebra
$(A, P)$ is called \textbf{invariant} if it satisfies the following
equations:
\begin{equation}\label{eq:Fro}
\mathcal{B}(x\cdot y,z)=\mathcal{B}(x,y\cdot z),
\end{equation}
\begin{equation}\label{eq:QuadLie}
\mathcal{B}([x,y],z)=\mathcal{B}(x,[y,z]),
\end{equation}
for all $x,y,z\in A$. A relative Poisson algebra $(A,P)$ is
called \textbf{Frobenius} if there is a nondegenerate invariant
  bilinear form $\mathcal{B}$ on $(A,P)$, which is denoted
 by $(A,P,\mathcal{B})$.
\end{defi}

\delete{
\textcolor{blue}{
The following proposition shows that (Frobenius) relative Poisson algebras are closed under tensor products, which generalizes the fact for Jacobi algebras in
\cite{Ago}.
\begin{pro}
    Let $(A_{1},\cdot_{1},[-,-]_{1},P_{1})$ and $(A_{2},\cdot_{2},[-,-]_{2},P_{2})$ be relative Poisson algebras. Then there is a relative Poisson algebra $(A_{1}\otimes A_{2},\cdot,[-,-],P_{1}\otimes P_{2})$, where
    \begin{eqnarray*}
        &&(x\otimes a)\cdot(y\otimes b)=x\cdot_{1}y\otimes a\cdot_{2}b,\\
        &&[x\otimes a,y\otimes b]=x\cdot_{1}y\otimes[a,b]_{2}+[x,y]_{1}\otimes a\cdot_{2}b,\\
        &&(P_{1}\otimes P_{2})x\otimes a=P_{1}(x)\otimes a+x\otimes P_{2}(a),
    \end{eqnarray*}
for all $x,y\in A_{1},a,b\in A_{2}$. Moreover, if there are bilinear forms $\mathcal{B}_{1}$ and $\mathcal{B}_{2}$ such that $(A_{1},P_{1},\mathcal{B}_{1})$ and $(A_{2},P_{2},\mathcal{B}_{2})$ are Frobenius, then $(A_{1}\otimes A_{2},P_{1}\otimes P_{2},\mathcal{B}_{1}\otimes\mathcal{B}_{2})$ is also Frobenius, where
\begin{equation*}
    \mathcal{B}_{1}\otimes\mathcal{B}_{2}(x\otimes a,y\otimes b)=\mathcal{B}_{1}(x,y)\mathcal{B}_{2}(a,b),\;\;\forall x,y\in A_{1}, a,b\in A_{2}.
\end{equation*}
\end{pro}
\begin{proof}
    It is known that $(A_{1}\otimes A_{2},\cdot)$ is a commutative associative algebra\underline{, and}\textcolor{blue}{(and)} $(A_{1}\otimes A_{2},[-,-])$ is a Lie algebra. Moreover, we have
    \begin{eqnarray*}
        (P_{1}\otimes P_{2})(x\otimes a\cdot y\otimes b)&=&(P_{1}\otimes P_{2})(x\cdot_{1}y\otimes a\cdot_{2}b)\\
        &=&P_{1}(x\cdot_{1}y)\otimes a\cdot_{2}b+x\cdot_{1}y\otimes P_{2}(a\cdot_{2}b)\\
        &=&P_{1}(x)\cdot_{1}y\otimes a\cdot_{2}b+x\cdot_{1}P_{1}(y)\otimes a\cdot_{2}b\\
        &&+x\cdot_{1}y\otimes P_{2}(a)\cdot_{2}b+x\cdot_{1}y\otimes a\cdot_{2}P_{2}(b)\\
        &=&(P_{1}\otimes P_{2})x\otimes a\cdot y\otimes b+x\otimes a\cdot(P_{1}\otimes P_{2})y\otimes b.
    \end{eqnarray*}
Thus $P_{1}\otimes P_{2}$ is a derivation of $(A_{1}\otimes A_{2},\cdot)$\underline{, and}\textcolor{blue}{(and)} similarly $P_{1}\otimes P_{2}$ is also a derivation of $(A_{1}\otimes A_{2},[-,-])$. Moreover, we have
\begin{eqnarray*}
&&[z\otimes c,(x\otimes a)\cdot(y\otimes b)]=x\cdot_{1}y\cdot_{1}z\otimes[c,a\cdot_{2}b]_{2}+[z,x\cdot_{1}y]_{1}\otimes a\cdot_{2}b\cdot_{2}c,\\
&&[z\otimes c,x\otimes a]\cdot(y\otimes b)=x\cdot_{1}y\cdot_{1}z\otimes[c,a]_{2}\cdot_{2}b+[z,x]_{1}\cdot_{1}y\otimes a\cdot_{2}b\cdot_{2}c,\\
&&(x\otimes a)\cdot[z\otimes c,y\otimes b]=x\cdot_{1}y\cdot_{1}z\otimes a\cdot_{2}[c,b]_{2}+x\cdot_{1}[z,y]_{1}\otimes a\cdot_{2}b\cdot_{2}c,\\
&&(x\otimes a)\cdot(y\otimes b)\cdot(P_{1}\otimes P_{2})(z\otimes c)=x\cdot_{1}y\cdot_{1}z\otimes a\cdot_{2}b\cdot_{2}P_{2}(c)+x\cdot_{1}y\cdot_{1}P_{1}(z)\otimes a\cdot_{2}b\cdot_{2}c.
\end{eqnarray*}
Thus
\begin{small}
    \begin{equation*}
[z\otimes c,(x\otimes a)\cdot(y\otimes b)]=[z\otimes c,x\otimes a]\cdot(y\otimes b)+(x\otimes a)\cdot[z\otimes c,y\otimes b]+(x\otimes a)\cdot(y\otimes b)\cdot(P_{1}\otimes P_{2})(z\otimes c),
    \end{equation*}
\end{small}
for all $x,y,z\in A_{1},a,b,c\in A_{2}$, that is,  $(A_{1}\otimes A_{2},\cdot,[-,-],P_{1}\otimes P_{2})$ is a relative Poisson algebra. Moreover, if $(A_{1},P_{1},\mathcal{B}_{1})$ and $(A_{2},P_{2},\mathcal{B}_{2})$ are Frobenius, we have
\begin{eqnarray*}
    \mathcal{B}_{1}\otimes\mathcal{B}_{2}((x\otimes a)\cdot(y\otimes b),z\otimes c)&=&\mathcal{B}_{1}\otimes\mathcal{B}_{2}(x\cdot_{1}y\otimes a\cdot_{2}b,z\otimes c)\\
    &=&\mathcal{B}_{1}(x\cdot_{1}y,z)\mathcal{B}_{2}(a\cdot_{2}b,c)\\
    &=&\mathcal{B}_{1}(x,y\cdot_{1}z)\mathcal{B}_{2}(a,b\cdot_{2}c)\\
    &=&\mathcal{B}_{1}\otimes\mathcal{B}_{2}( x\otimes a , (y\otimes b)\cdot(z\otimes c)),
\end{eqnarray*}
and similarly
\begin{equation*}
    \mathcal{B}_{1}\otimes\mathcal{B}_{2}([x\otimes a,y\otimes b],z\otimes c)=\mathcal{B}_{1}\otimes\mathcal{B}_{2}(x\otimes a,[y\otimes b,z\otimes c]).
\end{equation*}
Thus $(A_{1}\otimes A_{2},P_{1}\otimes P_{2},\mathcal{B}_{1}\otimes\mathcal{B}_{2})$ is also Frobenius.
\end{proof}
}}

Let $\mathcal{B}$ be a nondegenerate bilinear form on a relative Poisson algebra $(A,P)$. Then there is a unique map $\hat{P}:A\rightarrow A$ given by
\begin{equation}\label{eq:adjmap}
\mathcal{B}(P(x),y)=\mathcal{B}(x,\hat{P}(y)), \;\;\forall x,y\in A,
\end{equation}
that is, $\hat{P}$ is the adjoint linear transformation of $P$ under the nondegenerate bilinear form $\mathcal{B}$.

We have the following characterization of Frobenius relative Poisson algebras.
\begin{pro}\label{pro:adjrep}
    Let $(A,P,\mathcal{B})$ be a Frobenius relative Poisson algebra. Let $\hat{P}$ be the adjoint map of $P$ with respect to $\mathcal{B}$. Then $\hat{P}$ dually represents $(A,P)$,
    that is, $(-\mathcal{L}^{*},\mathrm{ad}^{*},\hat{P}^{*},A^{*})$ is a representation of $(A,P)$. Furthermore, as representations of $(A,P)$,
    $(-\mathcal{L}^{*},\mathrm{ad}^{*},\hat{P}^{*},A^{*})$ and $(\mathcal{L},\mathrm{ad},P,A)$ are equivalent. Conversely, let $(A,P)$ be
    a relative Poisson algebra and $Q:A\rightarrow A$ be a linear map. If $(-\mathcal{L}^{*},\mathrm{ad}^{*},{Q}^{*},A^{*})$ is a representation which is equivalent to $(\mathcal{L},\mathrm{ad},P,A)$,
    then there exists a nondegenerate bilinear form $\mathcal{B}$ such that $(A,P,\mathcal{B})$ is a Frobenius relative Poisson algebra and $\hat{Q}=P$.
\end{pro}
\begin{proof}
    Let $x,y,z,w\in A$. Since $P$ is a derivation of the Lie algebra
    $(A,[-,-])$, we have
    $$\begin{array}{ll}
        0&=\mathcal{B}(P([x,y])-[P(x),y]-[x,P(y)],z)\\
        &=\mathcal{B}([x,y],\hat{P}(z))-\mathcal{B}(P(x),[y,z])-\mathcal{B}(x,[P(y),z])\\
        &=\mathcal{B}(x,[y,\hat{P}(z)]-\hat{P}([y,z])-[P(y),z]).\\
    \end{array}$$
    Thus by the nondegeneracy of $\mathcal{B}$,
    Eq.~(\ref{eq:dualadj2}) holds. Similarly, Eq.~(\ref{eq:dualadj1})
    holds since $P$ is a derivation of the commutative associative
    algebra $(A,\cdot)$. Moreover, by Eq.~(\ref{eq:GLR}), we have
    $$\begin{array}{ll}
        0&=\mathcal{B}([x\cdot y,z]-x\cdot [y,z]-[x,z]\cdot y+x\cdot y\cdot P(z),w)\\
        &=\mathcal{B}(z,[w,x\cdot y]+[x,y\cdot w]+[y,w\cdot x]+\hat{P}(w\cdot x\cdot y)).
    \end{array}$$
    Thus Eq.~(\ref{eq:dualadj3}) holds. Hence $\hat{P}$ dually
    represents $(A,P)$.
    Define a linear map $\varphi:A\rightarrow A^{*}$ by
    \begin{equation}\label{eq:OdinPair}
        \langle \varphi(x),y\rangle=\mathcal{B}(x,y),\;\;\forall x,y\in A.
    \end{equation}
    By the nondegeneracy of $\mathcal{B}$, $\varphi$ is a linear
    isomorphism. Moreover,  we have
    $$\begin{array}{ll}
        \langle \varphi(\mathrm{ad}(x)y),z\rangle&=\langle \varphi([x,y]),z\rangle=\mathcal{B}([x,y],z)=-\mathcal{B}(y,[x,z])\\
        &=-\langle \varphi(y),[x,z]\rangle=\langle
        \mathrm{ad}^{*}(x)\varphi(y),z\rangle.
    \end{array}$$
    Thus $\varphi\mathrm{ad}(x)=\mathrm{ad}^{*}(x)\varphi$ for all
    $x\in A$. Similarly
    $\varphi\mathcal{L}(x)=-\mathcal{L}^{*}(x)\varphi$ for all $x\in
    A$. Moreover, we have
    $$\langle \varphi(P(x)),y\rangle=\mathcal{B}(P(x),y)=\mathcal{B}(x,\hat{P}(y))=\langle \varphi(x),\hat{P}(y)\rangle=\langle \hat{P}^{*}(\varphi(x)),y\rangle.$$
    Hence  $\varphi P=\hat{P}^{*}\varphi$. Therefore  as
    representations of $(A,P)$,
    $(-\mathcal{L}^{*},\mathrm{ad}^{*},\hat{P}^{*},A^{*})$ and
  $(\mathcal{L},\mathrm{ad},P,A)$ are equivalent.

    Conversely, suppose that $\varphi:A\rightarrow A^{*}$ is the
    linear isomorphism giving the equivalence between the two
    representations $(-\mathcal{L}^{*},\mathrm{ad}^{*},Q^{*},A^{*})$
    and  $(\mathcal{L},\mathrm{ad},P,A)$. Define a bilinear form
    $\mathcal{B}$ on $A$ by Eq.~(\ref{eq:OdinPair}). Then by a similar
    proof as above, $\mathcal{B}$ is a nondegenerate invariant
    bilinear form on $(A,P)$ such that $\hat{Q}=P$.
\end{proof}

\delete{
\begin{pro}\label{pro:adjrep}
Let $(A,P)$ be a relative Poisson algebra, and $\mathcal{B}$ be a nondegenerate invariant bilinear form on $(A,P)$. Let $\hat{P}$ be the adjoint map of $P$ with respect to $\mathcal{B}$. Then $\hat{P}$ dually represents $(A,P)$, that is, $(-\mathcal{L}^{*},\mathrm{ad}^{*},\hat{P}^{*},A^{*})$ is a representation of $(A,P)$. Furthermore, as representations of $(A,P)$, $(-\mathcal{L}^{*},\mathrm{ad}^{*},\hat{P}^{*},A^{*})$ is equivalent to $(\mathcal{L},\mathrm{ad},P,A)$. Conversely, let $(A,P)$ be a relative Poisson algebra, and $Q:A\rightarrow A$ be a linear map. If $(-\mathcal{L}^{*},\mathrm{ad}^{*},{Q}^{*},A^{*})$ is a representation which is equivalent to $(\mathcal{L},\mathrm{ad},P,A)$, then there exists a nondegenerate invariant bilinear form $\mathcal{B}$ on $(A,P)$ such that $\hat{Q}=P$.
\end{pro}
\begin{proof}
Let $x,y,z,w\in A$. Since $P$ is a derivation of the Lie algebra
$(A,[-,-])$, we have
$$\begin{array}{ll}
    0&=\mathcal{B}(P([x,y])-[P(x),y]-[x,P(y)],z)\\
    &=\mathcal{B}([x,y],\hat{P}(z))-\mathcal{B}(P(x),[y,z])-\mathcal{B}(x,[P(y),z])\\
    &=\mathcal{B}(x,[y,\hat{P}(z)]-\hat{P}([y,z])-[P(y),z]).\\
\end{array}$$
Thus by the nondegeneracy of $\mathcal{B}$,
Eq.~(\ref{eq:dualadj2}) holds. Similarly, Eq.~(\ref{eq:dualadj1})
holds since $P$ is a derivation of the commutative associative
algebra $(A,\cdot)$. Moreover, by Eq.~(\ref{eq:GLR}), we have
$$\begin{array}{ll}
    0&=\mathcal{B}([x\cdot y,z]-x\cdot [y,z]-[x,z]\cdot y+x\cdot y\cdot P(z),w)\\
    &=\mathcal{B}(z,[w,x\cdot y]+[x,y\cdot w]+[y,w\cdot x]+\hat{P}(w\cdot x\cdot y)).
\end{array}$$
Thus Eq.~(\ref{eq:dualadj3}) holds. Hence $\hat{P}$ dually
represents $(A,P)$.
Define a linear map $\varphi:A\rightarrow A^{*}$ by
\begin{equation}\label{eq:OdinPair}
\langle \varphi(x),y\rangle=\mathcal{B}(x,y),\forall x,y\in A.
\end{equation}
By the nondegeneracy of $\mathcal{B}$, $\varphi$ is a linear
isomorphism. Moreover,  we have
$$\begin{array}{ll}
    \langle \varphi(\mathrm{ad}(x)y),z\rangle&=\langle \varphi([x,y]),z\rangle=\mathcal{B}([x,y],z)=-\mathcal{B}(y,[x,z])\\
    &=-\langle \varphi(y),[x,z]\rangle=\langle
    \mathrm{ad}^{*}(x)\varphi(y),z\rangle.
\end{array}$$
Thus $\varphi\mathrm{ad}(x)=\mathrm{ad}^{*}(x)\varphi$ for all
$x\in A$. Similarly
$\varphi\mathcal{L}(x)=-\mathcal{L}^{*}(x)\varphi$ for all $x\in
A$. Moreover, we have
$$\langle \varphi(P(x)),y\rangle=\mathcal{B}(P(x),y)=\mathcal{B}(x,\hat{P}(y))=\langle \varphi(x),\hat{P}(y)\rangle=\langle \hat{P}^{*}(\varphi(x)),y\rangle.$$
Hence  $\varphi P=\hat{P}^{*}\varphi$. Therefore  as
representations of $(A,P)$,
$(-\mathcal{L}^{*},\mathrm{ad}^{*},\hat{P}^{*},A^{*})$ is
equivalent to $(\mathcal{L},\mathrm{ad},P,A)$.

Conversely, suppose that $\varphi:A\rightarrow A^{*}$ is the
linear isomorphism giving the equivalence between the two
representations $(-\mathcal{L}^{*},\mathrm{ad}^{*},Q^{*},A^{*})$
and  $(\mathcal{L},\mathrm{ad},P,A)$. Define a bilinear form
$\mathcal{B}$ on $A$ by Eq.~(\ref{eq:OdinPair}). Then by a similar
proof as above, $\mathcal{B}$ is a nondegenerate invariant
bilinear form on $(A,P)$ such that $\hat{Q}=P$.
\end{proof}}

\begin{defi}\label{defi:SMT}
Let $(A,P)$ be a relative Poisson algebra. Suppose that
$(A^{*},Q^{*})$ is a relative Poisson algebra. If there is a
relative Poisson algebra structure on the direct sum $A\oplus
A^*$ of vector  spaces such that $(A\oplus
A^{*},P+Q^{*},\mathcal{B}_{d})$ is a Frobenius relative Poisson
algebra, where $\mathcal{B}_{d}$ is given by
\begin{equation}\label{eq:BilinearForm}
\mathcal{B}_{d}(x+a^{*},y+b^{*})=\langle x,b^{*}\rangle+\langle a^{*},y\rangle,\;\;\forall x,y\in A, a^{*},b^{*}\in A^{*},
\end{equation}
and both $(A,P)$ and $(A^{*},Q^{*})$ are relative Poisson
subalgebras, then $((A\oplus A^{*},P+Q^{*}$, $\mathcal{B}_{d}),(A,P)$,
$(A^{*},Q^{*}))$ is called a \textbf{Manin triple of
relative Poisson algebras}. We denote it by $((A\bowtie
A^{*},P+Q^{*}$, $\mathcal{B}_{d}),(A,P),(A^{*},Q^{*}))$.
\end{defi}

The notation $A\bowtie A^*$ is justified since the relative
Poisson algebra structure on $A\oplus A^*$ comes from a matched
pair of relative Poisson algebras $(A,P)$ and $(A^*,Q^*)$ in Proposition~\ref{pro:matched pairs}.



\begin{lem}\label{lem:standard Manin triple of relative Poisson algebras}
Let $((A\bowtie A^{*},P+Q^{*},\mathcal{B}_{d}),(A,P),(A^{*},Q^{*}))$
be a  Manin triple of relative Poisson algebras.
\begin{enumerate}
\item The adjoint $\widehat{ P+Q^*}$ of $P+Q^*$  with respect to
$\mathcal{B}_{d}$ is $Q+P^{*}$. Further $Q+P^{*}$ dually represents
$(A\bowtie A^{*},P+Q^{*})$.

\item $Q$ dually represents $(A,P)$.

\item $P^{*}$ dually represents $(A^{*},Q^{*})$.
\end{enumerate}
\end{lem}
\begin{proof}
Let $x,y\in A, a^{*},b^{*}\in A^{*}$. Then we have
$$\begin{array}{ll}
    \mathcal{B}_{d}((P+Q^{*})(x+a^{*}),y+b^{*})&=\mathcal{B}_{d}(P(x)+Q^{*}(a^{*}),y+b^{*})=\langle P(x),b^{*}\rangle+\langle Q^{*}(a^{*}),y\rangle\\
    &=\langle x,P^{*}(b^{*})\rangle+\langle a^{*},Q(y)\rangle=\mathcal{B}_{d}(x+a^{*},(Q+P^{*})(y+b^{*})).
\end{array}$$
Hence the adjoint $\widehat{ P+Q^*}$ of $P+Q^*$ with respect to
$\mathcal B_{d}$ is $Q+ P^*$. Then by Proposition \ref{pro:adjrep},
$Q+P^{*}$ dually represents $(A\bowtie A^{*},P+Q^{*})$.
By 
Corollary \ref{cor:dualadj2}, it holds if and only if for all
$x,y,z\in A, a^{*},b^{*},c^{*}\in A^{*}$, {\small\begin{eqnarray*}
&&(x+a^{*})\cdot(Q+P^{*})(y+b^{*})-(P+Q^{*})(x+a^{*})\cdot(y+b^{*})-(Q+P^{*})((x+a^{*})\cdot(y+b^{*}))=0,\\
&&[x+a^{*},(Q+P^{*})(y+b^{*})]-[(P+Q^{*})(x+a^{*}),y+b^{*}]-(Q+P^{*})[(x+a^{*})\cdot(y+b^{*})]=0,\\
&&(P+Q^{*}+Q+P^{*})((x+a^{*})\cdot(y+b^{*})\cdot(z+c^{*}))=0.
\end{eqnarray*}}
Now taking $a^{*}=b^{*}=c^{*}=0$ in the above equations, we get Eqs.~(\ref{eq:dualadj1}), (\ref{eq:dualadj2}) and (\ref{eq:eqdualadj1}). Hence $Q$ dually represents $(A,P)$. Similarly, $P^{*}$ dually represents $(A^{*},Q^{*})$ by taking $x=y=z=0$.
\end{proof}


\begin{thm}\label{thm:equivalencestandard Manin triple and matched pair}
Let $(A,\cdot_{A},[-,-]_{A},P)$ be a relative Poisson algebra.
Suppose that there is a relative Poisson algebra structure
$(A^*, \cdot_{A^*}, [-,-]_{A^*}, Q^*)$ on the dual space $A^*$.
Then there is a   Manin triple of relative Poisson
algebras $((A\bowtie A^{*},P+Q^{*},\mathcal{B}_{d}),(A,P),$
$(A^{*},Q^{*}))$ if and only if
$((A,P),(A^{*},Q^{*}),-\mathcal{L}_{A}^{*},\mathrm{ad}_{A}^{*},-\mathcal{L}_{A^*}^{*},\mathrm{ad}_{A^*}^{*})$
is a matched pair of relative Poisson algebras.
\end{thm}
\begin{proof}
It is known in \cite{CP1} that there is a Lie algebra structure on the direct sum $A\oplus A^*$ of vector spaces
such that both $(A,[-,-]_{A})$ and $(A^*, [-,-]_{A^*})$ are Lie subalgebras and the bilinear form $\mathcal B_{d}$ on $A\oplus A^*$ given by
Eq.~(\ref{eq:BilinearForm}) satisfies Eq.~(\ref{eq:QuadLie}) if and only if
$(A,A^{*},\mathrm{ad}^{*}_{A},\mathrm{ad}^{*}_{A^{*}})$ is a matched pair of Lie algebras.
 Similarly, by \cite{Bai2010}, there is a commutative associative algebra structure on $A\oplus A^*$
such that both $(A,\cdot_{A})$ and $(A^*, \cdot_{A^*})$ are
commutative associative  subalgebras and
$\mathcal B_{d}$ satisfies Eq.~(\ref{eq:Fro}) if and only if
$(A,A^{*},-\mathcal{L}^{*}_{A},-\mathcal{L}^{*}_{A^{*}})$ is a matched
pair of commutative associative algebras. Hence if $((A\bowtie
A^{*},P+Q^{*},\mathcal{B}_{d}),(A,P),$ $(A^{*},Q^{*}))$ is a
Manin triple of relative Poisson algebras, then by
Proposition~\ref{pro:matched pairs} with
$$A_1=A,P_1=P,A_2=A^*,P_2=Q^*,\mu_1=-\mathcal{L}_{A}^{*},\rho_1=\mathrm{ad}_{A}^{*},\mu_2=-\mathcal{L}_{A^*}^{*},\rho_2=\mathrm{ad}_{A^*}^{*},$$
$((A,P),(A^{*},Q^{*}),-\mathcal{L}_{A}^{*},\mathrm{ad}_{A}^{*},-\mathcal{L}_{A^*}^{*},\mathrm{ad}_{A^*}^{*})$
is a matched pair of relative Poisson algebras. Conversely, if
$((A,P),(A^{*},Q^{*}),-\mathcal{L}_{A}^{*},\mathrm{ad}_{A}^{*},-\mathcal{L}_{A^*}^{*},\mathrm{ad}_{A^*}^{*})$
is a matched pair of relative Poisson algebras, then by
Proposition~\ref{pro:matched pairs} again,  there is a relative
Poisson algebra $(A\bowtie A^*,P+Q^*)$ obtained from the matched
pair with both $(A,P)$ and $(A^{*},Q^{*})$ as relative Poisson
subalgebras. Moreover, the bilinear form $\mathcal B_{d}$ is
invariant on $(A\bowtie A^*,P+Q^*)$. Hence $((A\bowtie
A^{*},P+Q^{*},\mathcal{B}_{d}),(A,P),$ $(A^{*},Q^{*}))$ is a Manin
triple of relative Poisson algebras.
\end{proof}



\subsection{Relative Poisson bialgebras}\

We recall the notions of (commutative and cocommutative) infinitesimal bialgebras (\cite{Bai2010}) and Lie bialgebras
(\cite{CP1}) before we introduce the notion of relative Poisson
bialgebras.

\begin{defi}
A \textbf{cocommutative coassociative coalgebra} is a pair
$(A,\Delta)$, such that $A$ is a vector space and
$\Delta:A\rightarrow A\otimes A$ is a linear map
satisfying
\begin{equation}\label{eq:symmetric}
\tau\Delta=\Delta,
\end{equation}
\begin{equation}\label{AssoCo}
(\mathrm{id}\otimes \Delta)\Delta=(\Delta\otimes\mathrm{id})\Delta,
\end{equation}
where $\tau: A\otimes A\rightarrow A\otimes A$ is the exchanging operator defined as $\tau(x\otimes
y)=y\otimes x$, for all $x,y\in A$.
\end{defi}


\begin{defi}
A \textbf{commutative and cocommutative infinitesimal bialgebra}
is a triple $(A,\cdot,\Delta)$ such that:


(1) $(A,\cdot)$ is a commutative associative algebra;

(2) $(A,\Delta)$ is a cocommutative coassociative coalgebra;

(3) $\Delta$ satisfies the following condition:
\begin{equation}\label{AssoBia}
\Delta(x\cdot y)=(\mathcal{L}(x)\otimes \mathrm{id})\Delta(y)+(\mathrm{id}\otimes\mathcal{L}(y))\Delta(x),\;\;\forall x,y\in A.
\end{equation}
\end{defi}

\begin{defi}\label{de:LieCo}
A \textbf{Lie coalgebra} is a pair $(A,\delta)$, such that $A$ is a vector space and $\delta:A\rightarrow A\otimes A$ is a linear map satisfying
\begin{equation}\label{eq:skewsymmetric}
\tau\delta=-\delta,
\end{equation}
\begin{equation}\label{eq:LieCo}
(\mathrm{id}+\xi+\xi^{2})(\mathrm{id}\otimes\delta)\delta=0,
\end{equation}
where $\xi:A\otimes A\otimes A\rightarrow A\otimes A\otimes A$ is the linear map
defined as $\xi(x\otimes y\otimes z)=y\otimes z\otimes x$, for all
$x,y,z\in A.$
\end{defi}

\begin{defi}\label{de:LieBia}
A \textbf{Lie bialgebra} is a triple $(A,[-,-],\delta)$, such that

(1) $(A,[-,-])$ is a Lie algebra;

(2) $(A,\delta)$ is a Lie coalgebra;

(3) $\delta$ is a 1-cocycle of $(A,[-,-])$ with values in $A\otimes A$, that is,
\begin{equation}\label{eq:LieBia}
\delta([x,y])=(\mathrm{ad}(x)\otimes \mathrm{id}+\mathrm{id}\otimes\mathrm{ad}(x))\delta(y)-(\mathrm{ad}(y)\otimes \mathrm{id}+\mathrm{id}\otimes\mathrm{ad}(y))\delta(x),\;\;\forall x,y\in A.
\end{equation}
\end{defi}

Now we give the definitions of a relative Poisson coalgebra and a relative Poisson bialgebra.

\begin{defi}
Let $A$ be a vector space,  and $\Delta,\delta:A\rightarrow A\otimes A$ and $Q:A\rightarrow A$ be linear maps. Then $(A,\Delta,\delta,Q)$ is called a \textbf{relative Poisson coalgebra} if $(A,\Delta)$ is a cocommutative coassociative coalgebra, $(A,\delta)$ is a Lie coalgebra, and the following conditions are satisfied:
\begin{equation}\label{eq:Co1}
\Delta Q=(Q\otimes \mathrm{id}+\mathrm{id}\otimes Q)\Delta,
\end{equation}
\begin{equation}\label{eq:Co2}
\delta Q=(Q\otimes \mathrm{id}+\mathrm{id}\otimes Q) \delta,
\end{equation}
\begin{equation}\label{eq:Co3}
(\mathrm{id}\otimes\Delta)\delta(x)-(\delta\otimes
\mathrm{id})\Delta(x)-(\tau\otimes
\mathrm{id})(\mathrm{id}\otimes\delta)\Delta(x)-(Q\otimes
\mathrm{id}\otimes \mathrm{id})(\Delta\otimes
\mathrm{id})\Delta(x)=0, \;\;\forall x\in A.
\end{equation}
\end{defi}

\begin{pro} Under the finite dimensional assumption,
$(A,\Delta,\delta,Q)$ is a relative Poisson coalgebra if and only if $(A^{*},\Delta^{*},\delta^{*},$ $Q^{*})$ is a relative Poisson algebra.
\end{pro}
\begin{proof}
Obviously, $(A,\Delta)$ is a cocommutative coassociative coalgebra if and only if $(A^{*},\Delta^{*})$ is a commutative associative algebra and
 $(A,\delta)$ is a Lie coalgebra if and only if $(A^{*},\delta^{*})$ is a Lie algebra.

For all $a^{*},b^{*}\in A^{*}$, set $\Delta^{*}(a^{*}\otimes
b^{*})=a^{*}\cdot b^{*}$, $\delta^{*}(a^{*}\otimes b^{*})=[a^{*},
b^{*}]$. Then we have {\small\begin{eqnarray*}
&&\langle (\mathrm{id}\otimes\Delta)\delta(x),a^{*}\otimes b^{*}\otimes c^{*}\rangle=\langle x,\delta^{*}(\mathrm{id}\otimes \Delta^{*})(a^{*}\otimes b^{*}\otimes c^{*})\rangle=\langle x,[a^{*},b^{*}\cdot c^{*}]\rangle,\\
&&\langle (\delta\otimes \mathrm{id})\Delta(x) ,a^{*}\otimes b^{*}\otimes c^{*}\rangle=\langle x, \Delta^{*}(\delta^{*}\otimes \mathrm{id})(a^{*}\otimes b^{*}\otimes c^{*})\rangle=\langle x,[a^{*},b^{*}]\cdot c^{*}\rangle,\\
&&\langle(\tau\otimes \mathrm{id})(\mathrm{id}\otimes\delta)\Delta(x), a^{*}\otimes b^{*}\otimes c^{*}\rangle=\langle x, \Delta^{*}(\mathrm{id}\otimes\delta^{*})(\tau\otimes \mathrm{id})(a^{*}\otimes b^{*}\otimes c^{*})\rangle=\langle x,b^{*}\cdot[a^{*},c^{*}]\rangle,\\
&&\langle(Q\otimes \mathrm{id}\otimes \mathrm{id})(\Delta\otimes \mathrm{id})\Delta(x), a^{*}\otimes b^{*}\otimes c^{*}\rangle=\langle x, \Delta^{*}(\Delta^{*}\otimes \mathrm{id})(Q^{*}\otimes \mathrm{id}\otimes \mathrm{id})(a^{*}\otimes b^{*}\otimes c^{*})\rangle\\
&&\mbox{}\hspace{6.3cm}=\langle x,Q^{*}(a^{*})\cdot b^{*}\cdot
c^{*}\rangle,
\end{eqnarray*}}
for all $x\in A$ and $a^{*},b^{*},c^{*}\in A^{*}$. Thus
Eq.~(\ref{eq:GLR}) holds for $(A^{*},\Delta^{*},\delta^{*},$
$Q^{*})$ as a relative Poisson algebra if and only if
Eq.~(\ref{eq:Co3}) holds. Similarly, $Q^*$ is a derivation of
$(A^{*},\Delta^{*})$ as a commutative associative algebra  if and
only if Eq.~(\ref{eq:Co1}) holds and $Q^*$ is a derivation of
$(A^{*},\delta^{*})$ as a Lie algebra if and only if
Eq.~(\ref{eq:Co2}) holds.
\end{proof}

\begin{defi}\label{de:GPBA}
A \textbf{relative Poisson bialgebra} is a collection $(A,\cdot,[-,-],\Delta,\delta,P,$ $Q)$ satisfying the following conditions:

(1) $(A,\cdot,[-,-],P)$ is a relative Poisson algebra;

(2) $(A,\Delta,\delta,Q)$ is a relative Poisson coalgebra;

(3) $\Delta$ satisfies Eq.~(\ref{AssoBia}) and hence $(A,\cdot,\Delta)$ is a commutative and cocommutative infinitesimal bialgebra;

(4)  $\delta$ satisfies Eq.~(\ref{eq:LieBia}) and hence  $(A,[-,-],\delta)$ is a Lie bialgebra;

(5) $Q$ dually represents $(A,P)$, that is, Eqs.~(\ref{eq:dualadj1}),~(\ref{eq:dualadj2}) and (\ref{eq:eqdualadj1}) hold;

(6) $P^{*}$ dually represents $(A^{*},Q^{*})$, that is, the following equations hold:
\begin{equation}\label{eq:Bi1}
\Delta P=(P\otimes \mathrm{id}-\mathrm{id}\otimes Q) \Delta,
\end{equation}
\begin{equation}\label{eq:Bi2}
\delta P=(P\otimes \mathrm{id}-\mathrm{id}\otimes Q) \delta,
\end{equation}
\begin{equation}\label{eq:Bi3}
(\Delta\otimes \mathrm{id}) \Delta (P+Q)=0;
\end{equation}

(7) The following equations hold:
\begin{eqnarray}
&&
\delta(x\cdot y)-(\mathrm{id}\otimes\mathrm{ad}(y))\Delta(x)-(\mathcal{L}(x)\otimes \mathrm{id})\delta(y)\nonumber\\
&&-(\mathrm{id}\otimes\mathrm{ad}(x))\Delta(y)-(\mathcal{L}(y)\otimes \mathrm{id})\delta(x)-(\mathrm{id}\otimes Q)\Delta(x\cdot y)=0,
\label{eq:Bi4}\\
&&\Delta([x,y])-(\mathcal{L}(y)\otimes \mathrm{id})\delta(x)-(\mathrm{id}\otimes\mathrm{ad}(x))\Delta(y)\nonumber\\
&&+(\mathrm{id}\otimes\mathcal{L}(y))\delta(x)-(\mathrm{ad}(x)\otimes \mathrm{id})\Delta(y)+\Delta(P(x)\cdot y)=0,
\label{eq:Bi5}
\end{eqnarray}
for all $x,y\in A$.
\end{defi}

\begin{rmk}
The notion of a Poisson bialgebra given in \cite{NB} is recovered
from the above notion of a relative Poisson bialgebra by
letting $P=Q=0$.
\end{rmk}

\begin{rmk}\label{rmk:dual}
        Let $(A,\cdot_{A},[-,-]_{A},\Delta_{A},\delta_{A},P,Q)$ be a
        relative Poisson bialgebra. It is straightforward to show that
         $(A^{*},\cdot_{A^{*}},[-,-]_{A^{*}},\Delta_{A^{*}},\delta_{A^{*}},Q^{*},P^{*})$
        is also a relative Poisson bialgebra, where
        $\cdot_{A^{*}},[-,-]_{A^{*}}:A^{*}\otimes A^{*}\rightarrow A^{*}$
        and $\Delta_{A^{*}},\delta_{A^{*}}:A^{*}\rightarrow A^{*}\otimes
        A^{*}$ are respectively given by
        \begin{equation}\label{eq:operations on dual space1}
            a^{*}\cdot_{A^{*}}b^{*}=\Delta^{*}_{A}(a^{*}\otimes b^{*}),\;\;
            [a^{*},b^{*}]_{A^{*}}=\delta^{*}_{A}(a^{*}\otimes b^{*}),
        \end{equation}
        \begin{equation}\label{eq:operations on dual space2}
            \langle \Delta_{A^{*}}(a^{*}),x\otimes y\rangle=-\langle
            a^{*},x\cdot_{A} y\rangle,\;\;\langle \delta_{A^{*}}(a^{*}),x\otimes
            y\rangle=-\langle a^{*},[x,y]_{A}\rangle,
        \end{equation}
        for all $x,y\in A, a^{*},b^{*}\in A^{*}$.
\end{rmk}

\begin{thm} \label{thm:bia-match}
Let $(A,\cdot_{A},[-,-]_{A},P)$ be a relative Poisson algebra.
Suppose that there is a relative Poisson algebra structure
$(A^*, \cdot_{A^*}, [-,-]_{A^*}, Q^*)$ on the dual space $A^*$ which is given by
a relative Poisson coalgebra  $(A,\Delta,\delta,Q)$.
Then  $(A,\cdot_A,[-,-]_A,\Delta,\delta,P,Q)$ is a relative Poisson bialgebra if and only if
$((A,P),(A^{*},Q^{*}),-\mathcal{L}_{A}^{*}$, $\mathrm{ad}_{A}^{*},-\mathcal{L}_{A^*}^{*},\mathrm{ad}_{A^*}^{*})$
is a matched pair of relative Poisson algebras.
\end{thm}

\begin{proof}
By \cite{CP1}, $\delta$ satisfies Eq.~(\ref{eq:LieBia}) if and
only if $(A,A^{*},\mathrm{ad}^{*}_{A},\mathrm{ad}^{*}_{A^{*}})$ is
a matched pair of Lie algebras and by \cite{Bai2010}, $\Delta$
satisfies Eq.~(\ref{AssoBia}) if and only if
$(A,A^{*},-\mathcal{L}^{*}_{A},-\mathcal{L}^{*}_{A^{*}})$ is a
matched pair of commutative associative algebras. Moreover, by
Definition~\ref{defi:dual},   $Q$ dually represents $(A,P)$ if and
only if $(-\mathcal{L}_{A}^{*}$, $\mathrm{ad}_{A}^{*}, Q^*,A^*)$
is a representation of $(A,P)$, and $P^{*}$ dually represents
$(A^{*},Q^{*})$ if and only if
$(-\mathcal{L}_{A^*}^{*},\mathrm{ad}_{A^*}^{*}, P,A)$ is a
representation of $(A^*,Q^*)$. Next we show that
\begin{center}
Eq.~(\ref{eq:MP1})$\Longleftrightarrow$ Eq.~(\ref{eq:Bi4})$\Longleftrightarrow$ Eq.~(\ref{eq:MP4}) and Eq.~(\ref{eq:MP2})$\Longleftrightarrow$ Eq.~(\ref{eq:Bi5})$\Longleftrightarrow$ Eq.~(\ref{eq:MP3})
\end{center}
in the case that $P_{1}=P,P_{2}=Q^{*}$,
$\mu_{1}=-\mathcal{L}^{*}_{A},\mu_{2}=-\mathcal{L}^{*}_{A^{*}},\rho_{1}=\mathrm{ad}^{*}_{A},\rho_{2}=\mathrm{ad}^{*}_{A^{*}}$.
As an example, we give an explicit proof for the fact that
Eq.~(\ref{eq:MP1})$\Longleftrightarrow$ Eq.~(\ref{eq:Bi4}). The
proof for the other equivalences is similar. Let $x,y\in A,
a^{*},b^{*}\in A^{*}$. Then we have
\begin{eqnarray*}
\langle \mathrm{ad}^{*}_{A^{*}}(a^{*})(x\cdot_{A} y),b^{*}\rangle&=&\langle x\cdot_{A} y,[b^{*},a^{*}]_{A^{*}}\rangle
=\langle \delta(x\cdot_{A} y),b^{*}\otimes a^{*}\rangle,\\
-\langle \mathcal{L}^{*}_{A^{*}}(\mathrm{ad}^{*}_{A}(y)a^{*})x,b^{*}\rangle&=&\langle x,b^{*}\cdot_{A^{*}}(\mathrm{ad}^{*}_{A}(y)a^{*})\rangle=\langle x,\Delta^{*}(\mathrm{id}\otimes\mathrm{ad}^{*}_{A}(y))(b^{*}\otimes a^{*})\rangle\\
 &=&-\langle (\mathrm{id}\otimes\mathrm{ad}_{A}(y))\Delta(x),b^{*}\otimes a^{*}\rangle,\\
 -\langle x\cdot\mathrm{ad}^{*}_{A^{*}}(a^{*})y,b^{*}\rangle&=&\langle \mathrm{ad}^{*}_{A^{*}}(a^{*})y,\mathcal{L}^{*}_{A}(x)b^{*}\rangle
 =\langle y,[\mathcal{L}^{*}_{A}(x)b^{*},a^{*}]_{A^{*}}\rangle\\
 &=&\langle y,\delta^{*}(\mathcal{L}^{*}_{A}(x)\otimes \mathrm{id})(b^{*}\otimes a^{*})\rangle=-\langle (\mathcal{L}_{A}(x)\otimes \mathrm{id})\delta(y),b^{*}\otimes a^{*}\rangle,\\
\langle \mathcal{L}^{*}_{A^{*}}(Q^{*}(a^{*}))(x\cdot_{A} y),b^{*}\rangle&=&-\langle x\cdot_{A} y,b^{*}\cdot_{A^{*}} Q^{*}(a^{*})\rangle=-\langle x\cdot_{A} y,\Delta^{*}(\mathrm{id}\otimes Q^{*})(b^{*}\otimes a^{*})\rangle\\
    &=&-\langle (\mathrm{id}\otimes Q)\Delta(x\cdot_{A} y),b^{*}\otimes a^{*}\rangle.
 \end{eqnarray*}
Thus Eq.~(\ref{eq:MP1}) holds if and only if  Eq.~(\ref{eq:Bi4}) holds. Therefore the conclusion holds.
\end{proof}

Combining Theorems~\ref{thm:equivalencestandard Manin triple and matched pair} and~\ref{thm:bia-match}, we have the following conclusion.

\begin{cor}\label{thm:equivalence3}
Let $(A,\cdot_{A},[-,-]_{A},P)$ be a relative Poisson algebra.
Suppose that there is a relative Poisson algebra structure
$(A^*, \cdot_{A^*}, [-,-]_{A^*}, Q^*)$ on the dual space $A^*$ which is given by
a relative Poisson coalgebra  $(A,\Delta,\delta,Q)$.
Then the following conditions are equivalent:
\begin{enumerate}
\item There is a   Manin triple of relative Poisson
algebras $((A\bowtie A^{*},P+Q^{*},\mathcal{B}_{d}),(A$, $P)$,
$(A^{*},Q^{*}))$;

\item  $((A,P),(A^{*},Q^{*}),-\mathcal{L}_{A}^{*},\mathrm{ad}_{A}^{*},-\mathcal{L}_{A^{*}}^{*},\mathrm{ad}_{A^{*}}^{*})$ is a matched pair of relative Poisson algebras;

\item  $(A,\cdot_A,[-,-]_A,\Delta,\delta,P,Q)$ is a relative Poisson bialgebra.
\end{enumerate}
\end{cor}

\section{Coboundary relative Poisson bialgebras}

We study the coboundary relative Poisson bialgebras, which lead to
the introduction of the relative Poisson Yang-Baxter equation
(RPYBE) in a relative Poisson algebra. In particular, an
antisymmetric solution of the RPYBE in a relative Poisson algebra
gives a coboundary relative Poisson bialgebra. We also introduce
the notion of $\mathcal{O}$-operators of relative Poisson algebras
to interpret the RPYBE and an $\mathcal O$-operator gives an
antisymmetric solution of the RPYBE in a semi-direct product
relative Poisson algebra. Finally, the notion of relative
pre-Poisson algebras is introduced to construct $\mathcal
O$-operators of their sub-adjacent relative Poisson algebras.


\subsection{Coboundary relative Poisson bialgebras}\

Recall that  a commutative and cocommutative infinitesimal
bialgebra $(A,\cdot,\Delta)$ is called \textbf{co-boundary} if
there exists an $r\in A\otimes A$ such that
\begin{equation}\label{AssoCob}
\Delta(x)=(\mathrm{id}\otimes\mathcal{L}(x)-\mathcal{L}(x)\otimes \mathrm{id})r,\;\;\forall x\in A.
\end{equation}
A Lie bialgebra $(A,[-,-],\delta)$ is called \textbf{coboundary }if there exists an $r\in A\otimes A$ such that
\begin{equation}\label{eq:LieCob}
\delta(x)=(\mathrm{ad}(x)\otimes \mathrm{id}+\mathrm{id}\otimes\mathrm{ad}(x))r, \;\;\forall x\in A.
\end{equation}

Therefore they motivate us to give the following notion.
\begin{defi}
A relative Poisson bialgebra $(A,\cdot,[-,-],\Delta,\delta,P,Q)$ is called \textbf{coboundary} if there exists an $r\in A\otimes A$ such that Eqs.~(\ref{AssoCob}) and (\ref{eq:LieCob}) hold.
\end{defi}


Let $A$ be a vector space with a  bilinear operation
$\diamond:A\otimes A\rightarrow A$. Let $r=\sum
\limits_{i}a_{i}\otimes
b_{i}\in A\otimes A$. Set
$$r_{12}\diamond r_{13}=\sum_{i,j}a_{i}\diamond a_{j}\otimes b_{i}\otimes b_{j},\;\; r_{12}\diamond r_{23}=\sum_{i,j}a_{i}\otimes b_{i}\diamond a_{j}\otimes b_{j},\;\; r_{13}\diamond r_{23}=\sum_{i,j}a_{i}\otimes  a_{j}\otimes b_{i}\diamond b_{j}.$$

Let $(A,\cdot)$ be a commutative associative algebra and
$\Delta:A\rightarrow A\otimes A$ be a linear map defined by
Eq.~(\ref{AssoCob}). Then $\Delta$ satisfies Eq.~(\ref{AssoBia})
automatically. Moreover, by \cite{Bai2010}, $\Delta$ makes
$(A,\Delta)$ into a cocommutative coassociative coalgebra such
that $(A,\cdot,\Delta)$ is a commutative and cocommutative
infinitesimal bialgebra if and only if for all $x\in A,$
 \begin{equation}\label{eq:AYBE1}
 (\mathrm{id}\otimes\mathcal{L}(x)-\mathcal{L}(x)\otimes \mathrm{id})(r+\tau(r))=0,
 \end{equation}
  \begin{equation}\label{eq:AYBE2}
 (\mathrm{id}\otimes \mathrm{id}\otimes\mathcal{L}(x)-\mathcal{L}(x)\otimes \mathrm{id}\otimes \mathrm{id})\textbf{A}(r)=0,
 \end{equation}
where
\begin{equation}\label{eq:AYBE}
\textbf{A}(r)=r_{12}\cdot r_{13}-r_{12}\cdot r_{23}+r_{13}\cdot
r_{23}.
\end{equation}
The equation $\textbf{A}(r)=0$ is called \textbf{associative
Yang-Baxter equation (AYBE)} in $(A,\cdot)$.

Let $(A,[-,-])$ be a Lie algebra and $\delta:A\rightarrow
A\otimes A$ be a linear map defined by Eq.~(\ref{eq:LieCob}). Then
$\delta$ satisfies Eq.~(\ref{eq:LieBia}) automatically. Moreover,
by \cite{CP1}, $\delta$ makes $(A,\delta)$ into a Lie coalgebra
such that $(A,[-,-],\delta)$ is a Lie bialgebra if and only if for
all $x\in A,$
\begin{equation}\label{eq:CYBE1}
(\mathrm{ad}(x)\otimes \mathrm{id}+\mathrm{id}\otimes\mathrm{ad}(x))(r+\tau(r))=0,
\end{equation}
\begin{equation}\label{eq:CYBE2}
(\mathrm{ad}(x)\otimes \mathrm{id}\otimes \mathrm{id}+\mathrm{id}\otimes\mathrm{ad}(x)\otimes \mathrm{id}+\mathrm{id}\otimes \mathrm{id}\otimes\mathrm{ad}(x))\textbf{C}(r)=0,
\end{equation}
where
\begin{equation}\label{eq:CYBE}
\textbf{C}(r)=[r_{12},r_{13}]+[r_{12},r_{23}]+[r_{13},r_{23}].
\end{equation}
The equation $\textbf{C}(r)=0$ is called \textbf{classical
Yang-Baxter equation (CYBE)} in $(A,[-,-])$.

\begin{pro}\label{pro:RPYBE}
Let $(A,\cdot,[-,-],P)$ be a relative Poisson algebra which is
dually represented by $Q$. Let $r=\sum\limits_{i}a_{i}\otimes
b_{i}\in A\otimes A$ and $\Delta,\delta:A\rightarrow A\otimes A$
be linear maps defined by Eqs.~(\ref{AssoCob}) and
(\ref{eq:LieCob}) respectively. Define $\textbf{A}(r)$ and
$\textbf{C}(r)$
by Eqs.~(\ref{eq:AYBE}) and (\ref{eq:CYBE}) respectively. 
\begin{enumerate}
\item Eq.~(\ref{eq:Co1}) holds if and only if the following equation holds:
\begin{equation}\label{eq:tr1}
(\mathrm{id}\otimes\mathcal{L}(x))(\mathrm{id}\otimes P-Q\otimes \mathrm{id})r+(\mathcal{L}(x)\otimes \mathrm{id})(\mathrm{id}\otimes Q-P\otimes \mathrm{id})r=0,\;\;\forall x\in A.
\end{equation}

\item Eq.~(\ref{eq:Co2}) holds if and only if the following equation holds:
\begin{equation}\label{eq:tr2}
(\mathrm{id}\otimes\mathrm{ad}(x))(\mathrm{id}\otimes P-Q\otimes \mathrm{id})r-(\mathrm{ad}(x)\otimes \mathrm{id})(\mathrm{id}\otimes Q-P\otimes \mathrm{id})r=0,\;\;\forall x\in A.
\end{equation}

\item Eq.~(\ref{eq:Co3}) holds if and only if the following
equation holds: {\small
\begin{eqnarray}
&&(\mathrm{ad}(x)\otimes \mathrm{id}\otimes \mathrm{id}+Q\otimes
\mathrm{id}\otimes\mathcal{L}(x))\textbf{A}(r)+(\mathrm{id}\otimes \mathrm{id}\otimes\mathcal{L}(x)-\mathrm{id}\otimes\mathcal{L}(x)\otimes \mathrm{id})\textbf{C}(r)\nonumber\\
&&\hspace{0.0cm}+\sum_{j}(\mathrm{ad}(a_{j})\otimes \mathrm{id})(\mathcal{L}(x)\otimes \mathrm{id}-\mathrm{id}\otimes\mathcal{L}(x))(r+\tau(r))\otimes b_{j}+\sum_{j}(\mathrm{id}\otimes\mathcal{L}(x\cdot a_{j}))(Q\otimes \mathrm{id}-\mathrm{id}\otimes P)r\otimes b_{j}\nonumber\\
&&\hspace{0.0cm}+\sum_{j}(\mathrm{id}\otimes
\mathrm{id}\otimes\mathcal{L}(x\cdot
b_{j}))(\mathrm{id}\otimes\tau)((\mathrm{id}\otimes P-Q\otimes
\mathrm{id})r\otimes a_{j})=0,\;\;\forall x\in A.\label{eq:tr3}
\end{eqnarray}
}

\item Eq.~(\ref{eq:Bi1}) holds if and only if the following equation holds:
\begin{equation}\label{eq:tr4}
(\mathrm{id}\otimes\mathcal{L}(x)-\mathcal{L}(x)\otimes \mathrm{id})(\mathrm{id}\otimes Q-P\otimes \mathrm{id})r=0,\;\;\forall x\in A.
\end{equation}

\item Eq.~(\ref{eq:Bi2}) holds if and only if the following equation holds:
\begin{equation}\label{eq:tr5}
(\mathrm{ad}(x)\otimes \mathrm{id}+\mathrm{id}\otimes\mathrm{ad}(x))(\mathrm{id}\otimes Q-P\otimes \mathrm{id})r=0,\;\;\forall x\in A.
\end{equation}

\item Eq.~(\ref{eq:Bi3}) holds if and only if the following equation holds:
\begin{equation}\label{eq:tr6}
(\mathrm{id}\otimes \mathrm{id}\otimes\mathcal{L}((P+Q)x))\textbf{A}(r)=0,\;\;\forall x\in A.
\end{equation}

\item Eq.~(\ref{eq:Bi4}) holds if and only if the following equation holds:
\begin{equation}\label{eq:tr7}
(\mathcal{L}(x\cdot y)\otimes \mathrm{id})(\mathrm{id}\otimes Q-P\otimes \mathrm{id})r=0,\;\;\forall x,y\in A.
\end{equation}

\item Eq.~(\ref{eq:Bi5}) holds automatically.
\end{enumerate}
\end{pro}

\begin{proof} Let $x,y\in A$. (1) Substituting Eq.~(\ref{AssoCob})
into Eq.~(\ref{eq:tr1}), we have
\begin{small}
\begin{eqnarray*}
&&\Delta Q(x)-(Q\otimes \mathrm{id}+\mathrm{id}\otimes Q)\Delta(x)\\
    &&=\sum_{i}(a_{i}\otimes Q(x)\cdot b_{i}-Q(x)\cdot a_{i}\otimes b_{i}-Q(a_{i})\otimes x\cdot b_{i}
    +Q(x\cdot a_{i})\otimes b_{i}
    -a_{i}\otimes Q(x\cdot b_{i})+x\cdot a_{i}\otimes  Q(b_{i}))\\
    &&\stackrel{(\ref{eq:dualadj1})}{=}\sum_{i}(a_{i}\otimes P(b_{i})\cdot x-P(a_{i})\cdot x\otimes b_{i}-Q(a_{i})\otimes x\cdot b_{i}+x\cdot a_{i}\otimes Q(b_{i}))\\
    &&=(\mathrm{id}\otimes\mathcal{L}(x))(\mathrm{id}\otimes P-Q\otimes \mathrm{id})r+(\mathcal{L}(x)\otimes \mathrm{id})(\mathrm{id}\otimes Q-P\otimes \mathrm{id})r.
\end{eqnarray*}
\end{small}
Thus Eq.~(\ref{eq:Co1}) holds if and only if Eq.~(\ref{eq:tr1})
holds.

(2) By a similar proof as the one of (1), Eq.~(\ref{eq:Co2}) holds if and
only if Eq.~(\ref{eq:tr1}) holds.

(3) Substituting Eqs.~(\ref{AssoCob}) and~(\ref{eq:LieCob}) into
Eq.~(\ref{eq:Co3}), we have {\small \begin{eqnarray*}
    &&(\mathrm{id}\otimes\Delta)\delta(x)-(\delta\otimes \mathrm{id})\Delta(x)-(\tau\otimes \mathrm{id})(\mathrm{id}\otimes\delta)\Delta(x)-(Q\otimes \mathrm{id}\otimes \mathrm{id})(\Delta\otimes \mathrm{id})\Delta(x)\\
    &&=\sum_{i,j}([x,a_{i}]\otimes a_{j}\otimes b_{i}\cdot b_{j}
    -[x,a_{i}]\otimes b_{i}\cdot a_{j}\otimes b_{j}
    +a_{i}\otimes a_{j}\otimes [x,b_{i}]\cdot b_{j}-a_{i}\otimes [x,b_{i}]\cdot a_{j}\otimes
    b_{j}\\
    &&\hspace{0.2cm}-[a_{i},a_{j}]\otimes b_{j}\otimes x\cdot b_{i}
    -a_{j}\otimes [a_{i},b_{j}]\otimes x\cdot b_{i}
    +[x\cdot a_{i},a_{j}]\otimes b_{j}\otimes b_{i}
    +a_{j}\otimes [x\cdot a_{i},b_{j}]\otimes b_{i}\\
    &&\hspace{0.2cm}-[x\cdot b_{i},a_{j}]\otimes a_{i}\otimes b_{j}-a_{j}\otimes a_{i}\otimes [x\cdot b_{i},b_{j}]
    +[b_{i},a_{j}]\otimes x\cdot a_{i}\otimes b_{j}
    +a_{j}\otimes x\cdot a_{i}\otimes [b_{i},b_{j}]\\
    &&\hspace{0.2cm}-Q(a_{j})\otimes a_{i}\cdot b_{j}\otimes x\cdot b_{i}+Q(a_{i}\cdot a_{j})\otimes b_{j}\otimes x\cdot
b_{i}+Q(a_{j})\otimes x\cdot a_{i}\cdot b_{j}\otimes b_{i}\\
    &&\hspace{0.2cm}-Q(x\cdot a_{i}\cdot a_{j})\otimes b_{j}\otimes
    b_{i})\\
    &&=W(1)+W(2)+W(3),
\end{eqnarray*}
} where {\small \begin{eqnarray*} W(1):&=&
\sum_{i,j}([x,a_{i}]\otimes a_{j}\otimes b_{i}\cdot b_{j}
    -[x,a_{i}]\otimes b_{i}\cdot a_{j}\otimes b_{j}+[x\cdot a_{i},a_{j}]\otimes b_{j}\otimes b_{i}\\
    &\mbox{}& -[x\cdot b_{i},a_{j}]\otimes a_{i}\otimes b_{j}-Q(x\cdot a_{i}\cdot a_{j})\otimes b_{j}\otimes b_{i})\\
&\stackrel{(\ref{eq:AYBE})}{=}&(\mathrm{ad}(x)\otimes \mathrm{id}\otimes \mathrm{id})\textbf{A}(r)+\sum_{i,j}(-[x,a_{i}\cdot a_{j}]\otimes b_{i}\otimes b_{j}+[x\cdot a_{i},a_{j}]\otimes b_{j}\otimes b_{i}\\
&\mbox{}&-[x\cdot b_{i},a_{j}]\otimes a_{i}\otimes b_{j}-Q(x\cdot a_{i}\cdot a_{j})\otimes b_{j}\otimes b_{i})\\
&\stackrel{(\ref{eq:dualadj3})}{=}&(\mathrm{ad}(x)\otimes \mathrm{id}\otimes \mathrm{id})\textbf{A}(r)+\sum_{i,j}([a_{j},x\cdot a_{i}]\otimes b_{i}\otimes b_{j}+[a_{j},x\cdot b_{i}]\otimes a_{i}\otimes b_{j})\\
&=&(\mathrm{ad}(x)\otimes \mathrm{id}\otimes
\mathrm{id})\textbf{A}(r)+\sum_{j}(\mathrm{ad}(a_{j})\otimes
\mathrm{id})(\mathcal{L}(x)\otimes \mathrm{id})(r+\tau(r))\otimes
b_{j},
\end{eqnarray*}
\begin{eqnarray*}
W(2): &=&\sum_{i,j}(-a_{i}\otimes [x,b_{i}]\cdot a_{j}\otimes
b_{j}+a_{j}\otimes [x\cdot a_{i},b_{j}]\otimes
b_{i}+[b_{i},a_{j}]\otimes x\cdot a_{i}\otimes b_{j}\\
  &\mbox{}&  +a_{j}\otimes x\cdot a_{i}\otimes [b_{i},b_{j}]+Q(a_{j})\otimes x\cdot a_{i}\cdot b_{j}\otimes b_{i})\\
&=&\sum_{i,j}(-a_{i}\otimes[x,b_{i}]\cdot a_{j}\otimes b_{j}+a_{j}\otimes[x\cdot a_{i},b_{j}]\otimes b_{i}+a_{j}\otimes x\cdot a_{i}\otimes[b_{i},b_{j}]\\
&\mbox{}&+[a_{j},a_{i}]\otimes x\cdot b_{i}\otimes b_{j}+Q(a_{j})\otimes x\cdot a_{i}\cdot b_{j}\otimes b_{i})-\sum_{j}(\mathrm{ad}(a_{j})\otimes\mathcal{L}(x))(r+\tau(r))\otimes b_{j}\\
&\stackrel{(\ref{eq:CYBE})}{=}&-(\mathrm{id}\otimes\mathcal{L}(x)\otimes
\mathrm{id})\textbf{C}(r)+ \sum_{i,j}(-a_{i}\otimes[x,b_{i}]\cdot
a_{j}\otimes b_{j}+a_{i}\otimes [x\cdot a_{j},b_{i}]\otimes
b_{j}\\
&\mbox{}&+a_{i}\otimes x\cdot[b_{i},a_{j}]\otimes
b_{j}+Q(a_{j})\otimes x\cdot a_{i}\cdot b_{j}\otimes b_{i})
-\sum_{j}(\mathrm{ad}(a_{j})\otimes\mathcal{L}(x))(r+\tau(r))\otimes b_{j}\\
&\stackrel{(\ref{eq:GLR})}{=}&-(\mathrm{id}\otimes\mathcal{L}(x)\otimes
\mathrm{id})\textbf{C}(r)+ \sum_{i,j}(-a_{i}\otimes x\cdot
a_{j}\cdot P(b_{i})\otimes
b_{j}+Q(a_{j})\otimes x\cdot a_{i}\cdot b_{j}\otimes b_{i})\\
&\mbox{}&-\sum_{j}(\mathrm{ad}(a_{j})\otimes\mathcal{L}(x))(r+\tau(r))\otimes
b_{j}\\
&=&-(\mathrm{id}\otimes\mathcal{L}(x)\otimes
\mathrm{id})\textbf{C}(r)+\sum_{j}(\mathrm{id}\otimes\mathcal{L}(x\cdot
a_{j})\otimes \mathrm{id})((Q\otimes
\mathrm{id}-\mathrm{id}\otimes P)r\otimes
b_{j})\\
&\mbox{}&-\sum_{j}(\mathrm{ad}(a_{j})\otimes\mathcal{L}(x))(r+\tau(r))\otimes
b_{j},
\end{eqnarray*}
\begin{eqnarray*}
W(3):&=&\sum_{i,j}(a_{i}\otimes a_{j}\otimes [x,b_{i}]\cdot
b_{j}-[a_{i},a_{j}]\otimes b_{j}\otimes x\cdot b_{i}-a_{j}\otimes
[a_{i},b_{j}]\otimes x\cdot b_{i}\\
    &&-a_{j}\otimes a_{i}\otimes [x\cdot b_{i},b_{j}]-Q(a_{j})\otimes a_{i}\cdot b_{j}\otimes x\cdot b_{i}
    +Q(a_{i}\cdot a_{j})\otimes b_{j}\otimes x\cdot b_{i})\\
&\stackrel{(\ref{eq:CYBE})}{=}&(\mathrm{id}\otimes\mathrm{id}\otimes\mathcal{L}(x))\textbf{C}(r)
+\sum_{i,j}(-a_{i}\otimes a_{j}\otimes x\cdot[b_{i},b_{j}]+a_{i}\otimes a_{j}\otimes[x,b_{i}]\cdot b_{j}\\
&&-a_{i}\otimes a_{j}\otimes[x\cdot b_{j},b_{i}]-Q(a_{j})\otimes
a_{i}\cdot b_{j}\otimes x\cdot b_{i}
    +Q(a_{i}\cdot a_{j})\otimes b_{j}\otimes x\cdot b_{i})\\
&\stackrel{(\ref{eq:GLR})}{=}&(\mathrm{id}\otimes\mathrm{id}\otimes\mathcal{L}(x))\textbf{C}(r)+\sum_{i,j}(a_{i}\otimes
a_{j}\otimes x\cdot b_{j}\cdot P(b_{i})-Q(a_{j})\otimes a_{i}\cdot
b_{j}\otimes x\cdot b_{i}\\
&&    +Q(a_{i}\cdot a_{j})\otimes b_{j}\otimes x\cdot b_{i})\\
&\stackrel{(\ref{eq:AYBE})}{=}&
(\mathrm{id}\otimes\mathrm{id}\otimes\mathcal{L}(x))\textbf{C}(r)+\sum_{i,j}(a_{i}\otimes
a_{j}\otimes x\cdot b_{j}\cdot P(b_{i})-Q(a_{i})\otimes
a_{j}\otimes x\cdot b_{i}\cdot b_{j})\\
&&+ (Q\otimes \mathrm{id}\otimes\mathcal{L}(x))\textbf{A}(r)\\
&=&(\mathrm{id}\otimes\mathrm{id}\otimes\mathcal{L}(x))\textbf{C}(r)+(Q\otimes
\mathrm{id}\otimes\mathcal{L}(x))\textbf{A}(r)
\\
&&+\sum_{j}(\mathrm{id}\otimes
\mathrm{id}\otimes\mathcal{L}(x\cdot
b_{j}))(\mathrm{id}\otimes\tau)((\mathrm{id}\otimes P-Q\otimes
\mathrm{id})r\otimes a_{j}).
\end{eqnarray*}
} Thus Eq.~(\ref{eq:Co3}) holds if and only if  Eq.~(\ref{eq:tr3})
holds.

(4) By a similar proof as the one of (1), Eq.~(\ref{eq:Bi1}) holds if and
only if Eq.~(\ref{eq:tr4}) holds.

(5) By a similar proof as of the one (1), Eq.~(\ref{eq:Bi2}) holds if and
only if Eq.~(\ref{eq:tr5}) holds.

(6) Substituting Eq.~(\ref{AssoCob}) into  Eq.~(\ref{eq:Bi3}), we have
\begin{small}
\begin{eqnarray*}
    &&(\Delta\otimes \mathrm{id})\Delta(P+Q)(x)\\
    &&=\sum_{i,j}(a_{j}\otimes a_{i}\cdot b_{j}\otimes ((P+Q)x)\cdot b_{i}-a_{i}\cdot a_{j}\otimes b_{j}\otimes ((P+Q)x)\cdot b_{i}\\
    &&\hspace{1cm}+(((P+Q)x)\cdot a_i)\cdot a_j\otimes b_j\otimes b_i-a_j\otimes (((P+Q)x)\cdot a_i)\cdot b_j\otimes b_i) \\
    &&\stackrel{(\ref{eq:eqdualadj2})}{=}\sum_{i,j}(a_{j}\otimes a_{i}\cdot b_{j}\otimes ((P+Q)x)\cdot b_{i}-a_{i}\cdot a_{j}\otimes b_{j}\otimes ((P+Q)x)\cdot b_{i}\\
    &&\hspace{1cm}-a_{i}\otimes a_{j}\otimes((P+Q)x)\cdot b_{i}\cdot b_{j})\\
    &&=-(\mathrm{id}\otimes \mathrm{id}\otimes\mathcal{L}((P+Q)x))\textbf{A}(r).
\end{eqnarray*}
\end{small}
Thus Eq.~(\ref{eq:Bi3}) holds if and only if Eq.~(\ref{eq:tr6}) holds.

(7) Substituting Eqs.~(\ref{AssoCob}) and~(\ref{eq:LieCob}) into
Eq.~(\ref{eq:Bi4}),
we have {\small \begin{eqnarray*}
    &&\delta(x\cdot y)-(\mathrm{id}\otimes\mathrm{ad}(y))\Delta(x)-(\mathcal{L}(x)\otimes \mathrm{id})\delta(y)-(\mathrm{id}\otimes\mathrm{ad}(x))\Delta(y)-(\mathcal{L}(y)\otimes \mathrm{id})\delta(x)\\
    &&\hspace{1cm}-(\mathrm{id}\otimes Q)\Delta(x\cdot y)\\
    &&=\sum_{i}([x\cdot y,a_{i}]\otimes b_{i}+a_{i}\otimes [x\cdot y,b_{i}]-a_{i}\otimes [y,x\cdot b_{i}]+x\cdot a_{i}\otimes [y,b_{i}]\\
    &&\hspace{1cm}-a_{i}\otimes [x,y\cdot b_{i}]+y\cdot a_{i}\otimes [x,b_{i}]-x\cdot [y,a_{i}]\otimes b_{i}-x\cdot a_{i}\otimes [y,b_{i}]\\
    &&\hspace{1cm}-y\cdot [x,a_{i}]\otimes b_{i}-y\cdot a_{i}\otimes [x,b_{i}]
    -a_{i}\otimes Q(x\cdot y\cdot b_{i})+x\cdot y\cdot a_{i}\otimes Q(b_{i}))\\
    &&=\sum_{i}(a_{i}\otimes [x\cdot y,b_{i}]-a_{i}\otimes [y,x\cdot b_{i}]
-a_{i}\otimes [x,y\cdot b_{i}]-a_{i}\otimes Q(x\cdot y\cdot b_{i})) \\
&&\hspace{1cm}+
    \sum_{i}([x\cdot y,a_{i}]\otimes b_{i}-x\cdot [y,a_{i}]\otimes b_{i}
-y\cdot [x,a_{i}]\otimes b_{i}+x\cdot y\cdot a_{i}\otimes Q(b_{i}))\\
&&\stackrel{(\ref{eq:dualadj3})}{=}\sum_{i}([x\cdot y,a_{i}]\otimes b_{i}-x\cdot [y,a_{i}]\otimes b_{i}
-y\cdot [x,a_{i}]\otimes b_{i}+x\cdot y\cdot a_{i}\otimes Q(b_{i}))\\
&&\stackrel{(\ref{eq:GLR})}{=}\sum_{i}(-x\cdot y\cdot P(a_{i})\otimes b_{i}+x\cdot y\cdot a_{i}\otimes Q(b_{i}))\\
&&=(\mathcal{L}(x\cdot y)\otimes \mathrm{id})(\mathrm{id}\otimes Q-P\otimes \mathrm{id})r.
\end{eqnarray*}}
Thus Eq.~(\ref{eq:Bi4}) holds if and only if Eq.~(\ref{eq:tr7}) holds.

(8) Substituting Eqs.~(\ref{AssoCob}) and~(\ref{eq:LieCob}) into
Eq.~(\ref{eq:Bi5}),
we have {\small \begin{eqnarray*}
    &&\Delta([x,y])-(\mathcal{L}(y)\otimes \mathrm{id})\delta(x)
    +(\mathrm{id}\otimes\mathcal{L}(y))\delta(x)-(\mathrm{id}\otimes\mathrm{ad}(x))\Delta(y)
    -(\mathrm{ad}(x)\otimes \mathrm{id})\Delta(y)\\
    &&\hspace{1cm} +\Delta(P(x)\cdot y)\\
    &&=\sum_{i}(a_{i}\otimes [x,y]\cdot b_{i}-[x,y]\cdot a_{i}\otimes b_{i}-y\cdot[x,a_{i}]\otimes b_{i}-y\cdot a_{i}\otimes [x,b_{i}]\\
    &&\hspace{1cm}+[x,a_{i}]\otimes y\cdot b_{i}+a_{i}\otimes y\cdot[x,b_{i}]-a_{i}\otimes [x,y\cdot b_{i}]+y\cdot a_{i}\otimes[x,b_{i}]\\
    &&\hspace{1cm} -[x,a_{i}]\otimes y\cdot b_{i}+[x,y\cdot a_{i}]\otimes b_{i}+a_{i}\otimes P(x)\cdot y\cdot b_{i}-P(x)\cdot y\cdot a_{i}\otimes b_{i})\\
    &&=\sum_{i}(a_{i}\otimes [x,y]\cdot b_{i}+a_{i}\otimes y\cdot[x,b_{i}]-a_{i}\otimes [x,y\cdot b_{i}]+a_{i}\otimes P(x)\cdot y\cdot b_{i})\\
   &&\hspace{1cm}+ \sum_{i}(-[x,y]\cdot a_{i}\otimes b_{i}-y\cdot[x,a_{i}]\otimes b_{i}+[x,y\cdot a_{i}]\otimes b_{i}-P(x)\cdot y\cdot a_{i}\otimes b_{i})\\
   &&\stackrel{(\ref{eq:GLR})}{=}0.
\end{eqnarray*}}
Thus Eq.~(\ref{eq:Bi5}) holds automatically.
\end{proof}

Summarizing the above study, we have the following conclusion.

\begin{thm}\label{thm:GPA-bi}
Let $(A,\cdot,[-,-],P)$ be a relative Poisson algebra which is
dually represented by $Q$. Let $r=\sum\limits_{i}a_{i}\otimes b_{i}\in
A\otimes A$  and $\Delta,\delta:A\rightarrow A\otimes A$ be linear
maps defined by (\ref{AssoCob}) and (\ref{eq:LieCob})
respectively. Then $(A,\cdot,[-,-],\Delta,\delta,P,Q)$ is a
coboundary relative Poisson bialgebra if and only if
Eqs.~(\ref{eq:AYBE1}), (\ref{eq:AYBE2}),
(\ref{eq:CYBE1}), (\ref{eq:CYBE2}) and
(\ref{eq:tr1})-(\ref{eq:tr7}) hold.
\end{thm}


Let $A$ be a vector space and $r=\sum\limits_{i}a_{i}\otimes b_{i}\in
A\otimes A$. $r$ can be identified with a linear map from $A^{*}$
to $A$ as follows:
\begin{equation}\label{eq:identify}
r(a^{*})=\sum_{i}\langle a^{*},a_{i}\rangle b_{i},\;\; \forall a^{*}\in A^{*}.
\end{equation}

There is an analogue of the Drinfeld classical double (\cite{CP1})
for a relative Poisson bialgebra.

\begin{thm}\label{thm:sub-bialgebras}
Let  $(A,\cdot_{A},[-,-]_{A},\Delta_{A},\delta_{A},P,Q)$ be a
relative Poisson bialgebra.
Let $(A^{*},\cdot_{A^{*}}$,
$[-,-]_{A^{*}}$,
$\Delta_{A^{*}},\delta_{A^{*}},Q^{*},P^{*})$ be the
relative Poisson bialgebra given in Remark~\ref{rmk:dual}.
Then there is a coboundary relative Poisson bialgebra structure
on the direct sum $A\oplus A^{*}$ of vector spaces which contains
these two relative Poisson bialgebra structures  on $A$ and
$A^{*}$ respectively as relative Poisson sub-bialgebras.
\end{thm}
\begin{proof}
Let $r\in A\otimes A^{*}\subset (A\oplus A^{*})\otimes (A\oplus
A^{*})$ correspond to the identity map $\mathrm{id}:A\rightarrow
A$. Let $\lbrace e_{1},\cdots, e_{n}\rbrace$ be a basis of $A$ and $\lbrace e^{*}_{1},\cdots, e^{*}_{n}\rbrace$ be the dual
basis. Then by Eq.~(\ref{eq:identify}), $r=\sum\limits_{i}e_{i}\otimes
e^{*}_{i}$. Since
$(A,\cdot_{A},[-,-]_{A},\Delta_{A},\delta_{A},P,Q)$ is a
relative Poisson bialgebra, there is a relative Poisson
algebra $(A\bowtie A^{*},P+Q^{*})$ given by the matched pair
$((A,P),(A^{*},Q^{*}),-\mathcal{L}^{*}_{A},-\mathcal{L}^{*}_{A^{*}},\mathrm{ad}^{*}_{A},$
$\mathrm{ad}^{*}_{A^{*}})$, that is,
$$x\cdot_{A\bowtie A^{*}}y=x\cdot_{A}y,\;\; x\cdot_{A\bowtie A^{*}}a^{*}=-\mathcal{L}^{*}_{A}(x)a^{*}-\mathcal{L}^{*}_{A^{*}}(a^{*})x,\;\;a^{*}\cdot_{A\bowtie A^{*}} b^{*}=a^{*}\cdot_{A^{*}}b^{*},$$
$$[x,y]_{A\bowtie A^{*}}=[x,y]_{A},\;\;[x,a^{*}]_{A\bowtie A^{*}}=\mathrm{ad}^{*}_{A}(x)a^{*}-\mathrm{ad}^{*}_{A^{*}}(a^{*})x,\;\; [a^{*},b^{*}]_{A\bowtie A^{*}}=[a^{*},b^{*}]_{A^{*}},$$
for all $x,y\in A, a^{*},b^{*}\in A^{*}$. By Lemma
\ref{lem:standard Manin triple of relative Poisson algebras},
$Q+P^{*}$ dually represents $(A\bowtie A^{*},P+Q^{*})$. Define two
linear maps $\Delta_{A\bowtie A^{*}},\delta_{A\bowtie
A^{*}}:A\bowtie A^{*}\rightarrow (A\bowtie A^{*})\otimes(A\bowtie
A^{*})$ respectively  by
$$\Delta_{A\bowtie A^{*}}(u)=(\mathrm{id}\otimes\mathcal{L}_{A\bowtie A^{*}}(u)-\mathcal{L}_{A\bowtie A^{*}}(u)\otimes\mathrm{id})r,$$
$$\delta_{A\bowtie A^{*}}(u)=(\mathrm{ad}_{A\bowtie A^{*}}(u)\otimes\mathrm{id}+\mathrm{id}\otimes\mathrm{ad}_{A\bowtie A^{*}}(u))r,$$
for all $u\in A\bowtie A^{*}$. Then by \cite[Theorem
2.3.6]{Bai2010}, $r$ satisfies the AYBE in the commutative
associative algebra $(A\bowtie A^{*},\cdot_{A\bowtie A^{*}})$, and
the following equation holds:
$$(\mathrm{id}\otimes\mathcal{L}_{A\bowtie A^{*}}(u)-\mathcal{L}_{A\bowtie A^{*}}(u)\otimes\mathrm{id})(r+\tau(r))=0,\;\;\forall u\in A\bowtie A^{*}.$$
By \cite[Propositions 1.4.2 and 2.1.11]{CP1}, $r$ satisfies the
CYBE in the Lie algebra $(A\bowtie A^{*},[-,-]_{A\bowtie A^{*}})$,
and the following equation holds:
$$(\mathrm{ad}_{A\bowtie A^{*}}(u)\otimes\mathrm{id}+\mathrm{id}\otimes\mathrm{ad}_{A\bowtie A^{*}}(u))(r+\tau(r))=0,\;\;\forall u\in A\bowtie A^{*}.$$
Moreover, we have
\begin{eqnarray*}
((P+Q^{*})\otimes\mathrm{id}-\mathrm{id}\otimes(Q+P^{*}))r&=&\sum_{i=1}^{n}((P+Q^{*})\otimes\mathrm{id}-\mathrm{id}\otimes(Q+P^{*}))(e_{i}\otimes e^{*}_{i})\\
&=&\sum_{i=1}^{n}(P(e_{i})\otimes e^{*}_{i}-e_{i}\otimes
P^{*}(e^{*}_{i}))=0,
\end{eqnarray*}
\begin{eqnarray*}
((Q+P^{*})\otimes\mathrm{id}-\mathrm{id}\otimes(P+Q^{*}))r&=&\sum_{i=1}^{n}((Q+P^{*})\otimes\mathrm{id}-\mathrm{id}\otimes(P+Q^{*}))(e_{i}\otimes e^{*}_{i})\\
&=&\sum_{i=1}^{n}(Q(e_{i})\otimes e^{*}_{i}-e_{i}\otimes
Q^{*}(e^{*}_{i}))=0.
\end{eqnarray*}
Then by Theorem \ref{thm:GPA-bi}, $(A\bowtie A^{*},\cdot_{A\bowtie
A^{*}},[-,-]_{A\bowtie A^{*}},\Delta_{A\bowtie
A^{*}},\delta_{A\bowtie A^{*}},P+Q^{*},Q+P^{*})$ is a coboundary
relative Poisson bialgebra.
Moreover, again by \cite[Theorem
2.3.6]{Bai2010} and \cite[Propositions 1.4.2 and 2.1.11]{CP1}, we
have
$$\Delta_{A\bowtie A^{*}}(x)=\Delta_{A}(x), \;\delta_{A\bowtie A^{*}}(x)=\delta_{A}(x),\; \forall x\in A.$$
Thus $(A\bowtie A^{*},\cdot_{A\bowtie A^{*}},[-,-]_{A\bowtie
A^{*}},\Delta_{A\bowtie A^{*}},\delta_{A\bowtie
A^{*}},P+Q^{*},Q+P^{*})$ contains
$(A,\cdot_{A},[-,-]_{A},\Delta_{A},$   $\delta_{A},P,Q)$ as a
relative Poisson sub-bialgebra. Similarly,  $(A\bowtie
A^{*},\cdot_{A\bowtie A^{*}},[-,-]_{A\bowtie A^{*}},$
$\Delta_{A\bowtie A^{*}}$, $\delta_{A\bowtie A^{*}},$
$P+Q^{*},Q+P^{*})$ contains
$(A^{*},\cdot_{A^{*}},[-,-]_{A^{*}},\Delta_{A^{*}},\delta_{A^{*}},Q^{*},P^{*})$
as a relative Poisson sub-bialgebra.
\end{proof}

A direct consequence of Theorem \ref{thm:GPA-bi} is given as
follows.

\begin{cor}\label{cor:GPA-bi}
Let $(A,\cdot,[-,-],P)$ be a relative Poisson algebra which is
dually represented by $Q$. Let $r\in A\otimes A$ and $\Delta,\delta:A\rightarrow A\otimes A$
be linear maps defined by Eqs.~(\ref{AssoCob}) and (\ref{eq:LieCob})
respectively. If $r$ is antisymmetric  in the sense that $r+\tau(r)=0$  and satisfies the AYBE in
the commutative associative algebra $(A,\cdot)$, the CYBE in
the Lie algebra $(A,[-,-])$ and the following equations:
\begin{equation}\label{eq:PYBE1}
(P\otimes \mathrm{id}-\mathrm{id}\otimes Q)r=0,
\end{equation}
\begin{equation}\label{eq:PYBE2}
(Q\otimes \mathrm{id}-\mathrm{id}\otimes P)r=0,
\end{equation}
then $(A,\cdot,[-,-],\Delta,\delta,P,Q)$ is a coboundary relative Poisson bialgebra.
\end{cor}

It motivates us to give the following notion.

\begin{defi}
Let $(A,\cdot,[-,-],P)$ be a relative Poisson algebra,
$Q:A\rightarrow A$ be a linear map and $r\in A\otimes A$.  $r$ is
called a solution of the \textbf{relative  Poisson Yang-Baxter
equation (RPYBE) associated to $Q$ ($Q$-RPYBE)} in $(A,P)$ if $r$
satisfies the AYBE in the commutative associative algebra
$(A,\cdot)$, the CYBE in the Lie algebra $(A,[-,-])$ and
Eqs.~(\ref{eq:PYBE1})-(\ref{eq:PYBE2}).
\end{defi}

\begin{rmk}\label{rmk:equiv}
If $r\in A\otimes A$ is antisymmetric, then Eq.~(\ref{eq:PYBE1})
holds if and only if Eq.~(\ref{eq:PYBE2}) holds. On the other
hand, the notion of \textbf{Poisson Yang-Baxter equation
(PYBE)} in a Poisson algebra was given in \cite{NB} whose solutions
are exactly the solutions of both the AYBE and the CYBE. So from
the form, the $Q$-RPYBE (defined in relative Poisson algebras) is
exactly the PYBE (defined in Poisson algebras) satisfying the
additional Eqs.~(\ref{eq:PYBE1})-(\ref{eq:PYBE2}).
\end{rmk}


\begin{thm}\label{thm:Q-PYBE}
Let $(A,\cdot, [-,-],P)$ be a relative Poisson algebra and $r\in
A\otimes A$ be antisymmetric. Let $Q:A\rightarrow A$ be a linear
map. Then $r$ is a solution of the $Q$-RPYBE in $(A,P)$ if and
only if $r$ satisfies
\begin{equation}\label{eq:Q-PYBE1}
[r(a^{*}),r(b^{*})]=r(\mathrm{ad}^{*}(r(a^{*}))b^{*}-\mathrm{ad}^{*}(r(b^{*}))a^{*}),
\end{equation}
\begin{equation}\label{eq:Q-PYBE2}
r(a^{*})\cdot r(b^{*})=-r(\mathcal{L}^{*}(r(a^{*}))b^{*}+\mathcal{L}^{*}(r(b^{*}))a^{*}),
\end{equation}
\begin{equation}\label{eq:Q-PYBE3}
Pr=rQ^{*},
\end{equation}
for all $a^{*},b^{*}\in A^{*}$.
\end{thm}
\begin{proof}
Let $r=\sum\limits_{i}a_{i}\otimes b_{i}$. By \cite{Kup}, $r$ is
an antisymmetric solution of  the CYBE in the Lie algebra
$(A,[-,-])$ if and only if Eq.~(\ref{eq:Q-PYBE1}) holds. By
\cite{Bai2010}, $r$ is an antisymmetric solution of the AYBE in
the commutative associative algebra $(A,\cdot)$ if and only if
Eq.~(\ref{eq:Q-PYBE2}) holds. Moreover, for all $a^{*}\in A^{*}$,
by Eq.~(\ref{eq:identify}), we have
$$r(Q^{*}(a^{*}))=\sum_{i}\langle Q^{*}(a^{*}),a_{i}\rangle b_{i}=\sum_{i}\langle a^{*},Q(a_{i})\rangle b_{i},\;\;P(r(a^{*}))=\sum_{i}\langle a^{*},a_{i}\rangle P(b_{i}).$$
So Eq.~(\ref{eq:PYBE1}) holds if and only if Eq.~(\ref{eq:Q-PYBE3}) holds. This completes the proof.
\end{proof}

 \subsection{$\mathcal{O}$-operators of relative Poisson algebras}\

\begin{defi}\label{defi:O-operator}
Let $(A,P)$ be a relative Poisson algebra, $(\mu,\rho,V)$ be a
compatible structure on $(A,P)$ and $\alpha:V\rightarrow V$ be a
linear map. A linear map $T:V\rightarrow A$ is called a
\textbf{weak $\mathcal{O}$-operator of $(A,P)$ associated to $(\mu,\rho,V)$
and $\alpha$} if the following equations hold:
\begin{equation}\label{eq:O1}
T(u)\cdot T(v)=T(\mu(T(u))v+\mu(T(v))u),
\end{equation}
\begin{equation}\label{eq:O2}
[T(u), T(v)]=T(\rho(T(u))v-\rho(T(v))u),
\end{equation}
\begin{equation}\label{eq:O3}
PT=T\alpha,
\end{equation}
for all $u,v\in V$. If in addition, $(\mu,\rho,\alpha,V)$ is a representation of
$(A,P)$,  then $T$ is called an \textbf{$\mathcal{O}$-operator of $(A,P)$
associated to $(\mu,\rho,\alpha,V)$}.
\end{defi}

\begin{rmk}
In fact,  $T$ is called  {\bf an $\mathcal{O}$-operator of the commutative associative algebra $(A,\cdot)$ associated to $(\mu,V)$} (\cite{Bai2010}) if $T$ satisfies Eq.~(\ref{eq:O1}), and
  $T$ is called an {\bf $\mathcal{O}$-operator of the Lie algebra $(A,[-,-])$ associated to $(\rho,V)$} (\cite{Kup}) if $T$ satisfies Eq.~(\ref{eq:O2}).
\end{rmk}

Therefore  Theorem \ref{thm:Q-PYBE} is rewritten in terms of
$\mathcal{O}$-operators as follows.

\begin{cor}
Let $(A,P)$ be a relative Poisson algebra and $r\in A\otimes A$ be
antisymmetric. Let $Q:A\rightarrow A$ be a linear map. Then $r$ is
a solution of  the $Q$-RPYBE in $(A,P)$ if and only if $r$ is a
weak $\mathcal{O}$-operator of $(A,P)$ associated to
$(-\mathcal{L}^{*},\mathrm{ad}^{*},A^{*})$ and $Q^{*}$. If in
addition, $(A,P)$ is dually represented by $Q$, then $r$ is a
solution of the $Q$-RPYBE in $(A,P)$ if and only if $r$ is an
$\mathcal{O}$-operator of $(A,P)$ associated to the representation
$(-\mathcal{L}^{*},\mathrm{ad}^{*},Q^{*},A^{*})$.
\end{cor}

We consider the semi-direct product relative Poisson algebras which are dually represented.

\begin{thm}\label{pro:eqGPA}
Let $(A,P)$ be a  relative Poisson algebra and $(\mu,\rho,V)$ be a compatible structure on $(A,P)$. Let $Q:A\rightarrow A$ and $\alpha,\beta:V\rightarrow V$ be linear maps.
Then the following conditions are equivalent.
\begin{enumerate}
\item There is a relative Poisson algebra
$(A\ltimes_{\mu,\rho}V,P+\alpha)$ which is dually represented by
the linear operator $Q+\beta$;

\item There is a relative Poisson algebra
$(A\ltimes_{-\mu^{*},\rho^{*}}V^{*},P+\beta^{*})$ which is dually
represented by the linear operator $Q+\alpha^{*}$;

\item The following conditions are satisfied:
\begin{enumerate}
\item $(\mu,\rho,\alpha,V)$ is a representation of $(A,P)$; \item
$\beta$ dually represents $(A,P)$ on $(\mu,\rho,V)$; \item $Q$
dually represents $(A,P)$; \item for all $x\in A,v\in V$, we have
\begin{equation}\label{eq:cond1}
\mu(Q(x))v-\mu(x)\alpha(v)-\beta(\mu(x)v)=0,
\end{equation}
\begin{equation}\label{eq:cond2}
\rho(Q(x))v-\rho(x)\alpha(v)-\beta(\rho(x)v)=0.
\end{equation}
\end{enumerate}
\end{enumerate}
\end{thm}

\begin{proof}
$(1)\Longleftrightarrow (3)$. By Proposition \ref{pro:semipro}, $(A\ltimes_{\mu,\rho}V,P+\alpha)$ is a relative Poisson algebra if and only if
 $(\mu,\rho,\alpha,V)$ is a representation of the relative Poisson algebra $(A,P)$.
Let $x,y,z\in A,u,v,w\in V$. By Corollary \ref{cor:dualadj2}, $Q+\beta$ dually represents $(A\ltimes_{\mu,\rho}V,P+\alpha)$ if and only if the following equations are satisfied:
\begin{eqnarray*}
0&=&(x+u)\cdot(Q+\beta)(y+v)-(P+\alpha)(x+u)\cdot(y+v)-(Q+\beta)((x+u)\cdot (y+v))\\
&=& x\cdot Q(y)-P(x)\cdot y-Q(x\cdot y)+\mu(x)\beta(v)-\mu(P(x))v-\beta(\mu(x))v\\
&\mbox{}&+\mu(Q(y))u
-\mu(y)\alpha(u)
-\beta(\mu(y)u);\\
0&=&[x+u,(Q+\beta)(y+v)]-[(P+\alpha)(x+u),y+v]-(Q+\beta)([x+u,y+v])\\
&=& [x,Q(y)]-[P(x),y]-Q([x,y])+\rho(x)\beta(v)-\rho(P(x))v-\beta(\rho(x))v\\
&&+\rho(Q(y))u-\rho(y)\alpha(u)-\beta(\rho(y)u);\\
0&=&(P+\alpha+Q+\beta)((x+u)\cdot (y+v)\cdot (z+w))\\
&=&(P+Q)(x\cdot y\cdot z)+(\alpha+\beta)(\mu(x\cdot y)w+\mu(x\cdot z)v+\mu(y\cdot z)u).
\end{eqnarray*}
If the above equations hold, then
\begin{enumerate}
\item[(i)] Eqs.~(\ref{eq:dualrep1}), ~(\ref{eq:dualrep2}) and
~(\ref{eq:eqdualrep1}) (where $y$ is replaced by $z$) hold by
letting $y=u=0$; \item[(ii)] Eqs.~(\ref{eq:dualadj1}),
(\ref{eq:dualadj2}) and~(\ref{eq:eqdualadj1}) hold by letting
$u=v=w=0$; \item[(iii)] Eqs.~ (\ref{eq:cond1})
and~(\ref{eq:cond2}) (where $x$ is replaced by $y$, and $v$ by
$u$)  hold by letting $x=v=0$.
\end{enumerate}
Conversely, obviously, if Eqs.~(\ref{eq:dualrep1}), ~(\ref{eq:dualrep2}), ~(\ref{eq:eqdualrep1}), (\ref{eq:dualadj1}), (\ref{eq:dualadj2}),~(\ref{eq:eqdualadj1}), (\ref{eq:cond1}) and~(\ref{eq:cond2}) hold,
then the above equations hold. Hence Condition (1) holds if and only if Condition (3) holds.

$(2)\Longleftrightarrow (3)$. From the above equivalence between Condition (1) and Condition (3), we have Condition (2) holds if and only if the  items (a)-(c) in Condition (3) as well as
 the following two equations hold (for all $x\in A, u^{*}\in V^{*}$):
\begin{equation}\label{eq:cond3}
-\mu^{*}(Q(x))u^{*}+\mu^{*}(x)\beta^{*}(u^{*})+\alpha^{*}(\mu^{*}(x)u^{*})=0,
\end{equation}
\begin{equation}\label{eq:cond4}
\rho^{*}(Q(x))u^{*}-\rho^{*}(x)\beta^{*}(u^{*})-\alpha^{*}(\rho^{*}(x)u^{*})=0.
\end{equation}
For all $x\in A, u^*\in V^*, v\in V$, we have
{\small\begin{eqnarray*} \langle
-\mu^{*}(Q(x))u^{*}+\mu^{*}(x)\beta^{*}(u^{*})+\alpha^{*}(\mu^{*}(x)u^{*}),v\rangle=\langle
u^*, -\mu(Q(x))v+\mu(x)\alpha(v)+\beta(\mu(x)v)\rangle.
\end{eqnarray*}}
Hence Eq.~(\ref{eq:cond3}) holds if and only if Eq.~(\ref{eq:cond1}) holds. Similarly, Eq.~(\ref{eq:cond4}) holds if and only if Eq.~(\ref{eq:cond2}) holds.
Hence Condition (2) holds if and only if Condition (3) holds.
\end{proof}

Next we show that $\mathcal O$-operators give antisymmetric
solutions of the RPYBE in semi-direct product relative Poisson
algebras and hence give rise to relative Poisson bialgebras.

\begin{thm}\label{thm:GPBA}
Let $(A,P)$ be a relative Poisson algebra and $(\mu,\rho,V)$ be
a compatible structure on $(A,P)$. Let $\beta$ dually represent
$(A,P)$ on $(\mu,\rho,V)$ and thus let
$(-\mu^{*},\rho^{*},\beta^{*},V^{*})$ be the  representation of
$(A,P)$ defined in Proposition~\ref{pro:dual}.
 Let $Q:A\rightarrow
A$ and $\alpha:V\rightarrow V$ be linear maps. Let $T:V\rightarrow
A$ be a linear map which is identified as an element in
$(A\ltimes_{-\mu^{*},\rho^{*}}V^{*}) \otimes
(A\ltimes_{-\mu^{*},\rho^{*}}V^{*})$.
\begin{enumerate}
\item  $r=T-\tau(T)$ is an antisymmetric solution of the
$(Q+\alpha^{*})$-RPYBE in the relative Poisson algebra
$(A\ltimes_{-\mu^{*},\rho^{*}}V^{*},P+\beta^{*})$ if and only if
$T$ is a weak $\mathcal{O}$-operator of $(A,P)$ associated to
$(\mu,\rho,V)$ and $\alpha$, and satisfies $T\beta=QT$. \item Assume that
$(A,P)$ is dually represented by $Q$ and $(\mu,\rho,\alpha,$ $V)$
is a representation of $(A,P)$. If $T$ is an
$\mathcal{O}$-operator of $(A,P)$ associated to $(\mu,\rho,\alpha,V)$
satisfying $T\beta=QT$, then $r=T-\tau(T)$ is an antisymmetric
solution of the $(Q+\alpha^{*})$-RPYBE in the relative Poisson
algebra $(A\ltimes_{-\mu^{*},\rho^{*}}V^{*},P+\beta^{*})$. If in
addition, Eqs.~(\ref{eq:cond1}) and (\ref{eq:cond2}) hold, then
$Q+\alpha^{*}$ dually represents the relative Poisson algebra
$(A\ltimes_{-\mu^{*},\rho^{*}}V^{*},P+\beta^{*})$. In this case,
there is a relative Poisson bialgebra
$(A\ltimes_{-\mu^{*},\rho^{*}}V^{*},\cdot,[-,-],\Delta,\delta,P+\beta^{*},Q+\alpha^{*})$,
where the linear maps $\Delta$ and $\delta$ are defined
respectively by Eqs.~(\ref{AssoCob}) and (\ref{eq:LieCob}) with
$r=T-\tau(T)$.
\end{enumerate}
 \end{thm}

\begin{proof}
(1). By \cite{Bai2007}, $r=T-\tau(T)$ satisfies the CYBE in the
Lie algebra $A\ltimes_{\rho^*}V^*$ if and only if
Eq.~(\ref{eq:O1}) holds and by \cite{Bai2010}, $r$ satisfies the
AYBE in the commutative associative algebra $A\ltimes_{-\mu^*}V^*$
if and only if Eq.~(\ref{eq:O2}) holds.

Let {$\{v_1,\cdots,v_m\}$} be a basis of $V$ and
$\{v^{*}_{1},\cdots, v^{*}_{m}\}$ be the dual basis. Then
$T=\sum\limits^{m}_{i=1}T(v_{i})\otimes v^{*}_{i}\in
(A\ltimes_{-\mu^{*},\rho^{*}}V^{*}) \otimes
(A\ltimes_{-\mu^{*},\rho^{*}}V^{*})$. Hence
$$r=T-\tau(T)=\sum\limits_{i=1}^{m}T(v_{i})\otimes v^{*}_{i}-v^{*}_{i}\otimes T(v_{i}).$$
Note that
$$((P+\beta^{*})\otimes \mathrm{id})r=\sum_{i=1}^{m}(P( T(v_{i}))\otimes v^{*}_{i}-\beta^{*}(v^{*}_{i})\otimes T(v_{i})),$$
$$(\mathrm{id}\otimes(Q+\alpha^{*}))r=\sum_{i=1}^{m}( T(v_{i})\otimes \alpha^{*}(v^{*}_{i})-v^{*}_{i}\otimes Q(T(v_{i}))).$$
Further
\begin{eqnarray*}
\sum_{i=1}^{m}\beta^{*}(v^{*}_{i})\otimes
T(v_{i})&=&\sum_{i=1}^{m}\sum_{j=1}^{m}\langle
\beta^{*}(v^{*}_{i}),v_{j}\rangle v_{j}^{*}\otimes T(v_{i})
    =\sum_{i=1}^{m}\sum_{j=1}^{m}v^{*}_{j}\otimes\langle v^{*}_{i},\beta(v_{j})\rangle T(v_{i})\\
    &=&\sum_{i=1}^{m}\sum_{j=1}^{m}v^{*}_{i}\otimes T(\langle \beta(v_{i}),v^{*}_{j}\rangle v_{j})=\sum_{i=1}^{m}v^{*}_{i}\otimes T(\beta(v_{i})),
\end{eqnarray*}
and similarly, $\sum\limits_{i=1}^{m}T(v_{i})\otimes\alpha^{*}(v^{*}_{i})=\sum\limits_{i=1}^{m}T(\alpha(v_{i}))\otimes v^{*}_{i}$.
Therefore $((P+\beta^{*})\otimes \mathrm{id})r=(\mathrm{id}\otimes(Q+\alpha^{*}))r$ if and only if $PT=T\alpha$ and $T\beta=QT$. Thus the conclusion holds.

(2) It follows from Item (1) and Theorem~\ref{pro:eqGPA}.
\end{proof}

Therefore starting from an $\mathcal O$-operator $T$ of a relative
Poisson algebra $(A,P)$ associated to a representation
$(\mu,\rho,\alpha,V)$, one gets an antisymmetric solution of the
$(Q+\alpha^{*})$-RPYBE in the relative Poisson algebra
$(A\ltimes_{-\mu^{*},\rho^{*}}V^{*},P+\beta^{*})$ for suitable
linear maps $\beta$ and $Q$ and hence gives rise to a relative
Poisson bialgebra on the latter. There are some natural choices
of $Q$ and $\beta$. For example, assume that $\beta=\pm \alpha,
Q=\pm P$ or $\beta=\theta \alpha^{-1}, Q=\theta P^{-1}$ for $0\ne
\theta\in K$ when $\alpha$ and $P$ are invertible. Note that in
these cases $T\beta=QT$ automatically and then one can get the
other corresponding constraint conditions due to
Theorem~\ref{thm:GPBA}. In particular, when $\beta=-\alpha$ and
$Q=-P$, there is not any constraint condition.


\begin{cor}\label{cor:cases}
Let $(\mu,\rho,\alpha,V)$ be a representation of a relative
Poisson algebra $(A,P)$ and  $T:V\rightarrow A$ be an
$\mathcal{O}$-operator of $(A,P)$ associated to
$(\mu,\rho,\alpha,V)$. Then $r=T-\tau(T)$ is an antisymmetric
solution of the $(-P+\alpha^{*})$-RPYBE in the relative Poisson
algebra $(A\ltimes_{-\mu^{*},\rho^{*}}V^{*},P-\alpha^{*})$.
Further the relative Poisson algebra
$(A\ltimes_{-\mu^{*},\rho^{*}}V^{*},P-\alpha^{*})$ is dually
represented by $-P+\alpha^{*}$ and there is a relative Poisson
bialgebra
$(A\ltimes_{-\mu^{*},\rho^{*}}V^{*},\cdot,[-,-],\Delta,\delta,P-\alpha^{*},-P+\alpha^{*})$,
where the linear maps $\Delta$ and $\delta$ are defined
respectively by Eqs.~(\ref{AssoCob}) and (\ref{eq:LieCob}) through
$r=T-\tau(T)$.
\end{cor}

\begin{proof}
By Corollary \ref{cor:automap},
$(-\mu^{*},\rho^{*},-\alpha^{*},V^{*})$ is a representation of
$(A,P)$ and $-P$ dually represents $(A,P)$. Moreover
Eqs.~(\ref{eq:cond1}) and (\ref{eq:cond2}) hold when
$\beta=-\alpha, Q=-P$. Hence the conclusion follows from
Theorem~\ref{thm:GPBA} when $\beta=-\alpha, Q=-P$.
\end{proof}

\subsection{Relative pre-Poisson algebras}\

Recall the notions of
Zinbiel algebras (\cite{Lod}) and pre-Lie algebras (\cite{Bur}).

\begin{defi}
A \textbf{Zinbiel algebra} is a vector space $A$ equipped with a
bilinear operation  $\star:A\otimes A\rightarrow A$ such that
\begin{equation}\label{eq:Zinbiel}
x\star(y\star z)=(y\star x)\star z+(x\star y)\star z, \;\;\forall x,y,z\in A.
\end{equation}
\end{defi}

Let $(A,\star)$ be a Zinbiel algebra. Define a new bilinear
operation $\cdot$ on $A$ by
\begin{equation}\label{eq:ZintoAss}
x\cdot y=x\star y+y\star x,\;\;\forall x,y\in A.
\end{equation}
Then $(A,\cdot)$ is a commutative associative algebra. Moreover,
$(\mathcal{L}_{\star},A)$ is a representation of the commutative
associative algebra $(A,\cdot)$, where
$\mathcal{L}_{\star}(x)y=x\star y$ for all $ x,y\in A.$

\begin{defi}
A \textbf{pre-Lie algebra} is a vector space $A$ equipped with a
bilinear operation $\circ:A\otimes A\rightarrow A$ such that
\begin{equation}\label{eq:preLie}
(x\circ y)\circ z-x\circ(y\circ z)=(y\circ x)\circ z-y\circ(x\circ z), \;\;\forall x,y,z\in A.
\end{equation}
\end{defi}

Let $(A,\circ)$ be a pre-Lie algebra. Define a new bilinear
operation $[-,-]$ on $A$ by
\begin{equation}\label{eq:preLietoLie}
[x,y]=x\circ y-y\circ x,\;\;\forall x,y\in A.
\end{equation}
Then $(A,[-,-])$ is a Lie algebra. Moreover,
$(\mathcal{L}_{\circ},A)$ is a representation of the Lie algebra
$(A,[-,-])$, where $\mathcal{L}_{\circ}(x)y=x\circ y$ for all $
x,y\in A.$

\begin{defi}\label{de:GPPA}
A \textbf{relative pre-Poisson algebra} is a quadruple
$(A,\star,\circ,P)$, where $(A,\star)$ is a Zinbiel algebra,
$(A,\circ)$ is a pre-Lie algebra,  $P:A\rightarrow A$ is a
derivation of both $(A,\star)$ and $(A,\circ)$, that is, for all
$x,y\in A,$
\begin{equation}\label{eq:GPPA1}
P(x\star y)=P(x)\star y+x\star P(y),
\end{equation}
\begin{equation}\label{eq:GPPA2}
P(x\circ y)=P(x)\circ y+x\circ P(y),
\end{equation}
and the following compatible conditions are satisfied:
\begin{equation}\label{eq:GPPA3}
(x\star y+y\star x)\circ z-x\star(y\circ z)-y\star(x\circ z)+(x\star y+y\star x)\star P(z)=0,
\end{equation}
\begin{equation}\label{eq:GPPA4}
y\circ(x\star z)-x\star(y\circ z)+(x\circ y-y\circ x)\star z-(x\star P(y)+P(y)\star x)\star z=0,
\end{equation}
for all $x,y,z\in A$.
\end{defi}

\begin{rmk}
Recall a pre-Poisson algebra (\cite{Agu2000.0}) is a triple $(A,\star,\circ)$, where $(A,\star)$ is a Zinbiel algebra and $(A,\circ)$ is a pre-Lie algebra such that the following conditions hold:
$$(x\circ y-y\circ x)\star z=x\star(y\circ z)-y\circ(x\star z),$$
$$(x\star y+y\star x)\circ z=x\star(y\circ z)+y\star(x\circ z),$$
for all $x,y,z\in A$.
Thus any pre-Poisson algebra is a relative pre-Poisson algebra
with the derivation $P=0$.
\end{rmk}


There is the following construction of relative pre-Poisson
algebras from Zinbiel algebras with their derivations, which is an
analogue of Example~\ref{ex:GPA}.

\begin{pro}\label{pro:generalized pre-Poisson1}
Let $(A,\star)$ be a Zinbiel algebra and $P$ be a derivation of
$(A,\star)$. Define a new bilinear operation $\circ:A\otimes
A\rightarrow A$ by
\begin{equation}
x\circ y=x\star P(y)-P(x)\star y,\;\;\forall x,y\in A.
\end{equation}
Then $(A,\circ)$ is a pre-Lie algebra and $(A,\star,\circ,P)$ is a relative pre-Poisson algebra.
\end{pro}

\begin{proof}
Let $x,y,z\in A$. Then we have
\begin{small}
\begin{eqnarray*}
    (x\circ y)\circ z-x\circ(y\circ z)&=&
    (x\star P(y))\star P(z)-(P(x)\star y)\star P(z)-(x\star P^{2}(y))\star z+(P^{2}(x)\star y)\star z\\
    &&-x\star(y\star P^{2}(z))+x\star(P^{2}(y)\star z)
    +P(x)\star(y\star P(z))-P(x)\star(P(y)\star z),\\
    (y\circ x)\circ z-y\circ(x\circ z)
    &=&(y\star P(x))\star P(z)-(P(y)\star x)\star P(z)-(y\star P^{2}(x))\star z+(P^{2}(y)\star x)\star z\\
    &&-y\star(x\star P^{2}(z))+y\star(P^{2}(x)\star z)+P(y)\star(x\star P(z))-P(y)\star(P(x)\star z).
\end{eqnarray*}
\end{small}
Note that by Eq.~(\ref{eq:Zinbiel}), we have
\begin{eqnarray*}
&&-x\star(y\star P^{2}(z))=-y\star(x\star P^{2}(z)),\\
&&-P(x)\star (P(y)\star z)=-P(y)\star(P(x)\star z),\\
&&-(P(x)\star y)\star P(z)+P(x)\star(y\star P(z))=(y\star P(x))\star P(z),\\
&&(x\star P(y))\star P(z)=-(P(y)\star x)\star P(z)+P(y)\star(x\star P(z)), \\
&&-(x\star P^{2}(y))\star z+x\star(P^{2}(y)\star z)=(P^{2}(y)\star x)\star z, \\
&&(P^{2}(x)\star y)\star z=-(y\star P^{2}(x))\star z+y\star(P^{2}(x)\star z).
\end{eqnarray*}
Thus Eq.~(\ref{eq:preLie}) holds, that is, $(A,\circ)$ is a pre-Lie algebra. Similarly, we show that
$P$ is a derivation of $(A,\circ)$ and Eqs.~(\ref{eq:GPPA3}) and (\ref{eq:GPPA4}) hold. Hence $(A,\star,\circ,P)$ is a relative pre-Poisson algebra.
\end{proof}

\begin{ex}\label{ex:generalized pre-Poisson1}
Let $(A,\star)$ be a 3-dimensional Zinbiel algebra with a basis
$\{e_1,e_2,e_3\}$ whose non-zero products (\cite{Kay2}) are given
as follows.
$$e_{1}\star e_{1}=e_{3},\; e_{1}\star e_{2}=e_{3}.$$
Define a linear map $P:A\rightarrow A$ as
$$P(e_{1})=e_{1}+e_{2},\;P(e_{2})=2e_{2},\;P(e_{3})=3e_{3}.$$
Then $P$ is a derivation of $(A,\star)$.
Hence 
there is a relative pre-Poisson algebra $(A,\star,\circ,P)$
with the following non-zero products of the pre-Lie algebra
$(A,\circ)$:
$$e_{1}\circ e_{1}=e_{3},\;e_{1}\circ e_{2}=e_{3}.$$
\end{ex}

\begin{pro}\label{pro:O-operator}
Let $(A,\star,\circ,P)$ be a relative pre-Poisson algebra. Then $(A,\cdot,[-,-],$ $P)$ is a relative Poisson algebra,
where $\cdot,[-,-]: A\otimes A\rightarrow A$ are defined by Eqs.~(\ref{eq:ZintoAss}) and (\ref{eq:preLietoLie}) respectively.
 Moreover, $(\mathcal{L}_{\star},\mathcal{L}_{\circ},P,A)$ is a representation of the relative Poisson algebra $(A,\cdot,[-,-],P)$ and hence
the identity map ${\rm id}$ is an $\mathcal O$-operator of $(A,\cdot,[-,-],P)$ associated to $(\mathcal{L}_{\star},\mathcal{L}_{\circ},P,A)$.
\end{pro}
\begin{proof}
Let $x,y,z\in A$. By Eqs.~(\ref{eq:GPPA1}) and (\ref{eq:GPPA2}),
$P$ is a derivation of both $(A,\cdot)$ and $(A,[-,-])$. Moreover,
by Eqs.~(\ref{eq:Zinbiel}), (\ref{eq:GPPA3}) and (\ref{eq:GPPA4}),
we have {\small\begin{eqnarray*}
&&[z,x\cdot y]-x\cdot[z,y]-y\cdot[z,x]-x\cdot y\cdot P(z)\\
&&=(z\circ(x\star y)+z\circ(y\star x)-(x\star y)\circ z-(y\star x)\circ z)-(x\star(z\circ y)-x\star(y\circ z)\\
&&\hspace{0.5cm}+(z\circ y)\star x-(y\circ z)\star x)-(y\star(z\circ x)-y\star(x\circ z)+(z\circ x)\star y-(x\circ z)\star y)\\
&&\hspace{0.5cm}-((x\star y+y\star x)\star P(z)+P(z)\star(x\star y+y\star x))\\
&&=(x\star P(z)+P(z)\star x)\star y+(y\star P(z)+P(z)\star y)\star x-P(z)\star(x\star y+y\star x)=0.
\end{eqnarray*}}
Thus $(A,\cdot,[-,-],P)$ is a relative Poisson algebra. Moreover, by Eqs.~(\ref{eq:GPPA1})-(\ref{eq:GPPA4}), we have
\begin{eqnarray*}
&&P(\mathcal{L}_{\star}(x)y)=\mathcal{L}_{\star}(P(x))y+\mathcal{L}_{\star}(x)P(y),\\
&&P(\mathcal{L}_{\circ}(x)y)=\mathcal{L}_{\circ}(P(x))y+\mathcal{L}_{\circ}(x)P(y),\\
&&\mathcal{L}_{\circ}(x\cdot y)z-\mathcal{L}_{\star}(x)\mathcal{L}_{\circ}(y)z-\mathcal{L}_{\star}(y)\mathcal{L}_{\circ}(x)z+\mathcal{L}_{\star}(x\cdot y)P(z)=0,\\
&&\mathcal{L}_{\circ}(y)\mathcal{L}_{\star}(x)z-\mathcal{L}_{\star}(x)\mathcal{L}_{\circ}(y)z+\mathcal{L}_{\star}([x,y])z-\mathcal{L}_{\star}(x\cdot P(y))z=0.
\end{eqnarray*}
Thus $(\mathcal{L}_{\star},\mathcal{L}_{\circ},P,A)$ is a representation of the relative Poisson algebra $(A,\cdot,[-,-],P)$.
 The rest of the conclusion is obvious.
\end{proof}

\begin{defi} Let $(A,\star,\circ,P)$ be a relative pre-Poisson algebra.
Define two bilinear operations
$\cdot,[-,-]: A\otimes A\rightarrow A$ by Eqs.~(\ref{eq:ZintoAss}) and (\ref{eq:preLietoLie}) respectively.
Then $(A,\cdot,[-,-],P)$ is called the \textbf{sub-adjacent relative Poisson algebra} of $(A,\star,\circ,P)$, and $(A,\star,\circ,P)$ is called a
\textbf{compatible relative pre-Poisson algebra} structure on the relative Poisson algebra $(A,\cdot,[-,-],P)$.
\end{defi}

By Corollary \ref{cor:cases} and Proposition \ref{pro:O-operator},
we obtain the following construction of antisymmetric solutions of
the RPYBE and hence relative Poisson bialgebras from relative
pre-Poisson algebras.

\begin{pro}\label{pro:final}
Let $(A,\star,\circ,P)$ be a relative pre-Poisson algebra and
$(A,\cdot,[-,-],P)$  be the sub-adjacent relative Poisson
algebra. Let $\{e_1,\cdots, e_n\}$ be a basis of $A$ and
$\{e^{*}_{1},\cdots, e^{*}_{n}\}$ be the dual basis. Then
\begin{equation}\label{eq:pro:final}
r=\sum^{n}_{i=1}(e_{i}\otimes e^{*}_{i}-e^{*}_{i}\otimes e_{i})
\end{equation}
is  an antisymmetric  solution of the $(-P+P^{*})$-RPYBE in the
relative Poisson algebra $(A$
$\ltimes_{-\mathcal{L}^{*}_{\star},\mathcal{L}^{*}_{\circ}}$
$A^{*}$, $P-P^{*})$. Further the relative Poisson algebra $(A$
$\ltimes_{-\mathcal{L}^{*}_{\star},\mathcal{L}^{*}_{\circ}}A^{*},P-P^{*})$
is dually represented by $(-P+P^*)$ and hence there is a relative
Poisson bialgebra
$(A\ltimes_{-\mathcal{L}^{*}_{\star},\mathcal{L}^{*}_{\circ}}A^{*},\cdot,[-,-],\Delta,\delta,P-P^{*},-P+P^{*})$,
where the linear maps $\Delta$ and $\delta$ are defined
respectively by Eqs.~(\ref{AssoCob}) and (\ref{eq:LieCob}) through
$r$.
\end{pro}

\section{Relative Poisson bialgebras and Frobenius Jacobi algebras}

We use relative Poisson bialgebras to construct Frobenius
Jacobi algebras.  In particular, there is a construction of Frobenius Jacobi algebras from relative
pre-Poisson algebras. We give an example to illustrate
this construction explicitly.

By Remark \ref{rmk:matchedpairJacobi}, the approach for the
bialgebra theory of relative Poisson algebras in terms of
matched pairs is not available for Jacobi algebras any more.  That
is, for a relative Poisson bialgebra
$(A,\cdot_A,[-,-]_{A},\Delta,\delta,P,$ $Q)$, it is impossible for
the commutative associative algebra structure on $A\oplus A^*$ to
be unital such that both $(A,\cdot_A)$ and $(A^*,\cdot_{A^*})$ are
unital, where $\cdot_{A^*}$ is given by the dual of $\Delta$.
However, it might be possible for the commutative associative
algebra structure on $A\oplus A^*$ to be unital such that one of
$(A,\cdot_A)$ and $(A^*,\cdot_{A^*})$ is unital. Note that in this
case, the induced relative Poisson algebra $A\bowtie A^*$ is a
Jacobi algebra and hence with the bilinear form $\mathcal B_{d}$
defined by Eq.~(\ref{eq:BilinearForm}), it is a Frobenius {Jacobi}
algebra. Explicitly,

\begin{pro}\label{pro:unit}
Let $(A,\cdot_A)$ be a unital commutative associative algebra with
the unit $1_A$ and $\Delta:A\rightarrow A\otimes A$ be a linear
map such that $(A,\cdot_A,\Delta)$ is a commutative and
cocommutative infinitesimal bialgebra. If the induced commutative
associative algebra structure on $A\oplus A^{*}$ is unital, then
the unit is $1_A$. Moreover, $1_A$ is the unit of the commutative
associative algebra structure on $A\oplus A^{*}$ if and only if
$\Delta(1_{A})=0$.
In particular, if $(A,\cdot,\Delta)$ is coboundary, then
$\Delta(1_{A})=0$ automatically.
\end{pro}

\begin{proof} By Remark \ref{rmk:matchedpairJacobi}, if the induced commutative associative
algebra structure on $A\oplus A^{*}$ is unital, then the unit is
$1_A$. Note that the commutative associative algebra structure on
$A\oplus A^{*}$ is given by {\small\begin{equation}\label{ASSOMP}
(x+a^{*})\cdot (y+b^{*})=x\cdot_{A}
y-\mathcal{L}^{*}_{A^*}(a^{*})y-\mathcal{L}^{*}_{A^*}(b^{*})x+a^{*}\cdot_{A^*}
b^{*}-\mathcal{L}^{*}_{A}(x)b^{*}-\mathcal{L}^{*}_{A}(y)a^{*},\;\;\forall
x,y\in A,a^{*},b^{*}\in A^{*},
\end{equation}}
where $\cdot_{A^*}$ is given by the dual of $\Delta$. Then $1_A$
is the unit of the commutative associative algebra structure on
$A\oplus A^{*}$ if and only if
\begin{eqnarray*}
&&1_{A}\cdot a^{*}=-\mathcal{L}^{*}_{A}(1_{A})a^{*}-\mathcal{L}^{*}_{A^*}(a^{*})1_{A}=a^{*},\;\;\forall a^{*}\in A^{*},\\
&&\Longleftrightarrow-\mathcal{L}^{*}_{A^*}(a^{*})1_{A}=0,\;\;\forall a^{*}\in A^{*},\\
&&\Longleftrightarrow\langle \Delta(1_{A}),a^{*}\otimes b^{*}\rangle=0,\;\;\forall a^{*},b^{*}\in A^{*},\\
&&\Longleftrightarrow\Delta(1_{A})=0.
\end{eqnarray*}
In particular, if $(A,\cdot,\Delta)$ is coboundary, that is, there
exists $r=\sum\limits_ia_i\otimes b_i\in A\otimes A$ such that
Eq.~(\ref{AssoCob}) holds, then we have
$$\Delta(1_{A})=({\rm id} \otimes\mathcal{L}(1_{A})-\mathcal{L}(1_{A})\otimes {\rm id})r=\sum_i(a_i\otimes b_i-a_i\otimes b_i)=0.$$
Hence the conclusion holds.
\end{proof}


Let $(A,\cdot,[-,-])$ be a Jacobi algebra and  $(A,\cdot,[-,-],
{\rm ad}(1_A))$ be the corresponding unital relative Poisson
algebra. By Corollary \ref{cor:dualadj2}, $Q$ dually represents
$(A,\cdot,[-,-],
{\rm ad}(1_A))$ if and only if $Q=-{\rm ad}(1_A)$. Then we have the following conclusion at the level of Jacobi algebras.

\begin{cor}\label{pro:unit2}
Let $(A,\cdot,[-,-])$ be a Jacobi algebra and  $(A,\cdot,[-,-],
{\rm ad}(1_A))$ be the corresponding unital relative Poisson
algebra. Suppose that
$(A,\cdot,[-,-],\Delta,\delta,\mathrm{ad}(1_{A}),-{\rm ad}(1_A))$ is a
relative Poisson bialgebra. Then the induced relative
Poisson algebra $(A\bowtie A^{*}, \mathrm{ad}(1_{A})-({\rm ad}(1_A))^{*})$ is
unital, that is, it is a Jacobi algebra, if and only if
$\Delta(1_{A})=0$. In this case, with the bilinear form $\mathcal
B_{d}$ defined by Eq.~(\ref{eq:BilinearForm}), it is a Frobenius
Jacobi algebra. In particular, if $(A,\cdot,[-,-],\Delta,\delta$,
$\mathrm{ad}(1_{A})$, $-{\rm ad}(1_A))$ is coboundary, then $\Delta(1_{A})=0$
automatically.
\end{cor}

Combining Corollaries~\ref{cor:GPA-bi} and~\ref{pro:unit2}
together, we have the following construction of Frobenius Jacobi
algebras from antisymmetric solutions of the RPYBE in Jacobi algebras.

\begin{cor}\label{cor:con}
Let $(A,\cdot,[-,-])$ be a Jacobi algebra and  $(A,\cdot,[-,-],
{\rm ad}(1_A))$ be the corresponding unital relative Poisson
algebra.  Let
$r\in A\otimes A$ and
$\Delta,\delta:A\rightarrow A\otimes A$ be linear maps defined by
Eqs.~(\ref{AssoCob}) and (\ref{eq:LieCob}) respectively. If $r$ is an
antisymmetric solution of the $Q$-RPYBE in the relative Poisson
algebra $(A,\cdot,[-,-], {\rm ad}(1_A))$, then
$(A,\cdot,[-,-],\Delta,\delta$, $\mathrm{ad}(1_{A}),-{\rm ad}(1_A))$ is a
coboundary relative Poisson bialgebra and hence  the induced
relative Poisson algebra $(A\bowtie A^{*},
\mathrm{ad}(1_{A})-({\rm ad}(1_A))^{*})$  with the bilinear form $\mathcal
B_{d}$ defined by Eq.~(\ref{eq:BilinearForm}) is a Frobenius
{Jacobi} algebra.
\end{cor}

\begin{rmk} Unfortunately, the construction of antisymmetric solutions of the RPYBE in relative
Poisson algebras from relative pre-Poisson algebras given in
Proposition \ref{pro:final} cannot be applied directly to the
above conclusion since the sub-adjacent relative Poisson
algebra $(A,\cdot,[-,-],P)$ of a relative pre-Poisson algebra
$(A,\star,\circ,P)$ is not a Jacobi algebra (hence the semi-direct
product relative Poisson algebra $(A$
$\ltimes_{-\mathcal{L}^{*}_{\star},\mathcal{L}^{*}_{\circ}}A^{*},P-P^{*})$
is not a Jacobi algebra, either). In fact, let $(A,\star)$ be a
Zinbiel algebra. Suppose that the commutative associative algebra
$(A,\cdot)$ defined by Eq.~(\ref{eq:ZintoAss}) has the unit
$1_{A}$. Then we have
$$x=x\star 1_{A}+1_{A}\star x,\;\;\forall x\in A.$$
Thus $1_{A}=2(1_{A}\star 1_{A})$. On the other hand, we have
$$x\star(1_{A}\star 1_A)=(x\cdot 1_{A})\star 1_A=x\star 1_A,\;\;\forall x\in A.$$
Hence $x\star 1_{A}=0$ and thus $1_{A}\star x=x$ for all $x\in A$.
Taking $x=1_A$, we have $x=0$, which is a contradiction.
\end{rmk}

Next we give an approach in which the construction given by
Corollary~\ref{cor:con} can be applied.

\begin{lem}\label{pro:standard Jacobi algebra}
Let $(A,\cdot,[-,-],P)$ be a relative Poisson algebra. Extend
the vector space $A$ to be a $(\dim A+1)$-dimensional vector space
$\tilde A=A\oplus\mathbb{K}e$. Define two bilinear operations
$\cdot_{\tilde A},[-,-]_{\tilde A}: \tilde A\otimes \tilde
A\rightarrow \tilde A$ and a linear map $\tilde P:\tilde
A\rightarrow \tilde A$ as
\begin{eqnarray}&&x\cdot_{\tilde A} y=x\cdot y,\;\; e\cdot_{\tilde A} x=x\cdot_{\tilde A} e=x,\;\;e\cdot_{\tilde A} e=e,\\
&&[x,y]_{\tilde A}=[x,y],\;\;[e,x]_{\tilde A}=-[x,e]_{\tilde A}=P(x),\;\;[e,e]_{\tilde A}=0,\\
&&\tilde P(x)=P(x), \tilde P (e)=0,
\end{eqnarray}
for all $x,y\in A$.
Then $(\tilde A, \cdot_{\tilde
        A},[-,-]_{\tilde A},\tilde{P})$ is a unital relative Poisson algebra. Consequently, there is a corresponding Jacobi algebra $(\tilde A, \cdot_{\tilde
        A},[-,-]_{\tilde A})$, which is called the
     \textbf{extended Jacobi algebra}  of the relative Poisson
algebra $(A,\cdot,[-,-],P)$.
\end{lem}

\begin{proof}
It is obvious that $(\tilde A, \cdot_{\tilde A})$ is a commutative associative algebra and $e$ is the unit.
Let $x,y,z\in A$. Then we have
$$[e,[x,y]_{\tilde A}]_{\tilde A}+[x,[y,e]_{\tilde A}]_{\tilde A}+[y,[e,x]_{\tilde A}]_{\tilde A}
=P([x,y])-[x,P(y)]+[y,P(x)]
=0.
$$
Hence $(\tilde A, [-,-]_{\tilde A})$ is a Lie algebra and
$\tilde P={\rm ad}_{\tilde A}(e)$. Moreover, we have
\begin{eqnarray*}
&&[e,x]_{\tilde A}\cdot_{\tilde A} y+x\cdot_{\tilde A}[e,y]_{\tilde A}+x\cdot_{\tilde A} y\cdot_{\tilde A} \tilde P(e)
=P(x)\cdot y+x\cdot P(y)
=P(x\cdot y)
=[e,x\cdot_{\tilde A} y]_{\tilde A},\\
&&[z,x]_{\tilde A}\cdot_{\tilde A} e+x\cdot_{\tilde A} [z,e]_{\tilde A}+x\cdot_{\tilde A} e\cdot_{\tilde A} \tilde P(z)
=[z,x]-x\cdot P(z)+x\cdot P(z)
=[z,x\cdot_{\tilde A} e]_{\tilde A}.
\end{eqnarray*}
Therefore $(\tilde A, \cdot_{\tilde
    A},[-,-]_{\tilde A},\tilde{P})$ is a unital relative Poisson algebra.
\end{proof}

\begin{pro}\label{standardJacobirep}
Let $(\mu,\rho,\alpha,V)$ be a representation of a relative
Poisson algebra $(A,\cdot,[-,-]$, $P)$. Then $(\tilde \mu, \tilde
\rho,V)$ is a representation of the extended Jacobi algebra
$(\tilde A, \cdot_{\tilde A},$ $[-,-]_{\tilde A})$ with the linear
maps $\tilde \mu, \tilde \rho:\tilde A\rightarrow {\rm End}(V)$
defined as
\begin{equation}\label{eq:tilde}\tilde \mu (x)=\mu(x), \;\;\tilde
\mu(e)={\rm id}_V,\;\;\tilde \rho(x)=\rho(x),\;\;\tilde
\rho(e)=\alpha, \;\;\forall x\in A.\end{equation}
 Moreover,
$(\tilde \mu,\tilde \rho,\alpha,V)$ is a representation of the
unital relative Poisson algebra $(\tilde A,\cdot_{\tilde A},$
$[-,-]_{\tilde A}$, $\tilde P)$.
\end{pro}

\begin{proof} The first part of the conclusion can be proved by
checking Eqs.~(\ref{eq:de:repJacobi2})
and~(\ref{eq:de:repJacobi3}) directly or as follows.
 Since $(\mu,\rho,\alpha,V)$ is a representation of
$(A,\cdot,[-,-],P)$, $(A\ltimes_{\mu,\rho} V,P+\alpha)$ is a
relative Poisson algebra. Hence there is an extended Jacobi
algebra $\widetilde {A\ltimes_{\mu,\rho}V}$ in which the bilinear
operations $\cdot_{\widetilde {A\ltimes_{\mu,\rho}V}}$ and
$[-,-]_{\widetilde {A\ltimes_{\mu,\rho}V}}$ on $A\oplus V$ are the
same as the ones of $(A\ltimes_{\mu,\rho} V,P+\alpha)$ and for all
$x\in A, u\in V$,
\begin{eqnarray*}
&&e\cdot_{\widetilde {A\ltimes_{\mu,\rho}V}} x=x,\;\;
e\cdot_{\widetilde {A\ltimes_{\mu,\rho}V}} u=u,\;\;
\\&&[e,x]_{\widetilde {A\ltimes_{\mu,\rho}V}}=(P+\alpha)x=P(x),\;\;[e,u]_{\widetilde
{A\ltimes_{\mu,\rho}V}}=(P+\alpha)u=\alpha(u).
\end{eqnarray*}
On the other hand, the above Jacobi algebra is exactly the Jacobi
algebra structure on the direct sum $\tilde A\oplus V=A\oplus
V\oplus \mathbb K e$ given by the Jacobi algebra $(\tilde A, \cdot_{\tilde
    A},[-,-]_{\tilde A})$ and
the linear maps $\tilde \mu, \tilde \rho:\tilde A\rightarrow {\rm
End}(V)$ defined by Eq.~(\ref{eq:tilde}) through
Eqs.~(\ref{eq:SDASSO}) and (\ref{eq:SDLie}). Thus $(\tilde \mu,
\tilde \rho,V)$ is a representation of $(\tilde A, \cdot_{\tilde
A},[-,-]_{\tilde A})$. The second part of the conclusion follows
from Proposition \ref{pro:repJacobi and GPA}.
\end{proof}



\begin{cor}\label{pro:OoperatorJacobi}
Let $T:V\rightarrow A$ be an $\mathcal{O}$-operator of a
relative Poisson algebra $(A,\cdot,[-,-],$ $P)$ associated to a
representation $(\mu,\rho,\alpha,V)$. Then
 $T:V\rightarrow
A\subset \tilde A$ is also an $\mathcal{O}$-operator of the unital
relative Poisson algebra $(\tilde A,\cdot_{\tilde
A},[-,-]_{\tilde A},\tilde P)$ (that is, the extended Jacobi
algebra $(\tilde A, \cdot_{\tilde A},[-,-]_{\tilde A})$)
associated to the representation $(\tilde \mu,\tilde
\rho,\alpha,V)$.
\end{cor}

\begin{proof}
It follows from Lemma~\ref{pro:standard Jacobi algebra} and
Proposition \ref{standardJacobirep}.
\end{proof}

Thus there is the following construction of Frobenius Jacobi
algebras from relative pre-Poisson algebras.

\begin{thm}\label{cor:final}
Let $(A,\star,\circ,P)$ be a relative pre-Poisson algebra and
$(A,\cdot,[-,-],P)$  be the sub-adjacent relative Poisson
algebra. Let $\{e_1,\cdots, e_n\}$ be a basis of $A$ and
$\{e^{*}_{1},\cdots, e^{*}_{n}\}$ be the dual basis.  Then
\begin{equation}\label{eq:cor:final}
r=\sum^{n}_{i=1}(e_{i}\otimes e^{*}_{i}-e^{*}_{i}\otimes e_{i})
\end{equation}
is an antisymmetric solution of the $(-\tilde P+P^{*})$-RPYBE in
the unital relative Poisson algebra $(\tilde A\ltimes
A^{*}:=\tilde A\ltimes_{-{\widetilde
{\mathcal{L}_{\star}}}^*,\widetilde
{\mathcal{L}_{\circ}}^{*}}A^{*},\tilde P-P^{*})$. Moreover, there
is a relative Poisson bialgebra $$(\tilde A\ltimes
A^{*},\cdot_{\tilde A\ltimes A^{*}},[-,-]_{\tilde A\ltimes
A^{*}},\Delta,\delta,\tilde P-P^{*},-\tilde P+P^{*}),$$ where the
linear maps $\Delta,\delta:\tilde A\ltimes A^{*}\rightarrow
(\tilde A\ltimes A^{*})\otimes (\tilde A\ltimes A^{*})$ are
defined respectively by Eqs.~(\ref{AssoCob}) and (\ref{eq:LieCob})
through $r$ such that the induced relative Poisson algebra
$(\tilde A\ltimes A^{*})\bowtie(\tilde A\ltimes A^{*})^*$  with
the bilinear form $\mathcal B_{d}$ defined by
Eq.~(\ref{eq:BilinearForm}) is a Frobenius Jacobi algebra.
\end{thm}

\begin{proof}
By Proposition \ref{pro:O-operator}, the identity map ${\rm id}_A$
is an $\mathcal O$-operator of the relative Poisson algebra
$(A,\cdot,[-,-],P)$ associated to
$(\mathcal{L}_{\star},\mathcal{L}_{\circ},P,A)$.
 By Corollary~\ref{pro:OoperatorJacobi}, ${\rm id}_A$ is also an
 $\mathcal O$-operator of the unital
relative Poisson algebra $(\tilde A,\cdot_{\tilde A},[-,-]_{\tilde
A},\tilde P)$ associated to the representation $(\widetilde
{\mathcal{L}_{\star}},\widetilde {\mathcal{L}_{\circ}}, P,A)$.
Then by Corollary~\ref{cor:cases}, $r$ is an antisymmetric
solution of the $(-\tilde P+P^{*})$-RPYBE in the relative Poisson
algebra $(\tilde A\ltimes A^{*},\tilde P-P^{*})$. Further the
relative Poisson algebra $(\tilde A\ltimes A^{*},\tilde P-P^{*})$
is dually represented by $-\tilde P+P^{*}$. The rest follows from
Corollary~\ref{cor:con}.
\end{proof}

\begin{ex}\label{ex:Jacobi algebra3}
Let $(A,\star,\circ,P)$ be the 3-dimensional relative
pre-Poisson algebra given by Example \ref{ex:generalized
pre-Poisson1}. Then the sub-adjacent relative Poisson algebra
$(A,\cdot,[-,-],P)$ is given by the following non-zero products:
$$e_{1}\cdot e_{1}=2e_{3},\;\;e_{1}\cdot e_{2}=e_{3},\;\;[e_{1},e_{2}]=e_{3}.$$
By Lemma~\ref{pro:standard Jacobi algebra}, we have the
extended Jacobi algebra $(\tilde A,\cdot_{\tilde A},[-,-]_{\tilde
A})$ whose non-zero products are given by the above products and
the following products:
\begin{eqnarray*}
&&e\cdot_{\tilde A}e_1=e_1,\;\; e\cdot_{\tilde A}e_2=e_2,\;\; e\cdot_{\tilde A}e_3=e_3,\;\; e\cdot_{\tilde A}e=e_,\\
&&[e,e_{1}]_{\tilde A}=P(e_{1})=e_{1}+e_{2},\;\;[e,e_{2}]_{\tilde A}=P(e_{2})=2e_{2},\;\;[e,e_{3}]_{\tilde A}=P(e_{3})=3e_{3}.
\end{eqnarray*}
Let $\{e_1^*,e_2^*,e_3^*\}$ be the basis of $A^*$ which is dual to
$\{e_1,e_2,e_3\}$. By Proposition \ref{standardJacobirep}, there
is a Jacobi algebra $(\tilde A\ltimes A^{*}:=\tilde
A\ltimes_{-{\widetilde {\mathcal{L}_{\star}}}^*,\widetilde
{\mathcal{L}_{\circ}}^{*}}A^{*}, \cdot_{\tilde A\ltimes A^{*}},
[-,-]_{\tilde A\ltimes A^{*}}) $, where the non-zero products are
given by the above non-zero products of $(\tilde A,\cdot_{\tilde
A},[-,-]_{\tilde A})$ and the following products:
\begin{eqnarray*}
&& e\cdot_{\tilde A\ltimes A^{*}} e^{*}_{1}=-{\widetilde {\mathcal{L}_{\star}}}^*(e)e^{*}_{1}=e^{*}_{1},\;\;
e\cdot_{\tilde A\ltimes A^{*}} e^{*}_{2}=-{\widetilde {\mathcal{L}_{\star}}}^*(e)e^{*}_{2}=e^{*}_{2},\\
&&e\cdot_{\tilde A\ltimes A^{*}} e^{*}_{3}=-{\widetilde {\mathcal{L}_{\star}}}^*(e)e^{*}_{3}=e^{*}_{3},\;\;
e_1\cdot_{\tilde A\ltimes A^{*}} e^{*}_{3}
=-{{\mathcal{L}_{\star}}}^*(e_1)e^{*}_{3}
=e^{*}_{1}+e_2^*,\\
&&[e,e^{*}_{1}]_{\tilde A\ltimes A^{*}}
={\widetilde {\mathcal{L}_{\circ}}}^*(e)e^{*}_{1}=-P^{*}(e^{*}_{1})=-e^{*}_{1},\;\;\\
&&[e,e^{*}_{2}]_{\tilde A\ltimes A^{*}}
={\widetilde {\mathcal{L}_{\circ}}}^*(e)e^{*}_{2}=-P^{*}(e^{*}_{2})=-e^{*}_{1}-2e^{*}_{2},\\
&&[e,e^{*}_{3}]_{\tilde A\ltimes A^{*}}
={\widetilde {\mathcal{L}_{\circ}}}^*(e)e^{*}_{3}=-P^{*}(e^{*}_{3})=-3e^{*}_{3},\\
&&[e_1,e^{*}_{3}]_{\tilde A\ltimes A^{*}}
=\mathcal{L}^{*}_{\circ}(e_{1})e^{*}_{3}=-e^{*}_{1}-e^{*}_{2}.
\end{eqnarray*}
In order to simplify the notations, we replace $\tilde A\ltimes
A^{*}$ by $J$, $\cdot_{\tilde A\ltimes A^{*}}$ by $\cdot$,
$[-,-]_{\tilde A\ltimes A^{*}}$ by $[-,-]$, and
$\{e,e_1,e_2,e_3,e_1^*,e_2^*,e_3^*\}$ by
$\{E,E_{1},E_{2},E_{3},E_{4},E_{5},E_{6}\}$. Then
$(J,\cdot,[-,-])$ is a Jacobi algebra with a basis
$\{E,E_{1},E_{2},E_{3},E_{4},E_{5},E_{6}\}$, in which $E$ is the
unit  and the other non-zero products are  given by
\begin{eqnarray*}
&&E_{1}\cdot E_{1}=2E_{3},\;\;E_{1}\cdot E_{2}=E_{3},\;\;E_{1}\cdot E_{6}=E_{4}+E_{5},\\
&&[E,E_{1}]=E_{1}+E_{2},\;\;[E,E_{2}]=2E_{2},\;\;[E,E_{3}]=3E_{3},\;\;[E,E_{4}]=-E_{4},\;\;\\
&&[E,E_{5}]=-E_{4}-2E_{5},\;\;[E,E_{6}]=-3E_{6},\;\;[E_{1},E_{2}]=E_{3},\;\;[E_{1},E_{6}]=-E_{4}-E_{5}.
\end{eqnarray*}
Obviously it is isomorphic to the Jacobi algebra $(\tilde A\ltimes
A^{*}, \cdot_{\tilde A\ltimes A^{*}}, [-,-]_{\tilde A\ltimes
A^{*}})$. {By Eq.~(\ref{eq:cor:final}), we set}
$$r=E_{1}\otimes E_{4}-E_{4}\otimes E_{1}+E_{2}\otimes E_{5}-E_{5}\otimes E_{2}+E_{3}\otimes E_{6}-E_{6}\otimes E_{3}.$$
Let $\Delta,\delta:J\rightarrow J\otimes J$ be linear maps
given by Eqs.~(\ref{AssoCob}) and (\ref{eq:LieCob}) respectively. Then 
\begin{eqnarray*}
&&\Delta(E_{1})=\Delta(E_{2})=-E_{3}\otimes E_{4}-E_{4}\otimes E_{3},\\
&&\Delta(E_{6})=-2E_{4}\otimes E_{4}-E_{4}\otimes E_{5}-E_{5}\otimes E_{4},\\
&&\delta(E_{1})=\delta(E_{2})=-E_{3}\otimes E_{4}+E_{4}\otimes E_{3},\\
&&\delta(E_{6})=-E_{4}\otimes E_{5}+E_{5}\otimes E_{4},\\
&&\Delta(E)=\Delta(E_{3})=\Delta(E_{4})=\Delta(E_{5})=\delta(E)=\delta(E_{3})=\delta(E_{4})=\delta(E_{5})=0.
\end{eqnarray*}
Hence $(J,\cdot,
[-,-],\Delta,\delta,\mathrm{ad}_{J}(E),-\mathrm{ad}_{J}(E))$ is a
relative Poisson bialgebra. Moreover, the non-zero {products} on
the relative Poisson algebra $J^{*}$ {are} given by
\begin{eqnarray*}&&E_{3}^{*}\cdot E^{*}_{4}=-E^{*}_{1}-E^{*}_{2},\;\;E^{*}_{4}\cdot E^{*}_{4}=-2E^{*}_{6},\;\;E^{*}_{4}\cdot E^{*}_{5}=-E^{*}_{6},\\
&&[E^{*}_{3},E^{*}_{4}]=-E^{*}_{1}-E^{*}_{2},\;\;[E^{*}_{4}, E^{*}_{5}]=-E^{*}_{6}.
\end{eqnarray*}
With the above relative Poisson algebra structures on $J$ and
$J^*$ together, $J\bowtie J^*$ is a relative Poisson algebra with
the unit $E$, in which the other non-zero products are given by
Eqs.~(\ref{eq:Asso}) and (\ref{eq:Lie}) with respect to the
matched pair
$((J,\mathrm{ad}_{J}(E)),(J^{*},-\mathrm{ad}_{J}(E)^{*}),-\mathcal{L}^{*}_{J},-\mathcal{L}^{*}_{J^{*}},\mathrm{ad}^{*}_{J},\mathrm{ad}^{*}_{J^{*}})$
as follows. {\small\begin{eqnarray*}
&&E_{1}\cdot E^{*}_{1}=E^{*},\;\; E_{1}\cdot E^{*}_{3}=-E_{4}+2E^{*}_{1}+E^{*}_{2},\;\; E_{1}\cdot E^{*}_{4}=-E_{3}+E^{*}_{6},\;\; E_{1}\cdot E^{*}_{5}=E^{*}_{6},\\
&&E_{2}\cdot E^{*}_{2}=E^{*},\;\; E_{2}\cdot E^{*}_{3}=-E_{4}+E^{*}_{1},\;\; E_{2}\cdot E^{*}_{4}=-E_{3},\;\; E_{3}\cdot E^{*}_{3}=E^{*},\;\; E_{4}\cdot E^{*}_{4}=E^{*},\\
&&E_{5}\cdot E^{*}_{5}=E^{*},\;\; E_{6}\cdot E^{*}_{4}=-2E_{4}-E_{5}+E^{*}_{1},\;\; E_{6}\cdot E^{*}_{5}=-E_{4}+E^{*}_{1},\;\; E_{6}\cdot E^{*}_{6}=E^{*},\\
&&{[E,E^{*}_{1}]=-E^{*}_{1},\;\; [E,E^{*}_{2}]=-E^{*}_{1}-2E^{*}_{2},\;\; [E,E^{*}_{3}]=-3E^{*}_{3},\;\; [E,E^{*}_{4}]=E^{*}_{4}+E^{*}_{5},\;\; [E,E^{*}_{5}]=2E^{*}_{5},}\\
&& [E,E^{*}_{6}]=3E^{*}_{6},\;\; [E_{1},E^{*}_{1}]=E^{*},\;\; [E_{1},E^{*}_{2}]=E^{*},\;\; [E_{1},E^{*}_{3}]=-E_{4}-E^{*}_{2},\;\;[E_{1},E^{*}_{4}]=E_{3}+E^{*}_{6},\\
&&[E_{1},E^{*}_{5}]=E^{*}_{6},\;\; [E_{2},E^{*}_{2}]=2E^{*},\;\; [E_{2},E^{*}_{3}]=-E_{4}+E^{*}_{1},\;\; [E_{2},E_{4}^{*}]=E_{3},\;\; [E_{3},E^{*}_{3}]=3E^{*},\\
&&[E_{4},E^{*}_{4}]=-E^{*},\;\; [E_{5},E^{*}_{4}]=-E^{*},\;\; [E_{5},E^{*}_{5}]=-2E^{*},\;\; [E_{6},E^{*}_{4}]=-E_{5}-E^{*}_{1},\;\;[E_{6},E^{*}_{5}]=E_{4}-E^{*}_{1},\\
&&[E_{6},E^{*}_{6}]=-3E^{*}.
\end{eqnarray*}}
Therefore by Theorem \ref{cor:final}, $(J\bowtie
J^{*},\mathcal{B}_{d})$ is a 14-dimensional Frobenius Jacobi
algebra, where $\mathcal{B}_{d}$ is given by
Eq.~(\ref{eq:BilinearForm}).
\end{ex}


\bigskip

 \noindent
 {\bf Acknowledgements.}  This work is supported by
NSFC (11931009, 12271265), the Fundamental Research Funds for the
Central Universities and Nankai Zhide Foundation. The authors
thank the referee for the valuable
  suggestions.

\end{document}